\documentclass[a4paper,reqno]{amsart}

\usepackage[margin=3cm]{geometry}
\usepackage{amsmath,amssymb,amsthm,mathrsfs,mathtools}
\usepackage{graphicx}
\usepackage{booktabs}
\usepackage{array}
\usepackage{enumitem}
\usepackage{xcolor}
\usepackage{placeins}
\usepackage{float}
\usepackage[colorlinks=true]{hyperref}

\hypersetup{
  colorlinks=true,
  linkcolor=violet,
  citecolor=cyan!60!black,
  urlcolor=cyan!60!black
}

\newtheorem{theorem}{Theorem}[section]
\newtheorem{lemma}[theorem]{Lemma}

\newtheorem{corollary}[theorem]{Corollary}
\newtheorem{proposition}[theorem]{Proposition}
\theoremstyle{remark}
\newtheorem{remark}[theorem]{Remark}
\theoremstyle{definition}

\numberwithin{equation}{section}
\allowdisplaybreaks

\newcommand{\T}{\mathbb T^2}
\newcommand{\R}{\mathbb R}
\newcommand{\Pa}{\mathbb P}
\newcommand{\dt}{h}
\newcommand{\B}[1]{B^{#1}}
\newcommand{\norm}[2]{\left\|#1\right\|_{#2}}
\newcommand{\Div}{\operatorname{div}}
\newcommand{\curl}{\operatorname{curl}}

\title[Good-vorticity integrator for Oldroyd--B]
{A Good-Vorticity Low-Regularity Integrator for the Inviscid Oldroyd--B System with Stress Diffusion}

\author[X. Cheng]{Xinyu Cheng}
 \address{X. Cheng, Research Institute of Intelligent Complex Systems, Fudan University, Shanghai, P.R. China}
\email{xycheng@fudan.edu.cn}

\author[Z. Luo]{Zhaonan Luo}
\address{Z. Luo, School of Science, Shenzhen Campus of Sun Yat-sen University, Shenzhen, P.R. China}
\email{luozhn7@mail.sysu.edu.cn}

\author[Z. Yang]{Zhaojie Yang}
\address{Z. Yang, School of Mathematics, Georgia Institute of Technology, Atlanta, United States}
\email{zyang768@gatech.edu}

\author[C. Yuan]{Cheng Yuan}
\address{C. Yuan, School of Mathematical Sciences, Fudan University, Shanghai, P.R. China}
\email{cyuan22@m.fudan.edu.cn}

\date{}


\begin{document}
\begin{abstract}
We establish sharp convergence results for a sequential semi-implicit
good-vorticity integrator for the two-dimensional inviscid Oldroyd--B
system with stress diffusion on the torus. For exact solutions satisfying
$u\in C_tB^{1+\alpha}_{\infty,1}$ and
$\tau\in C_tB^\alpha_{\infty,1}\cap L^1_tB^{2+\alpha}_{\infty,1}$,
$0\le\alpha\le2$, we prove convergence in the critical norms
$B^1_{\infty,1}\times B^0_{\infty,1}$, uniformly over the time nodes
on every finite interval where these regularity assumptions hold,
including the critical endpoint $\alpha=0$.
For $0<\alpha\le2$, we obtain the rate $O(h^{\alpha/2})$ without
any additional time H\"older assumption, reaching first order at
$\alpha=2$.
These algebraic exponents are optimal for this Eulerian scheme over
the stated solution classes, as shown by fixed-data lower bounds
in the case $a=b=0$ with nonzero velocity mean.
For a fixed shear solution at critical regularity, we establish matching
logarithmic upper and lower bounds
on the numerical error, ruling out every positive algebraic convergence
rate. We also prove that the gain of one stress derivative at the time
nodes is optimal when the velocity is controlled in $B^1_{\infty,1}$,
even for smooth initial stress.
\end{abstract}

\maketitle

\section{Introduction}

We consider the two-dimensional inviscid Oldroyd--B system with stress
diffusion
\begin{equation}
  \partial_t u+u\cdot\nabla u+\nabla p=\Div\tau,\qquad
  \Div u=0,
  \label{eq:u}
\end{equation}
\begin{equation}
  \partial_t\tau+u\cdot\nabla\tau+a\tau+Q(\nabla u,\tau)
  =
  D(u)+\Delta\tau ,
  \label{eq:tau}
\end{equation}
posed on $\T$. Here $u(t,x)\in\R^2$ is the velocity,
$\tau(t,x)\in\R^{2\times2}$ is the symmetric polymer stress,
$a\ge0$ is a damping parameter, and
\[
  D(u)=\frac{\nabla u+\nabla u^T}{2},\qquad
  \Omega(u)=\frac{\nabla u-\nabla u^T}{2}.
\]
The objective Oldroyd term is
\begin{equation}
  Q(\nabla u,\tau)
  =
  \tau\Omega(u)-\Omega(u)\tau
  -b\bigl(D(u)\tau+\tau D(u)\bigr),\qquad b\in[-1,1].
  \label{eq:Q}
\end{equation}
Throughout, the initial velocity is divergence free and the initial stress
satisfies \(\tau_0=\tau_0^T\). Both the equations and the numerical scheme
preserve the symmetry of the stress tensor. We do not require the stress
tensor to be positive semidefinite or to have zero trace.
The constitutive law goes back to Oldroyd \cite{O50,O58}.

Viscoelastic flow models also arise in biology, for example in ciliary
transport of airway mucus and sperm motion through cervical mucus.
Lauga \cite{L07} studied how fluid relaxation affects propulsion and
transport using an idealized waving sheet in an Oldroyd--B fluid.
Simulations of a finite-length swimmer by Teran, Fauci, and Shelley
\cite{TFS10} showed that viscoelastic stresses can increase swimming
speed and efficiency in certain regimes. These studies assume negligible
inertia and retain solvent viscosity. The system considered here retains
inertia and includes stress diffusion, with no solvent-viscosity term
in the momentum equation.

The continuous theory of stress-diffusive models examines how stress
dissipation controls the coupled evolution. Results include global
regularity \cite{ER15}, analysis of high-Reynolds and high-Weissenberg
regimes \cite{CWZZ20}, and decay under fractional diffusion
\cite{WWXZ22}.

At the discrete level, stability of the constitutive evolution has
motivated free-energy-dissipative schemes \cite{BLM09},
approximation of regularized models \cite{BB11}, and formulations based
on deformation tensors or characteristics \cite{PWW17,LNS16}.
Log-conformation methods address stress representation at high
Weissenberg numbers \cite{FK05,HGT26}, while trajectory-based
conformation updates preserve positive definiteness in a coupled
Stokes solver \cite{AP25}.

For viscous Oldroyd--B systems, error estimates are available for SUPG
finite elements \cite{EM03} and decoupled BDF2 \cite{R20}.
Temporal decoupling has also been developed for flows with zero solvent
viscosity \cite{DH10,DHM12}, while consistent splitting methods
separate the computation of the coupled fields \cite{PC23}.
Our question is whether a sequential stress--vorticity update converges
at critical Besov regularity, and which algebraic exponents are optimal
when additional spatial regularity is available.

The numerical difficulty is the interaction of a nondiffusive velocity
with a diffusive stress at their critical regularities. Good vorticity
cancels \(\curl\Div\tau\) in the vorticity equation, but its source
still contains \(Q(\nabla u,\tau)\). Both factors are only in
\(B^0_{\infty,1}\), where the product has no uniform algebra bound.
We therefore update stress first, using the old stress in its own source,
and then evaluate the constitutive term in the good-vorticity source with
\(\tau^{n+1}\). The stress solve supplies one additional derivative
of Besov regularity at each positive time node, so a mixed product
estimate controls
\(Q(\nabla u^n,\tau^{n+1})\) in \(B^0_{\infty,1}\).
Proposition~\ref{prop:critical-source-obstruction} makes the obstruction
explicit: uniformly bounded smooth critical inputs can have an
old-stress constitutive source of size \(\log K\), while the new-stress
source tends to zero at fixed step size as \(K\to\infty\).
This obstruction concerns a uniform source estimate; it does not
assert nonconvergence of the frozen-source method.

A further difficulty is that the final transport solve prevents the
stress iterates from retaining full parabolic smoothing. Discrete stress
smoothing and propagation of fixed-data frequency tails yield critical
convergence. With additional Besov regularity, coupled stability in
\(B^0_{\infty,1}\times B^{-1}_{\infty,1}\) gives algebraic error
rates. Fixed advected shears
establish optimality of the algebraic exponents over the stated parameter
and mean class, and exclude every positive algebraic rate for one fixed
critical datum.

Periodic pseudospectral methods have been used to study Oldroyd--B
dynamics \cite{TS07}. We use a dealiased Fourier implementation
\cite{CHQZ88} in Section~\ref{sec:numerics} to test temporal
convergence on smooth nonlinear data, check grid sensitivity, isolate
the constitutive source, and evaluate the logarithmic error of an
infinite-mode shear. These experiments examine the implementation and
the mechanisms in the analysis; the convergence theorems concern the
spatially untruncated scheme.

\textbf{Notation.} We identify \(\T=(\mathbb R/2\pi\mathbb Z)^2\).  For a scalar \(f\), a
vector \(v=(v_1,v_2)\), and a matrix field \(A=(A_{ij})\),
\begin{equation*}
\begin{split}
 &\nabla^\perp f=(-\partial_2f,\partial_1f),\qquad
 \curl v=\partial_1v_2-\partial_2v_1,\qquad \Div v=\partial_1v_1+\partial_2v_2,\\
 &(\nabla v)_{ij}=\partial_jv_i,\qquad (\Div A)_i=\sum_{j=1}^2\partial_jA_{ij}.
\end{split}
\end{equation*}
For a scalar- or matrix-valued field \(f\), we use
\(\Div(v\otimes f)\) as shorthand for the componentwise transport flux:
\[
 \Div(v\otimes f):=\sum_{j=1}^2\partial_j(v_jf),
\]
\(I\) denotes the identity operator, and \(\Pa\) is the periodic Leray
projector \(\Pa=I+\nabla(-\Delta)^{-1}\Div\), with \((-\Delta)^{-1}\)
set to zero on constants.  The notation
\([A,B]=AB-BA\) denotes the operator commutator.  With
\(|\T|=(2\pi)^2\), we normalize Fourier coefficients and spatial averages by
\[
 \widehat f(k)=(2\pi)^{-2}\int_{\T}f(x)e^{-ik\cdot x}\,dx,
 \qquad
 \langle f\rangle:=\widehat f(0)
 =(2\pi)^{-2}\int_{\T}f(x)\,dx .
\]
We write \(m_0=\langle u_0\rangle\in\mathbb R^2\) for the velocity mean.

For \(r\in\mathbb R\), we use the Sobolev norm
\begin{equation}
 \|f\|_{H^r(\mathbb{T}^{2})}^2=\sum_{k\in\mathbb Z^2}(1+|k|^2)^r|\widehat f(k)|^2.
 \label{eq:normalized-Sobolev-norm}
\end{equation}
Thus \(H^0=L^2\) uses the normalized measure \(|\T|^{-1}dx\).
The coefficient norm is Euclidean for vectors and Frobenius for tensors.
The discrete Fourier norms in the numerical experiments use the same normalization.

All Besov spaces are periodic and inhomogeneous; we write
\(B^s=B^s_{\infty,1}(\T)\).  We use the standard Littlewood--Paley and
Bony notation of \cite[Chapter~2]{BCD11}, with blocks indexed by
\(j\ge-1\) and blocks below \(-1\) set to zero.  Vector and tensor norms use any fixed
equivalent finite-dimensional norm.  The smooth cutoff used in the
numerical norm evaluation is specified in Section~\ref{sec:numerics}.

We write \(L_T^pX=L^p(0,T;X)\) and \(C_tX=C([0,T];X)\).
The notation \(A\lesssim B\) means
\(A\le CB\), with \(C\) independent of \(h\) and, unless explicitly stated,
of \(N\); \(C_T\) may also depend on the fixed solution bounds, \(a\), and
\(b\).  We use \(|k|\) for the Euclidean frequency norm,
\(|k|_\infty=\max_i|k_i|\), and \(\mathbf 1_E\) for the indicator of a set
\(E\).

\subsection{Low-regularity integrators and numerical Besov analysis}

Low-regularity integrators retain the dominant evolution within a
Duhamel step to reduce the regularity required in estimates of the
remainder. This principle has a general formulation \cite{RS21}
and has been developed for KdV \cite{WZ22,CLWY25}, rough wave solutions
\cite{CLLY25}, and Cahn--Hilliard \cite{C23}. For fluid flow,
Li--Ma--Schratz \cite{LMS22} combine exponential evolution and
semi-implicit transport to obtain viscosity-independent first-order
\(L^2\) estimates for Navier--Stokes under spatial Sobolev assumptions.

Besov spaces allow error estimates to track fractional spatial
regularity. Cheng--Luo--Wang \cite{CLWNS26} obtain viscosity-uniform
error estimates in Sobolev and Besov norms for a semi-implicit Fourier
method for Navier--Stokes and Euler flow under additional spatial
regularity. The SQG analysis
\cite{CLWSQG26} gives first-order \(B^0\) error bounds for
an exact active scalar bounded in \(B^2\), under the time-continuity
assumption of that analysis. The connection between Besov smoothness
and approximation rates also appears in Gaussian kernel methods on
manifolds: Lei, Shi, and Yuan \cite{LSY26} establish \(L^2\)
approximation rates governed by Besov regularity, with applications
to logistic classification.

For the present coupled system, the good-vorticity variable and the
operators used in velocity recovery are
\[
  \Gamma=\curl u-\mathcal R\tau,
  \qquad \mathcal R=\Delta^{-1}\curl\Div,
  \qquad \mathcal Kf=\nabla^\perp\Delta^{-1}(f-\langle f\rangle).
\]
The Navier--Stokes and SQG schemes in \cite{LMS22,CLWSQG26} use
the semigroup-smoothed old velocity as the advector. In the present
scheme, both transport solves use the recovered old velocity
\(u^n=m_0+\mathcal K(\Gamma^n+\mathcal R\tau^n)\), while the
heat factor acts on stress before the constitutive good-vorticity
source is evaluated. Section~\ref{sec:discrete-scheme} specifies
this sequential update.

To analyze this splitting, Lemma~\ref{ep:product-smoothing} estimates
the repeated heat--resolvent evolution for any divergence-free velocity
sequence uniformly bounded in \(B^s\), \(1\le s\le2\), without any
assumption on its time regularity. Applied to the stress recurrence,
it gives one additional derivative at positive time nodes, both at
critical regularity and under additional Besov regularity.
A heat-recurrence estimate also bounds the time sum of
\(\|E_h\tau^n\|_{B^2}\), weighted by \(h\).
When the velocity is controlled only in \(B^1\), the gain of one
derivative at the nodes is optimal, even for smooth initial stress;
Appendix~\ref{sec:smoothing-sharpness} shows why the final transport
solve prevents any greater gain in general.

\subsection{Main convergence results}
\label{sec:main-result}

We first prove critical convergence without a prescribed rate, then
give algebraic rates under additional Besov regularity. Once the stress
is bounded in \(B^\alpha\) and the velocity mean is fixed, a
\(B^\alpha\) bound on \(\Gamma\) is equivalent to a \(B^{1+\alpha}\)
bound on \(u\), by the velocity recovery formula.

\begin{theorem}[Convergence at critical regularity]
\label{thm:critical-convergence}
Let \(T>0\), \(a\ge0\), \(b\in[-1,1]\), and suppose that an exact
solution on \([0,T]\) satisfies
\begin{equation}
 u\in C([0,T];B^1),\qquad
 \tau\in C([0,T];B^0)\cap L^1(0,T;B^2).
 \label{ep:exact}
\end{equation}
Initialize the sequential scheme
\eqref{eq:recover-discrete}--\eqref{eq:scheme-gamma} with
\(\tau^0=\tau_0\), \(\Gamma^0=\curl u_0-\mathcal R\tau_0\), and the
fixed velocity mean \(m_0=\langle u_0\rangle\).
For all sufficiently small \(h=T/M\), its unique critical orbit exists
through \(T\), is bounded uniformly in \(B^1\times B^0\times B^0\), and
\begin{equation}
 \lim_{\substack{M\to\infty\\h=T/M}}
 \max_{0\le n\le M}
 \bigl(\|u(t_n)-u^n\|_{B^1}
       +\|\tau(t_n)-\tau^n\|_{B^0}
       +\|\Gamma(t_n)-\Gamma^n\|_{B^0}\bigr)=0.
 \label{ep:main}
\end{equation}
No algebraic rate is asserted. The step threshold may depend on the
fixed data and exact trajectory, including their frequency tails.
\end{theorem}

Additional Besov regularity makes the convergence quantitative. The main
estimate is obtained in the weaker velocity--stress norms
\(B^0\times B^{-1}\); interpolation then gives the critical-norm rate
and the intermediate weak estimate stated below.

\begin{theorem}[Algebraic convergence in the base solution class]
\label{thm:base-algebraic}
Let \(0<\alpha\le1\), \(T>0\), and let an exact solution on
\([0,T]\) satisfy
\begin{equation}
 \begin{aligned}
 u&\in C([0,T];B^{1+\alpha}),\\
 \tau&\in C([0,T];B^\alpha)
        \cap L^1(0,T;B^{2+\alpha}).
 \end{aligned}
 \label{eq:endpoint-exact-class}
\end{equation}
Set \(\Gamma=\curl u-\mathcal R\tau\) and \(m_0=\langle u_0\rangle\).
Initialize \eqref{eq:recover-discrete}--\eqref{eq:scheme-gamma}
with the exact data and keep the velocity mean equal to \(m_0\).
For an integer \(M\ge1\), set \(h=T/M\), \(t_n=nh\), and
\(e_f^n=f(t_n)-f^n\) for \(f\in\{u,\tau,\Gamma\}\).
For all sufficiently small \(h\), the numerical orbit exists through
time \(T\) and satisfies
\begin{align}
 \max_{0\le n\le M}
  \bigl(\|e_u^n\|_{B^0}+\|e_\tau^n\|_{B^{-1}}\bigr)
 &\le C_{T,\alpha}h^{(1+\alpha)/2},
 \label{eq:base-mixed-rate}\\
 \max_{0\le n\le M}
 \bigl(\|e_\tau^n\|_{B^{-\alpha}}+\|e_\Gamma^n\|_{B^{-\alpha}}\bigr)
 &\le C_{T,\alpha}h^\alpha,
 \label{eq:base-algebraic-weak}\\
 \max_{0\le n\le M}
 \bigl(\|e_\tau^n\|_{B^0}+\|e_\Gamma^n\|_{B^0}
             +\|e_u^n\|_{B^1}\bigr)
 &\le C_{T,\alpha}h^{\alpha/2}.
 \label{eq:base-algebraic-critical}
\end{align}
The constant and the upper bound on \(h\) depend only on \(\alpha,T,a,b\)
and the norms in \eqref{eq:endpoint-exact-class}.
Proposition~\ref{prop:exact-interval-continuation} establishes the required
uniform numerical bound.  No additional time H\"older assumption is needed.
\end{theorem}

The algebraic estimates hold for each fixed \(\alpha>0\).
Their constants and step restrictions need not be uniform as
\(\alpha\downarrow0\). Theorem~\ref{thm:critical-convergence} treats
\(\alpha=0\) separately, without taking a limit in these rates.

The same bounds also quantify convergence above the critical index.
Interpolating with the propagated \(B^\alpha\) bound gives the following
tradeoff between the output regularity and the convergence exponent.

\begin{corollary}[Convergence in intermediate Besov norms]
\label{cor:positive-output}
Under the hypotheses of Theorem~\ref{thm:base-algebraic}, for every
\(0\le\sigma<\alpha\),
\begin{equation}
 \max_{0\le n\le M}\left(
   \|e_\tau^n\|_{B^\sigma}
  +\|e_\Gamma^n\|_{B^\sigma}
  +\|e_u^n\|_{B^{1+\sigma}}\right)
 \le C_{T,\alpha,\sigma}h^{(\alpha-\sigma)/2}.
 \label{eq:endpoint-main}
\end{equation}
\end{corollary}

\begin{proof}
The case \(\sigma=0\) is Theorem~\ref{thm:base-algebraic}.
For \(0<\sigma<\alpha\), the exact-solution bounds and
Proposition~\ref{prop:exact-interval-continuation} bound both stress and good-vorticity
errors uniformly in \(B^\alpha\).  Dyadic interpolation gives
\[
 \|e_f^n\|_{B^\sigma}
 \lesssim \|e_f^n\|_{B^0}^{1-\sigma/\alpha}
           \|e_f^n\|_{B^\alpha}^{\sigma/\alpha},
 \qquad f=\tau,\Gamma.
\]
The velocity estimate follows from
\(e_u^n=\mathcal K(e_\Gamma^n+\mathcal Re_\tau^n)\).
\end{proof}

To improve the rate beyond the base class, we propagate the additional
regularity of the numerical stress and good vorticity. This stronger
bound allows a first-order estimate in a weaker norm to yield the
following critical-norm rate.

\begin{theorem}[Convergence rates under additional Besov regularity]
\label{thm:endpoint-rate-family}
Let \(0<\theta\le1\), \(T>0\), and suppose that an exact solution on
\([0,T]\) satisfies
\begin{equation}
 u\in C([0,T];B^{2+\theta}),\qquad
 \tau\in C([0,T];B^{1+\theta})
       \cap L^1(0,T;B^{3+\theta}).
 \label{eq:endpoint-rate-class}
\end{equation}
Initialize the scheme with the exact
data and conserved mean.  Then, for all sufficiently small \(h=T/M\),
\begin{equation}
 \max_{0\le n\le M}
 \left(
  \|\tau(t_n)-\tau^n\|_{B^0}
 +\|\Gamma(t_n)-\Gamma^n\|_{B^0}
 +\|u(t_n)-u^n\|_{B^1}
 \right)
 \le C_{T,\theta} h^{(1+\theta)/2}.
 \label{eq:endpoint-rate-family}
\end{equation}
The constant and the step threshold depend only on \(T,\theta,a,b\)
and the norms in \eqref{eq:endpoint-rate-class}, and not on \(h\).
In particular, \(\theta=1\) gives first-order convergence when the
solution has two additional spatial derivatives relative to the
critical class, as specified in \eqref{eq:endpoint-rate-class}.
For every \(\theta>0\), the exponent \((1+\theta)/2\) improves on
the exponent \(1/2\) obtained from Theorem~\ref{thm:base-algebraic}
with \(\alpha=1\). That theorem also covers the case \(\theta=0\).
\end{theorem}

\begin{remark}[Sharpness]
\label{rem:sharpness}
The convergence and smoothing results are sharp in the following senses.
\begin{enumerate}[label=(\roman*),leftmargin=*,itemsep=3pt]
\item \emph{Algebraic convergence exponents.}
Theorems~\ref{thm:base-algebraic} and~\ref{thm:endpoint-rate-family}
give the critical-norm rate \(h^{\alpha/2}\) for \(0<\alpha\le2\).
For every such \(\alpha\), Proposition~\ref{prop:positive-regularity-lower-bound}
gives fixed initial data whose velocity and good-vorticity errors are
bounded below by \(c h_m^{\alpha/2}/(1+|\log h_m|)^2\) along a
sequence \(h_m\to0\). Thus no larger algebraic exponent is possible,
even with a constant depending on the datum. Logarithmic improvements
are not excluded.

\item \emph{Critical convergence without an algebraic rate.}
Proposition~\ref{prop:endpt-lower-bound} and
Remark~\ref{rem:critical-shear-matching} give matching logarithmic
upper and lower bounds for one fixed critical shear. Hence even a
single fixed solution can have no positive algebraic convergence rate;
no universal logarithmic rate for critical data is asserted.

\item \emph{Nodal stress regularity.}
Under critical velocity control, the one-derivative nodal gain in
Proposition~\ref{ep:stress-smoothing} cannot be increased.
Appendix~\ref{sec:smoothing-sharpness} gives fixed \(u^0\in B^1\)
and smooth \(\tau^0\) for which
\(\tau^1\notin B^{1+\varepsilon}\) for every \(\varepsilon>0\).
\end{enumerate}
The convergence lower bounds in (i)--(ii) concern this Eulerian
time-discrete scheme and use \(a=b=0\) and nonzero velocity means.
They do not establish sharpness in every restricted parameter or mean class.
\end{remark}

The results concern time discretization without spatial truncation.
The Fourier-grid experiments in Section~\ref{sec:numerics} test
temporal convergence at fixed spatial resolutions; spatial
discretization errors are not covered by the analysis.

\subsection{Continuous well-posedness recalled}

The periodic local theory in \cite[Section~7]{CLYY}, stated for
\(a\in[0,1]\), constructs critical \(B^1\times B^0\) solutions with
an additional \(H^1\) assumption on the initial data and applies to
the \(D(u)\) normalization in \eqref{eq:tau}. A Galilean change of
variables handles nonzero velocity mean. Dyadic estimates and
\(L^1B^{2+\alpha}\) stress control propagate positive Besov
regularity on finite existence intervals.
Our theorems apply to any exact solution in the stated class, without
an additional \(H^1\) assumption. Uniqueness in the critical class is
proved in Lemma~\ref{ep:uniqueness}.

Section~\ref{sec:formulation-scheme} defines the scheme, and
Section~\ref{sec:besov-stability} collects the analytical estimates and
proves local discrete stability. Section~\ref{sec:critical-convergence}
develops discrete stress smoothing and frequency-tail estimates to prove
critical convergence. Sections~\ref{sec:quantitative-proofs}
and~\ref{sec:endpoint-rate-family} establish the algebraic rates;
Section~\ref{sec:endpoint-sharpness} proves the constitutive-source
obstruction and the error lower bounds.

\section{Good-Vorticity Formulation and Discrete Scheme}
\label{sec:formulation-scheme}

\subsection{Good-vorticity formulation}
\label{sec:good-vorticity}

Writing \(\omega=\curl u\), taking the curl of \eqref{eq:u} gives
\begin{equation}
  \partial_t\omega+u\cdot\nabla\omega=\curl\Div\tau .
  \label{eq:vorticity}
\end{equation}
For a periodic scalar \(f\), we use the convention
\begin{equation}
  \widehat{\Delta^{-1}f}(0)=0,\qquad
  \widehat{\Delta^{-1}f}(k)=-|k|^{-2}\widehat f(k)\quad(k\ne0).
  \label{eq:inverse-laplacian-convention}
\end{equation}
Thus \(\Delta^{-1}\) acts on the zero-mean part. Recall that the tensor operator \(\mathcal R=\curl\Div\Delta^{-1}\) is well defined
because \(\langle\curl\Div A\rangle=0\).
For a periodic velocity, \(\langle\curl u\rangle=0\), and hence both
\(\mathcal R\tau\) and \(\Gamma=\curl u-\mathcal R\tau\) have zero mean.
Integrating \eqref{eq:u} over the torus gives
\begin{equation}
  \langle u(t)\rangle=m_0,\qquad
  u(t)=m_0+\mathcal K\bigl(\Gamma(t)+\mathcal R\tau(t)\bigr).
  \label{eq:mean-mode-recovery}
\end{equation}
The numerical scheme likewise stores \(m_0\) separately and includes it
in the transport velocity.
Since $\Delta\mathcal R\tau=\curl\Div\tau$, applying $\mathcal R$ to
\eqref{eq:tau} gives
\begin{equation}
  \partial_t\mathcal R\tau-\Delta\mathcal R\tau+a\mathcal R\tau
  =
  \mathcal R D(u)-\mathcal RQ(\nabla u,\tau)-\mathcal R(u\cdot\nabla\tau).
  \label{eq:Rtau-raw}
\end{equation}
Writing
\[
  \mathcal R(u\cdot\nabla\tau)
  =
  u\cdot\nabla\mathcal R\tau
  +
  [\mathcal R,u\cdot\nabla]\tau ,
\]
and subtracting \eqref{eq:Rtau-raw} from \eqref{eq:vorticity}, we obtain
\begin{equation}
  \partial_t\Gamma+u\cdot\nabla\Gamma
  =
  a\mathcal R\tau
  -\mathcal R D(u)
  +\mathcal RQ(\nabla u,\tau)
  +[\mathcal R,u\cdot\nabla]\tau .
  \label{eq:gamma}
\end{equation}
The differentiated stress force $\curl\Div\tau$ has canceled.
The source in \eqref{eq:gamma} consists of order-zero operators applied
to \(\tau\), \(D(u)\), and \(Q(\nabla u,\tau)\), together with a
transport commutator.
We denote its right-hand side by \(G(u,\tau)\).

We first estimate this source at positive Besov indices, where its
constitutive product is controlled. These bounds will be used to
construct and continue the positive-regularity numerical orbit.

\begin{lemma}[Source and commutator estimates]
\label{lem:gamma-source}
\label{lem:critical-tame-high}
Let \(0<\alpha\le2\), let \(u\in B^{1+\alpha}\) be divergence free,
and let \(\tau\in B^\alpha\). Then
\begin{align}
 \|G(u,\tau)\|_{B^\alpha}
 &\le C_{a,\alpha}(1+\|u\|_{B^{1+\alpha}})
                (\|u\|_{B^{1+\alpha}}+\|\tau\|_{B^\alpha}),
 \label{eq:gamma-source}\\
 \|[\mathcal R,u\cdot\nabla]\tau\|_{B^\alpha}
 &\le C_\alpha\bigl(
     \|\nabla u\|_\infty\|\tau\|_{B^\alpha}
    +\|u-\langle u\rangle\|_{B^{1+\alpha}}\|\tau\|_{B^0}\bigr).
 \label{eq:critical-tame-R}
\end{align}
The second estimate will control high numerical norms through critical
norms in the continuation argument.
\end{lemma}

\begin{proof}
The order-zero periodic multiplier \(\mathcal R\), with zero symbol at
frequency zero, has uniformly bounded localized kernels and is bounded
on every \(B^r\). The linear terms in \(G\) are bounded by
\(C(a\|\tau\|_{B^\alpha}+\|u\|_{B^{1+\alpha}})\).
The positive-index product law gives
\begin{equation}
 \|\mathcal RQ(\nabla u,\tau)\|_{B^\alpha}
 \lesssim \|\nabla u\|_\infty\|\tau\|_{B^\alpha}
         +\|\tau\|_\infty\|\nabla u\|_{B^\alpha}.
 \label{eq:source-Q-detail}
\end{equation}

For the commutator, subtract the mean of \(u\), which commutes with
\(\mathcal R\). Bony's decomposition gives the localized bounds
\begin{align*}
 \|[\mathcal R,S_{k-1}u\cdot\nabla]\Delta_k\tau\|_\infty
 &\lesssim\|\nabla S_{k-1}u\|_\infty\|\Delta_k\tau\|_\infty,\\
 \|[\mathcal R,\Delta_ku\cdot\nabla]S_{k-1}\tau\|_\infty
 &\lesssim\|\Delta_ku\|_\infty
          \sum_{\ell\le k-2}2^\ell\|\Delta_\ell\tau\|_\infty,\\
 \|\Delta_j[\mathcal R,\Delta_ku\cdot\nabla]
                         \widetilde\Delta_k\tau\|_\infty
 &\lesssim2^j\|\Delta_ku\|_\infty
                    \|\widetilde\Delta_k\tau\|_\infty,
 \qquad j\le k+3.
\end{align*}
Here \(\widetilde\Delta_k\) is the sum of the three neighboring blocks.
The first bound is a kernel commutator estimate. In the second,
\(\mathcal R\) is localized to each low input block or annular output;
no \(L^\infty\) bound for \(\mathcal R S_{k-1}\) is used.
Incompressibility moves the derivative in the last bound to the output
block. The finitely many lowest periodic modes obey the same bounds.

Weighted summation of the low--high terms gives
\(C\|\nabla u\|_\infty\|\tau\|_{B^\alpha}\).
The high--low terms are bounded by
\[
 C\sum_k2^{(1+\alpha)k}\|\Delta_ku\|_\infty
           \sum_{\ell\le k-2}2^{\ell-k}\|\Delta_\ell\tau\|_\infty
 \le C\|u-\langle u\rangle\|_{B^{1+\alpha}}\|\tau\|_{B^0}.
\]
For the high--high terms, sum first in \(j\), using
\(\sum_{j\le k+3}2^{(1+\alpha)j}\lesssim2^{(1+\alpha)k}\).
This gives the same bound and proves \eqref{eq:critical-tame-R}.
In particular,
\begin{equation}
 \|[\mathcal R,u\cdot\nabla]\tau\|_{B^\alpha}
 \lesssim\|\nabla u\|_{B^\alpha}\|\tau\|_{B^\alpha}.
 \label{eq:source-comm-detail}
\end{equation}
Together with the linear and product bounds, this proves
\eqref{eq:gamma-source}.
\end{proof}

\subsection{Discrete scheme}
\label{sec:discrete-scheme}

Let $t_n=n\dt$.  For \(t\ge0\), define
\[
  E_t=e^{t(\Delta-a)}.
\]
The practical solver is Fourier pseudospectral with the Leray projection,
Biot--Savart recovery, and a standard $2/3$ dealiasing filter.  We evolve
$\tau$ and $\Gamma$ to retain the stress-force cancellation before
recovering the velocity.  The semi-implicit treatment of inviscid transport is
in the spirit of \cite{CLWNS26}.

Set \(m_0=\langle u_0\rangle\) and keep this zero Fourier mode fixed at every
time step.  Given $(\tau^n,\Gamma^n)$, recover
\begin{equation}
  u^n=m_0+\mathcal K\bigl(\Gamma^n+\mathcal R\tau^n\bigr).
  \label{eq:recover-discrete}
\end{equation}
The recovered velocity is divergence free and has mean \(m_0\).
The stress equation uses the old constitutive source in the step
\begin{equation}
  \tau^{n+1}
  +h\,u^n\cdot\nabla\tau^{n+1}
  =
  E_h\tau^n
  +hE_h\bigl(D(u^n)-Q(\nabla u^n,\tau^n)\bigr).
  \label{eq:scheme-tau-good}
\end{equation}
Diffusion and damping are integrated on the right-hand side.  The
stationary transport maximum principle applies to the operator
\(I+hu^n\cdot\nabla\) on the left.
Next define the good-vorticity source
\begin{equation}
  G^n
  =
  a\mathcal R\tau^n
  -\mathcal R D(u^n)
  +\mathcal RQ(\nabla u^n,\tau^{n+1})
  +[\mathcal R,u^n\cdot\nabla]\tau^n .
  \label{eq:Gn}
\end{equation}
Update $\Gamma$ by the semi-implicit transport step
\begin{equation}
  \Gamma^{n+1}
  +h\,u^n\cdot\nabla\Gamma^{n+1}
  =
  \Gamma^n+hG^n .
  \label{eq:scheme-gamma}
\end{equation}
Thus the two transport solves are performed in sequence. For later use,
write the stress source and the correction due to using \(\tau^{n+1}\) as
\begin{equation}
 S^n=D(u^n)-Q(\nabla u^n,\tau^n),\qquad
 J^n=\mathcal RQ(\nabla u^n,\tau^{n+1}-\tau^n),\qquad
 G^n=G(u^n,\tau^n)+J^n .
 \label{eq:source-delay}
\end{equation}
At critical regularity, \(S^0\in L^\infty\), and the heat factor
followed by the transport solve gives \(\tau^{n+1}\in B^1\).
Section~\ref{sec:critical-convergence} therefore estimates \(G^n\)
directly from \eqref{eq:Gn}. The decomposition with \(J^n\) is used
for positive-regularity error estimates; at the critical index its
summands are only distributions and need not separately belong to \(B^0\).
Each term in \(G^n\) has zero spatial mean: \(\mathcal R\) removes its zero
mode, and the remaining transport divergence integrates to zero.  Since
\(\Div u^n=0\), integration of \eqref{eq:scheme-gamma} therefore shows that
\(\langle\Gamma^{n+1}\rangle=\langle\Gamma^n\rangle\).  Exact
initialization gives \(\langle\Gamma^0\rangle=0\), so the zero-mean
condition needed by \(\mathcal K\) is preserved by the scheme.
Finally recover $u^{n+1}$ from \eqref{eq:recover-discrete}.

\section{Analytical Estimates and Discrete Stability}
\label{sec:besov-stability}

We collect the estimates used in the convergence proofs. Time regularity
controls increments of the exact solution, and uniqueness identifies
limits of the numerical scheme. The transport resolvent estimate gives
existence and stability for the linear equations at each time step.

\subsection{Multipliers and critical Besov estimates}
\label{sec:functional-framework}

\begin{lemma}[Multiplier and semigroup estimates]
\label{lem:multiplier-bookkeeping}
Let \(r\in\mathbb R\).  For a scalar \(f\), a tensor \(F\), and
\(t\ge0\),
\begin{align}
 \norm{\mathcal RF}{\B r}
 &\le C\norm{F}{\B r},                                      \label{eq:R-bookkeeping}\\
 \norm{\mathcal Kf}{B^{r+1}}
 &\le C\norm{f}{\B r},                                      \label{eq:K-bookkeeping}\\
 \norm{E_tf}{\B r}
 &\le e^{-at}\norm{f}{\B r}\le\norm{f}{\B r},              \label{eq:E-bookkeeping}\\
 \norm{(I-E_t)f}{\B r}
 &\le C_{a,\rho}t^{\rho/2}
       \norm{f}{B^{r+\rho}},
 \qquad 0<\rho\le2,\quad 0<t\le1.                           \label{eq:E-difference}
\end{align}
The heat estimates hold componentwise for vector and tensor fields as well,
without a zero-mean restriction.  With the convention that \(\mathcal K\)
annihilates constants, \(\curl\mathcal Kf=f-\langle f\rangle\); the
identity \(\curl\mathcal Kf=f\) requires \(\langle f\rangle=0\).
\end{lemma}

These are the periodic versions of the standard multiplier and localized
heat-semigroup bounds; see \cite[Proposition~2.78 and
Section~2.1.2]{BCD11}.  The contraction in \eqref{eq:E-bookkeeping} also
follows directly from the periodic maximum principle.  The multiplier
bound \(\min\{1,t(a+2^{2j})\}\lesssim_{a,\rho}
t^{\rho/2}2^{\rho j}\), also for \(j=-1\), gives
\eqref{eq:E-difference} by summation.

The nonlinear estimates below use Bony's decomposition
\cite[Chapter~2]{BCD11}; for fractional Leibniz rules and
Kato--Ponce commutator estimates in the Sobolev setting,
see \cite{Li19}.
For \(\sigma>0\) and \(0<\delta<1\), we use
\begin{align}
 \|fg\|_{B^0}
 &\le C_\sigma\|f\|_{B^0}\|g\|_{B^\sigma},
 \label{ep:positive-multiplier}\\
 \|fg\|_{B^{-\delta}}
 &\le C_\delta\|f\|_\infty\|g\|_\infty
 \le C_\delta\|f\|_{B^0}\|g\|_{B^0},
 \label{ep:negative-product}\\
 \|[\mathcal R,v\cdot\nabla]f\|_{B^0}
 &\le C\|v\|_{B^1}\|f\|_{B^0},\qquad \Div v=0.
 \label{ep:Rcomm}
\end{align}
For \eqref{ep:positive-multiplier}, Bony's paraproducts are bounded
by the displayed norms. In the high--high remainder, a block of
frequency \(2^k\) can contribute to \(k+C\) lower output blocks;
the factor \((k+C)2^{-\sigma k}\) is bounded. This proves the
estimate without asserting that \(B^0\) is an algebra.
Equation~\eqref{ep:negative-product} follows from
\(L^\infty\hookrightarrow B^{-\delta}\).
For \eqref{ep:Rcomm}, the proof of
\eqref{eq:source-comm-detail} applies at index zero: the high--low
sum has the factor \(2^{\ell-k}\), and incompressibility gives
\(2^{j-k}\) in the high--high sum. Both are summable. The
low--high kernel commutator is controlled by
\(\|\nabla v\|_\infty\|f\|_{B^0}\).

\subsection{Time regularity and uniqueness for the continuous system}
\label{sec:continuous-estimates}

For \(0<\alpha\le1\), the exact-solution assumptions
\eqref{eq:endpoint-exact-class} and the stress equation give
\begin{equation}
 \partial_t\tau
 =\Delta\tau-a\tau+D(u)-Q(\nabla u,\tau)-u\cdot\nabla\tau
 \in L^1(0,T;\B\alpha).
 \label{eq:tau-time-derivative}
\end{equation}
Indeed, \(\Delta\tau\in L^1_TB^\alpha\), the positive-index product
law controls \(Q(\nabla u,\tau)\), and
\[
 \|u\cdot\nabla\tau\|_{L^1_TB^\alpha}
 \lesssim \|u\|_{L^\infty_TB^{1+\alpha}}
           \|\tau\|_{L^1_TB^{1+\alpha}}<\infty.
\]
Thus \(\tau\in W^{1,1}(0,T;B^\alpha)\). In
Section~\ref{sec:endpoint-rate-family}, this estimate controls the sum
of the stress increments over the time steps.

We will use the following uniqueness result to identify limits of the
numerical scheme.

\begin{lemma}[Uniqueness in the critical class]
\label{ep:uniqueness}
There is at most one solution with fixed initial data in the class
\[
 u\in C([0,T];B^1),\qquad
 \tau\in C([0,T];B^0)\cap L^1(0,T;B^2).
\]
\end{lemma}

\begin{proof}
For two such solutions \((u,\tau)\) and \((v,z)\), set
\(w=v-u\), \(Z=z-\tau\). Their difference equations are
\begin{align*}
 \partial_tw+v\cdot\nabla w+w\cdot\nabla u+\nabla q
 &=\operatorname{div}Z,\qquad \operatorname{div}w=0,\\
 \partial_tZ+v\cdot\nabla Z+aZ
       +Q(\nabla v,Z)+Q(\nabla w,\tau)
 &=D(w)+\Delta Z-w\cdot\nabla\tau.
\end{align*}
Pair with \(w\) and \(Z\) in \(L^2\). Symmetry of \(Z\) gives
\[
 \langle\operatorname{div}Z,w\rangle+\langle D(w),Z\rangle=0.
\]
The only term with an uncontrolled derivative of \(w\) is handled by
componentwise integration by parts:
\begin{equation}
 |\langle Q(\nabla w,\tau),Z\rangle|
 \le C_b\|\tau\|_\infty\|w\|_2\|\nabla Z\|_2
      +C_b\|\nabla\tau\|_\infty\|w\|_2\|Z\|_2.
 \label{ep:uniqueness-Q}
\end{equation}
Young's inequality and the other direct \(L^2\) bounds yield
\begin{equation}
 \begin{aligned}
 \frac{d}{dt}\bigl(\|w\|_2^2+\|Z\|_2^2\bigr)+\|\nabla Z\|_2^2
 &\le C_b\bigl(1+\|\nabla u\|_\infty+\|\nabla v\|_\infty
          +\|\tau\|_\infty^2+\|\nabla\tau\|_\infty\bigr)\\
 &\hspace{18mm}\times\bigl(\|w\|_2^2+\|Z\|_2^2\bigr).
 \end{aligned}
 \label{ep:uniqueness-energy}
\end{equation}
The coefficient is integrable by the assumed uniform bounds and
\(\tau\in L^1(0,T;B^1)\). To justify the energy calculation, first
regularize the solutions in space. The equations give
\(u_t,v_t,\tau_t,z_t\in L^1_tL^2_x\), while
\[
 \int_0^T\|\nabla\tau\|_2^2\,dt
 \le\sup_t\|\tau(t)\|_2\int_0^T\|\Delta\tau(t)\|_2\,dt<\infty,
\]
and likewise for \(z\). Thus the regularized energy identity passes
to the limit. Gronwall's inequality with zero initial differences
proves uniqueness.
\end{proof}

\subsection{Transport estimates}

Each sequential update requires inversion of \(I+hv\cdot\nabla\).
The estimate below gives both solvability and control by the velocity
norms that remain available in the subsequent error analysis.

\begin{lemma}[Discrete transport stability]
\label{lem:transport-stability}
Let \(v\in B^{\max\{1,r\}}\) be divergence free and
\(-1<r\le2\). There are \(C_r>0\) and \(c_0(r)=(2C_r)^{-1}\) such that, if
\(h\|v\|_{B^{\max\{1,r\}}}\le c_0(r)\), every
\(g\in B^r\) has a unique solution \(f\in B^r\) of
\(f+hv\cdot\nabla f=g\), and
\begin{equation}
 \|f\|_{B^r}
 \le \frac{\|g\|_{B^r}}{1-C_rh\|v\|_{B^{\max\{1,r\}}}}
 \le \bigl(1+2C_rh\|v\|_{B^{\max\{1,r\}}}\bigr)\|g\|_{B^r}.
 \label{eq:transport-stability}
\end{equation}
In particular, for \(-1<r\le1\), only the \(B^1\) norm of the
velocity enters the estimate.
\end{lemma}

\begin{proof}
We use the dyadic transport commutator estimates of
\cite[Lemma~2.100]{BCD11}.
For divergence-free \(v\) and \(-1<r\le1\), they give
\begin{equation}
 \sum_{j\ge-1}2^{rj}\|[\Delta_j,v\cdot\nabla]f\|_\infty
 \le C_r\|v\|_{B^1}\|f\|_{B^r}.
 \label{eq:critical-tame-transport}
\end{equation}
To explain why \(r=1\) is included and \(r=-1\) is excluded, put
\(V_k=2^k\|\Delta_kv\|_\infty\),
\(F_k=2^{rk}\|\Delta_kf\|_\infty\), and
\(F_k^*=\sum_{|\ell-k|\le N}F_\ell\) for a fixed overlap width \(N\).
In Bony's decomposition, the high--low contribution is bounded by
\[
 C\sum_k\sum_{\ell\le k-2}
 V_kF_\ell\,2^{-(1-r)(k-\ell)}
 \le C\|V\|_{\ell^1}\|F\|_{\ell^1}.
\]
This includes \(r=1\), since both sequences are summable.
Incompressibility puts the derivative of the high--high remainder
on the output block, whose contribution is bounded by
\[
 C\sum_kV_kF_k^*
       \sum_{-1\le j\le k+N}2^{(1+r)(j-k)}
 \le C_r\|V\|_{\ell^1}\|F\|_{\ell^1},
 \qquad r>-1.
\]
The low--high kernel commutator is bounded by
\(C\|\nabla v\|_\infty\|f\|_{B^r}\); the remaining localized
terms and the lowest blocks obey the same bound as above.
For \(1<r\le2\), the usual tame commutator estimate instead gives
\[
 \sum_{j\ge-1}2^{rj}\|[\Delta_j,v\cdot\nabla]f\|_\infty
 \le C_r\bigl(\|\nabla v\|_\infty\|f\|_{B^r}
             +\|v\|_{B^r}\|\nabla f\|_\infty\bigr)
 \le C_r\|v\|_{B^r}\|f\|_{B^r}.
\]
Thus, in both ranges, the stationary maximum principle on each
dyadic block yields
\[
 \|f\|_{B^r}\le\|g\|_{B^r}+C_rh\|v\|_{B^{\max\{1,r\}}}\|f\|_{B^r}.
\]
Absorption proves \eqref{eq:transport-stability} and uniqueness.

For existence, let \(\Phi_t\) be the Lipschitz flow of \(v\).
The same commutator bounds for the time-dependent transport equation
give
\(\|g\circ\Phi_{-t}\|_{B^r}
\le e^{C_rt\|v\|_{B^{\max\{1,r\}}}}\|g\|_{B^r}\). Hence, for smooth \(g\),
\[
 (I+hv\cdot\nabla)^{-1}g
 =\int_0^\infty e^{-s}g\circ\Phi_{-hs}\,ds
\]
is well defined in \(B^r\) and solves the stationary equation by
integration by parts. Smooth dyadic truncations are dense in \(B^r\),
because the Besov summability index is one. The bound therefore
extends the inverse to every \(g\in B^r\), including \(r<0\),
without interpreting rough data pointwise.
\end{proof}

\subsection{Uniform discrete stability}
\label{sec:uniform-discrete-stability}

Combining the transport inverse with the positive-index source bounds
constructs a numerical orbit for a short interval. This provides the
starting point for the quantitative continuation argument in
Section~\ref{sec:exact-interval-continuation}.

\begin{proposition}[Uniform local stability of the discrete orbit]
\label{prop:discrete-high}
Let $0<\alpha\le1$ and let
$\tau^0,\Gamma^0\in B^\alpha$ and \(m_0\in\mathbb R^2\), with
\(\langle\Gamma^0\rangle=0\).
Write \(c_0=c_0(\alpha)\) for the threshold in
Lemma~\ref{lem:transport-stability} at \(r=\alpha\).
There exist \(T_*,h_*,M_*>0\), depending only on \(\alpha\), the
parameters \(a,b\), and the initial norm
\[
 |m_0|+\|\tau^0\|_{B^\alpha}+\|\Gamma^0\|_{B^\alpha},
\]
such that, for \(0<h\le h_*\), the scheme
\eqref{eq:recover-discrete}--\eqref{eq:scheme-gamma} has a unique solution for
$nh\le T_*$ and
\begin{equation}
  \sup_{nh\le T_*}
  \left(
    \norm{u^n}{\B{1+\alpha}}
    +\norm{\tau^n}{\B\alpha}
    +\norm{\Gamma^n}{\B\alpha}
  \right)
  \le M_*,
  \qquad
  h\sup_{nh\le T_*}\norm{u^n}{\B{1+\alpha}}\le c_0.
  \label{eq:discrete-high-bound}
\end{equation}
\end{proposition}

\begin{proof}
Set
\[
  X_n=|m_0|+\norm{\tau^n}{\B\alpha}+\norm{\Gamma^n}{\B\alpha}.
\]
The reconstruction formula and
Lemma~\ref{lem:multiplier-bookkeeping} give
\begin{equation}
 \norm{u^n}{\B{1+\alpha}}
 \le C\left(|m_0|+\norm{\Gamma^n}{\B\alpha}
            +\norm{\tau^n}{\B\alpha}\right)
 =CX_n.
 \label{eq:discrete-recovery-high}
\end{equation}
The positive-index algebra law first gives
\(\|S^n\|_{B^\alpha}\le C_a(X_n+X_n^2)\).
Whenever \(ChX_n\le c_0\), the stress resolvent and heat contraction give
\(\|\tau^{n+1}\|_{B^\alpha}\le C X_n\). The new constitutive term
therefore obeys the same source bound:
\begin{equation}
 \|S^n\|_{B^\alpha}+\|G^n\|_{B^\alpha}
 \le C_a(X_n+X_n^2).
 \label{eq:discrete-source-high}
\end{equation}
Apply the transport bound next to the good-vorticity equation.
Adding the two estimates and the conserved mean, for \(h\le1\), gives
\[
  X_{n+1}
  \le (1+ChX_n)\bigl(X_n+C_a h(X_n+X_n^2)\bigr)
  \le X_n+C_a h(1+X_n)^3.
\]

Choose \(h_*\le1\) and
\[
 R=2X_0+1,\qquad
 C_aT_*(1+R)^3\le R-X_0,\qquad
 Ch_*R\le c_0.
\]
If \(X_k\le R\) for \(k\le n\) and \((n+1)h\le T_*\), the step condition
holds and the scalar recurrence gives
\[
 X_{n+1}
 \le X_0+C_ah\sum_{k=0}^n(1+X_k)^3
 \le X_0+C_a(n+1)h(1+R)^3
 \le R.
\]
Induction, \eqref{eq:discrete-recovery-high}, and resolvent uniqueness
complete the proof with \(M_*=(1+C)R\).
\end{proof}
\section{Critical Convergence Without a Rate}
\label{sec:critical-convergence}

To prove Theorem~\ref{thm:critical-convergence}, we combine the
stress smoothing estimates in Section~\ref{ep:sec-smoothing} with the
uniform frequency tails in Section~\ref{ep:frequency-tails}.
On intervals where the numerical iterates satisfy a common critical-norm
bound, Proposition~\ref{ep:weak-compactness} gives strong compactness in
those norms. Passing to the limit in the discrete equations and using
uniqueness identifies each limit as the prescribed exact solution. The resulting conditional convergence,
together with the one-step bound, permits a first-exit argument that
continues the numerical orbit through $T$.

\subsection{Construction of the critical sequential step}

Throughout this section, a numerical prefix consists of the iterates
$0\le n\le J$ constructed from the fixed initial data, with $Jh\le T$.
Assume
\begin{equation}
 X_n:=|m_0|+\|\tau^n\|_{B^0}+\|\Gamma^n\|_{B^0}\le R,
 \qquad 0\le n\le J.
 \label{ep:prefix}
\end{equation}
Recovery gives \(\|u^n\|_{B^1}\le CR\). All conditional estimates
below are uniform in \(h,J\); their constants may depend on
\(R,T,a,b\) and the indicated regularity indices.

We first construct one further node. Its norm-growth bound will allow
the continuation argument to include a possible first exit node in the
conditional estimates.

\begin{lemma}[A critical step and a bound for its growth]
\label{ep:one-step}
If \(X_n\le R\) and \(h\) is sufficiently small in terms of \(R\),
the sequential stress and good-vorticity equations have unique
solutions with
\(\tau^{n+1}\in B^1\) and \(\Gamma^{n+1}\in B^0\).
Moreover, for \(0<h\le1\),
\begin{equation}
 X_{n+1}\le X_n+C_R\sqrt h .
 \label{ep:overshoot}
\end{equation}
This is a bound for norm growth, not a uniform critical-norm estimate
for the difference between successive good vorticities.
\end{lemma}

\begin{proof}
The frozen source \(S^n\) belongs to \(L^\infty\), with norm at most
\(C_R\). Localized heat bounds give
\(\|E_hf\|_{B^1}\le Ch^{-1/2}\|f\|_\infty\).
Thus Lemma~\ref{lem:transport-stability} at indices zero and one
constructs the unique stress solve, with
\[
 \|\tau^{n+1}\|_{B^1}\le C_Rh^{-1/2},\qquad
 \|\tau^{n+1}\|_{B^0}
 \le\|\tau^n\|_{B^0}+C_R\sqrt h.
\]
The second estimate uses heat contraction on \(\tau^n\) and
\(h\|E_hS^n\|_{B^0}\le C_R\sqrt h\).
Crucially, the constitutive term in \(G^n\) uses this newly computed
stress. Hence \eqref{ep:positive-multiplier} and \eqref{ep:Rcomm} give
\[
 \|G^n\|_{B^0}
 \le C_R\bigl(1+\|\tau^{n+1}\|_{B^1}\bigr)
 \le C_Rh^{-1/2}.
\]
The \(B^0\) transport inverse then constructs \(\Gamma^{n+1}\) and
bounds its norm by \(\|\Gamma^n\|_{B^0}+C_R\sqrt h\).
The symmetry and mean identities from Section~\ref{sec:discrete-scheme}
remain valid. Adding these bounds proves \eqref{ep:overshoot}.
\end{proof}

\subsection{Discrete stress smoothing and variation}
\label{ep:sec-smoothing}

The one-step bound \(\|\tau^{n+1}\|_{B^1}\le C_Rh^{-1/2}\)
from Lemma~\ref{ep:one-step} does not give a uniform bound on
\(h\sum_n\|\tau^n\|_{B^1}\). We therefore estimate the accumulated
smoothing over several steps. This yields two bounds needed for the
critical limit: a time-summed \(B^1\) bound on \(\tau^n\), \(n\ge1\), for the
frequency-tail argument in Section~\ref{ep:frequency-tails}, and a
time-summed \(B^2\) bound on \(E_h\tau^n\) that recovers the regularity of
the limiting stress in Section~\ref{ep:compactness-section}.

The argument has two stages. Lemma~\ref{ep:product-smoothing} estimates
the repeated split evolution, and Proposition~\ref{ep:stress-smoothing}
applies it to the stress. Once this gives an integrable forcing in
\(B^0\), Lemma~\ref{ep:heat-maximal} applies to an exact heat
recurrence and yields Corollary~\ref{ep:stress-variation}.
Throughout, we assume \eqref{ep:prefix} and recall the stress step
\begin{equation}
 S^n=D(u^n)-Q(\nabla u^n,\tau^n),\qquad
 \tau^{n+1}=(I+hu^n\cdot\nabla)^{-1}E_h(\tau^n+hS^n).
 \label{ep:stress-step}
\end{equation}

\paragraph{Smoothing over several steps.}
The estimate below applies to any divergence-free velocity sequence
bounded in \(B^s\). It requires no temporal regularity or evolution
equation for that sequence, and it does not assume that the transport
resolvents commute with the heat operator.
\begin{lemma}[Heat--resolvent smoothing]
\label{ep:product-smoothing}
Let \(1\le s\le2\), \(0<\delta<1\), and let \(v^n\) be
divergence free with \(\sup_n\|v^n\|_{B^s}\le R\). Put
\[
 P_n=(I+hv^n\cdot\nabla)^{-1},\qquad
 \mathcal U(n,k)=P_{n-1}E_h\cdots P_kE_h\quad(n>k).
\]
For sufficiently small \(h\), depending only on \(R,s,\delta\),
uniformly for \(0<(n-k)h\le T\), one has
\begin{equation}
 \|\mathcal U(n,k)f\|_{B^s}
 \le C_{T,R,s,\delta}((n-k)h)^{-(s-q)/2}\|f\|_{B^q},
 \qquad q\in\{s-1,s-1-\delta\}.
 \label{ep:product-smoothing-bound}
\end{equation}
\end{lemma}

\begin{proof}
Fix \(q\in\{s-1,s-1-\delta\}\) and put \(\beta=(s-q)/2<1\).
Write \(A_n=v^n\cdot\nabla\). We first work on intervals of length
at most one; all constants in Steps 1--2 depend only on \(R,s,\delta\).

\emph{Step 1: compare one split step with heat evolution.}
Lemma~\ref{lem:transport-stability} gives
\(\|P_n\|_{B^r\to B^r}\le1+Ch\) for \(r=q,s-1,s\).
The product and heat estimates give
\[
 \|A_ng\|_{B^{s-1}}\le C\|g\|_{B^s},\qquad
 \|A_nE_hg\|_{B^s}\le Ch^{-1/2}\|g\|_{B^s}.
\]
At \(s=1\), the first estimate follows from the mixed product
bound \(B^1\cdot B^0\subset B^0\).
Consequently, for \(m\ge1\), \(mh\le1\),
\begin{equation}
 \|E_{(m-1)h}P_nA_nE_hg\|_{B^s}
       \le C(mh)^{-1/2}\|g\|_{B^s}.
 \label{ep:parametrix-kernel}
\end{equation}
For \(m=1\), use the second estimate above. For \(m\ge2\), use
the first estimate and gain one derivative from the outer heat factor.

\emph{Step 2: sum the corrections on a short interval.}
Shift to \(k=0\), set \(z_0=f\), and let \(z_{n+1}=P_nE_hz_n\).
The first step gives
\[
 \|z_1\|_{B^q}\le C\|f\|_{B^q},\qquad
 \|z_1\|_{B^s}\le Ch^{-\beta}\|f\|_{B^q}.
\]
Starting at this smoothed node avoids applying a correction to the
rough initial datum. The identity \(P_jE_h=E_h-hP_jA_jE_h\) yields
\begin{equation}
 z_n=E_{(n-1)h}z_1-h\sum_{j=1}^{n-1}
              E_{(n-1-j)h}P_jA_jE_hz_j .
 \label{ep:parametrix}
\end{equation}
The first term is bounded by \(C(nh)^{-\beta}\|f\|_{B^q}\):
use the \(B^s\) bound on \(z_1\) for \(n=1\), and its \(B^q\)
bound with heat smoothing for \(n\ge2\).
Thus \eqref{ep:parametrix-kernel} gives, for \(nh\le1\),
\[
 \|z_n\|_{B^s}\le C(nh)^{-\beta}\|f\|_{B^q}
       +Ch\sum_{j=1}^{n-1}((n-j)h)^{-1/2}\|z_j\|_{B^s}.
\]
Since \(\beta<1\), splitting the sum at \(j=n/2\) gives
\[
 h\sum_{j=1}^{n-1}((n-j)h)^{-1/2}(jh)^{-\beta}
       \le C_\beta(nh)^{1/2-\beta}.
\]
For \(M_N=\max_{1\le n\le N}(nh)^\beta\|z_n\|_{B^s}\), it follows
that \(M_N\le C\|f\|_{B^q}+CT_0^{1/2}M_N\) whenever \(Nh\le T_0\).
Choose \(T_0=T_0(R,s,\delta)\le1\) with \(CT_0^{1/2}\le1/2\)
and absorb. The desired smoothing bound holds on every shifted interval
of length at most \(T_0\).

\emph{Step 3: extend to any finite interval.}
Heat contraction and transport stability give
\[
 \|\mathcal U(n,k)\|_{B^q\to B^q}\le e^{C(n-k)h}.
\]
If \((n-k)h>T_0\) and \(h\le T_0/2\), set
\(L=\lfloor T_0/h\rfloor\). In the factorization
\(\mathcal U(n,k)=\mathcal U(n,n-L)\mathcal U(n-L,k)\),
apply stability before the last \(L\) steps and smoothing over those
steps, whose duration lies in \([T_0/2,T_0]\).
This gives \eqref{ep:product-smoothing-bound}, with a constant depending
on \(T\) and a step threshold depending only on \(R,s,\delta\).
\end{proof}

\paragraph{A forced split recurrence.}
The stress also receives a source at each step. Applying the homogeneous
estimate to each contribution gives the Duhamel bound needed for this update.
For the velocity sequences and step sizes in
Lemma~\ref{ep:product-smoothing}, consider
\(z^{n+1}=P_nE_h(z^n+hg^n)\), with \(z^0\in B^{s-1}\),
\(g^n\in B^q\), and \(q\in\{s-1,s-1-\delta\}\).
For \(n\ge1\) and \(\beta=(s-q)/2<1\), the discrete Duhamel formula gives
\begin{equation}
 \|z^n\|_{B^s}\le C\left(t_n^{-1/2}\|z^0\|_{B^{s-1}}
 +h\sum_{k=0}^{n-1}(t_n-t_k)^{-\beta}\|g^k\|_{B^q}\right).
 \label{eq:forced-split-smoothing}
\end{equation}
Here \(C=C_{T,R,s,\delta}\). Consequently, for \(Jh\le T\),
\begin{equation}
 h\sum_{n=1}^J\|z^n\|_{B^s}
 \le C_{T,R,s,\delta}\left(\|z^0\|_{B^{s-1}}
             +h\sum_{k=0}^{J-1}\|g^k\|_{B^q}\right).
 \label{eq:forced-split-integrated}
\end{equation}
Indeed, \(z^n=\mathcal U(n,0)z^0+h\sum_{k<n}\mathcal U(n,k)g^k\);
apply the lemma and sum the kernels using
\[
 h\sum_{m=1}^J(mh)^{-\beta}\le \frac{T^{1-\beta}}{1-\beta}.
\]
For the critical stress, take \(s=1\) and \(q=-\delta\). The resulting
nodal regularity makes the transport term integrable in \(B^0\), allowing
us to rewrite the full step as a heat recurrence with integrable forcing.
The choice \(s=1+\alpha\), \(q=\alpha\) will later give the
stress estimate with \(\alpha\) additional Besov derivatives.

\begin{proposition}[Stress smoothing and integrable heat forcing]
\label{ep:stress-smoothing}
Under \eqref{ep:prefix}, for sufficiently small \(h\),
\begin{align}
 \|\tau^n\|_{B^1}
   &\le C_{T,R,\delta}
       \bigl(t_n^{-1/2}\|\tau^0\|_{B^0}
                              +t_n^{(1-\delta)/2}\bigr),
       &&1\le n\le J, \label{ep:stress-pointwise}\\
 h\sum_{n=1}^J\|\tau^n\|_{B^1}
   &\le C_{T,R,\delta}. \label{ep:stress-integrated}
\end{align}
Define
\begin{equation}
 F_n=E_hS^n-u^n\cdot\nabla\tau^{n+1}.
 \quad\text{Then}\quad
 \tau^{n+1}=E_h\tau^n+hF_n .
 \label{ep:heat-forcing-definition}
\end{equation}
These forcing terms satisfy
\begin{equation}
 \|F_n\|_{B^0}\le C_{T,R,\delta}(1+t_{n+1}^{-1/2}),
 \qquad h\sum_{n=0}^{J-1}\|F_n\|_{B^0}\le C_{T,R,\delta}.
 \label{ep:heat-forcing-bound}
\end{equation}
\end{proposition}

\begin{proof}
\emph{Smoothing of the numerical stress.}
The conditional bound gives \(\|S^n\|_\infty\le C_R\), hence
\(\|S^n\|_{B^{-\delta}}\le C_{R,\delta}\), including \(n=0\).
Apply \eqref{eq:forced-split-smoothing} to the stress step with
\(v^n=u^n\), \(z^n=\tau^n\), \(g^n=S^n\), \(s=1\), and
\(q=-\delta\). Since
\[
 h\sum_{m=1}^n(mh)^{-(1+\delta)/2}
       \le C_\delta t_n^{(1-\delta)/2},
\]
this proves \eqref{ep:stress-pointwise}; summation proves
\eqref{ep:stress-integrated}.

\emph{Control of the heat forcing.}
Rearranging the stress step gives \eqref{ep:heat-forcing-definition}.
For \(n\ge1\), the mixed product law gives
\[
 \|S^n\|_{B^0}\le C_R(1+\|\tau^n\|_{B^1}),\qquad
 \|u^n\cdot\nabla\tau^{n+1}\|_{B^0}
       \le C_R\|\tau^{n+1}\|_{B^1}.
\]
Together with heat contraction, \eqref{ep:stress-pointwise},
and \(t_{n+1}\le2t_n\), these prove the first estimate in
\eqref{ep:heat-forcing-bound} for \(n\ge1\).
At \(n=0\), only \(S^0\in L^\infty\) is available, but the heat
bound from Lemma~\ref{ep:one-step} gives
\(\|E_hS^0\|_{B^0}\le C_Rh^{-1/2}\).
Together with \eqref{ep:stress-pointwise} at \(n=1\), this proves
the first-cell estimate. Finally,
\(h\sum_{n=0}^{J-1}t_{n+1}^{-1/2}\le2\sqrt{Jh}\).
\end{proof}

\paragraph{Variation and regularity after heat evolution.}
The forcing bound \eqref{ep:heat-forcing-bound} allows a second use of
diffusion through the heat recurrence \eqref{ep:heat-forcing-definition}.
The next estimate controls the sum of the stress increments and the
time-summed \(B^2\) norm of \(E_h\tau^n\). The latter gives the regularity
of the limiting stress; the variation bound will also control the
sequential source correction in Section~\ref{sec:quantitative-proofs}.

\begin{lemma}[Discrete variation and smoothing after heat evolution]
\label{ep:heat-maximal}
Let \(r\in\mathbb R\) and suppose
\(z^{n+1}=E_hz^n+h f_n\), with \(z^0,f_n\in B^r\).
Then, for \(Jh\le T\) and \(0<h\le1\),
\begin{align}
 &\sum_{n=0}^{J-1}\|z^{n+1}-z^n\|_{B^r}
  +h\sum_{n=0}^{J-1}\|E_hz^n\|_{B^{r+2}}
 \notag\\
 &\hspace{25mm}\le C_{T,a,r}
       \left(\|z^0\|_{B^r}+h\sum_{n=0}^{J-1}\|f_n\|_{B^r}\right).
 \label{ep:heat-maximal-bound}
\end{align}
\end{lemma}

\begin{proof}
For \(j\ge0\), put \(\lambda_j=2^{2j}\).
Localized periodic heat kernels satisfy, for \(m\ge0\),
\[
 \|E_{mh}\Delta_j f\|_\infty
       \le Ce^{-cmh\lambda_j}\|\Delta_jf\|_\infty,
\]
and
\begin{equation}
 \|(E_h-I)E_{mh}\Delta_j f\|_\infty
 \le C_a\min(h\lambda_j,1)e^{-cmh\lambda_j}
                                      \|\Delta_jf\|_\infty .
 \label{ep:heat-difference-kernel}
\end{equation}
For \(h\lambda_j\le1\), integrate the localized generator
\((\Delta-a)E_t\) over a time interval of length \(h\);
for \(h\lambda_j\ge1\), bound the two heat operators separately.
Thus
\begin{equation}
 \sum_{m\ge0}\min(h\lambda_j,1)e^{-cmh\lambda_j}\le C,
 \qquad
 \sum_{m\ge0}h\lambda_j e^{-c(m+1)h\lambda_j}\le C .
 \label{ep:heat-geometric-sums}
\end{equation}
The discrete Duhamel formula is
\[
 z^n=E_{nh}z^0+h\sum_{k=0}^{n-1}E_{(n-1-k)h}f_k .
\]
Apply \eqref{ep:heat-difference-kernel} to this formula and use
\(z^{n+1}-z^n=(E_h-I)z^n+h f_n\).
After summing first in \(n\), the first geometric bound gives the
variation estimate at every high block.
For \(E_hz^n\), the extra heat step gives the second geometric
bound and hence the full \(B^{r+2}\) estimate.
Multiply these block estimates by \(2^{rj}\) and sum in \(j\).

The low block requires no spectral gap. Heat contraction and
Duhamel give
\[
 \max_{0\le n\le J}\|\Delta_{-1}z^n\|_\infty
 \le \|\Delta_{-1}z^0\|_\infty
                       +h\sum_{n=0}^{J-1}\|\Delta_{-1}f_n\|_\infty.
\]
Moreover, the low-frequency generator is bounded, so that
\(\|(E_h-I)\Delta_{-1}z^n\|_\infty
\le C_a h\|\Delta_{-1}z^n\|_\infty\).
These estimates control the low-frequency contributions to both sums
in \eqref{ep:heat-maximal-bound} by \(C_{T,a,r}\) times its right-hand side.
They also include the constant Fourier mode.
\end{proof}

Applying this estimate to the stress gives the following consequence.
The \(B^2\) bound applies to \(E_h\tau^n\), the stress after heat evolution;
the computed stress \(\tau^{n+1}\) also undergoes a transport solve, which
need not preserve \(B^2\).

\begin{corollary}[Stress variation and two-derivative bounds after heat evolution]
\label{ep:stress-variation}
Under \eqref{ep:prefix}, uniformly in the numerical prefix,
\begin{equation}
 \sum_{n=0}^{J-1}\|\tau^{n+1}-\tau^n\|_{B^0}
 +h\sum_{n=0}^{J-1}\|E_h\tau^n\|_{B^2}\le C_{T,R,\delta}.
 \label{ep:stress-variation-bound}
\end{equation}
\end{corollary}

\begin{proof}
Apply Lemma~\ref{ep:heat-maximal} with \(r=0\),
\(z^n=\tau^n\), and \(f_n=F_n\), using
\eqref{ep:heat-forcing-bound}.
\end{proof}

\begin{remark}[Sharpness of nodal stress smoothing]
\label{rem:nodal-smoothing-sharpness}
When the transport velocity is controlled only in \(B^1\), the
one-derivative nodal gain in Proposition~\ref{ep:stress-smoothing}
is sharp: even for smooth initial stress, the stress after one full
step can fail to belong to \(B^{1+\varepsilon}\) for every
\(\varepsilon>0\). Appendix~\ref{sec:smoothing-sharpness} gives
the counterexample and the complementary estimates at the other
stages of the split step.
\end{remark}

\subsection{Uniform critical frequency tails}
\label{ep:frequency-tails}

To obtain compactness in the critical norms, we need to control the
high-frequency tails as \(h\to0\). A bound in \(B^0\) alone does not
give this property. For integers \(N\ge0\), define
\begin{equation}
 \vartheta_N(j)=\min\{1,2^{(j-N)/2}\},\qquad
 \mathcal T_N^s(f)=\sum_{j\ge-1}\vartheta_N(j)2^{sj}
                                      \|\Delta_jf\|_\infty.
 \label{ep:tail-functional}
\end{equation}
These quantities are finite for every \(f\in B^s\), satisfy
\(\mathcal T_N^s(f)\le\|f\|_{B^s}\), and tend to zero as
\(N\to\infty\) by dominated convergence. Since
\(\vartheta_N(j)=1\) for \(j\ge N\), they control the corresponding
frequency tails. We will propagate them with constants independent
of \(N\), using the neighboring-index bound
\begin{equation}
 0<\vartheta_N(j)\le\vartheta_N(j+1)
                   \le2^{1/2}\vartheta_N(j).
 \label{ep:tail-weight-ratio}
\end{equation}

The next lemma shows that the transport and product estimates preserve
these tail factors, with constants uniform in the cutoff.

\begin{lemma}[Estimates for the frequency tails]
\label{ep:tail-calculus}
Let \(v\in B^1\) be divergence free, \(f\in B^0\), and
\(g\in B^\sigma\), \(\sigma>0\). Uniformly in \(N\ge0\),
\begin{align}
 \sum_{j\ge-1}\vartheta_N(j)
           \|[\Delta_j,v\cdot\nabla]f\|_\infty
 &\le C\|v\|_{B^1}\mathcal T_N^0(f),
 \label{ep:tail-transport-commutator}\\
 \mathcal T_N^0([\mathcal R,v\cdot\nabla]f)
 &\le C\|v\|_{B^1}\mathcal T_N^0(f),
 \label{ep:tail-R-commutator}\\
 \mathcal T_N^0(fg)
 &\le C_\sigma\bigl(\mathcal T_N^0(f)\|g\|_{B^\sigma}
                    +\|f\|_{B^0}\mathcal T_N^0(g)\bigr).
 \label{ep:tail-product}
\end{align}
For sufficiently small \(h\|v\|_{B^1}\), the transport inverse obeys
\begin{equation}
 \mathcal T_N^0((I+hv\cdot\nabla)^{-1}f)
 \le(1+Ch\|v\|_{B^1})\mathcal T_N^0(f).
 \label{ep:tail-resolvent}
\end{equation}
All estimates hold componentwise.
\end{lemma}

\begin{proof}
Write \(V_k=2^k\|\Delta_kv\|_\infty\) and
\(F_\ell=\vartheta_N(\ell)\|\Delta_\ell f\|_\infty\).
The low--high kernel commutator in
\([\Delta_j,v\cdot\nabla]f\) contributes at most
\(C\|\nabla v\|_\infty\mathcal T_N^0(f)\).
By \eqref{ep:tail-weight-ratio}, the high--low sum is bounded by
\[
 C\sum_k\sum_{\ell\le k-2}V_kF_\ell\,
       2^{-(k-\ell)}\frac{\vartheta_N(k)}{\vartheta_N(\ell)}
 \le C\sum_k\sum_{\ell\le k-2}V_kF_\ell\,2^{-(k-\ell)/2}
 \le C\|V\|_{\ell^1}\|F\|_{\ell^1}.
\]
In the high--high terms, incompressibility puts the derivative on
the output block. With \(F_k^*\) a fixed neighboring sum, their
contribution is at most
\[
 C\sum_kV_kF_k^*\sum_{-1\le j\le k+C}
     2^{j-k}\frac{\vartheta_N(j)}{\vartheta_N(k)}
 \le C\|V\|_{\ell^1}\|F\|_{\ell^1}.
\]
Monotonicity controls \(j\le k\), and
\eqref{ep:tail-weight-ratio} controls the remaining finitely many
indices. The terms \(T_{\nabla\Delta_jf}v\) and
\(\mathscr R(v,\nabla\Delta_jf)\) contribute, respectively,
\[
 C\sum_jF_j\sum_{k\ge j-C}2^{j-k}V_k,
 \qquad C\sum_jF_j\sum_{|k-j|\le C}V_k,
\]
which have the same bound. This proves
\eqref{ep:tail-transport-commutator}.

For \eqref{ep:tail-R-commutator}, subtract the mean of \(v\),
which commutes with \(\mathcal R\). The localized estimates in
Lemma~\ref{lem:gamma-source} give the same three sums: the
low--high bound uses \(\|\nabla v\|_\infty\), the high--low
bound has the factor \(2^{\ell-k}\), and the high--high bound
has \(2^{j-k}\). Each occurrence of \(\mathcal R\) is localized
to an input or output block. The lowest periodic block has only
finitely many modes and obeys the same estimates.

The two paraproducts in \(fg\) are bounded by
\(C\|g\|_\infty\mathcal T_N^0(f)\) and
\(C\|f\|_\infty\mathcal T_N^0(g)\).
For the remainder, monotonicity gives
\[
 \begin{aligned}
 \mathcal T_N^0(\mathscr R(f,g))
 &\le C\sum_{k\ge-1}(k+C_*)\vartheta_N(k)
             \|\Delta_kf\|_\infty\|\widetilde\Delta_kg\|_\infty\\
 &\le C_\sigma\mathcal T_N^0(f)\|g\|_{B^\sigma},
 \end{aligned}
\]
since \(\sup_{k\ge-1}(k+C_*)2^{-\sigma k}<\infty\), for a fixed
sufficiently large integer \(C_*\). This proves
\eqref{ep:tail-product}.

Finally, let \(z=(I+hv\cdot\nabla)^{-1}f\in B^0\), whose
existence follows from Lemma~\ref{lem:transport-stability}.
The blockwise stationary maximum principle gives
\[
 \mathcal T_N^0(z)\le\mathcal T_N^0(f)
                  +Ch\|v\|_{B^1}\mathcal T_N^0(z).
\]
All sums are finite because \(\mathcal T_N^0(z)\le\|z\|_{B^0}\).
Absorption proves \eqref{ep:tail-resolvent}, with a step restriction
independent of \(N\).
\end{proof}

Applied to the sequential updates, these estimates give a recurrence for the frequency tails
with coefficients made summable by stress smoothing. Only its first
source contribution needs separate treatment, because the initial stress
has no positive regularity.

\begin{proposition}[Conditional propagation of critical frequency tails]
\label{ep:uniform-critical-tails}
For the fixed initial data, put
\begin{equation}
 r_N=\vartheta_N(-1)|m_0|
          +\mathcal T_N^0(\tau^0)+\mathcal T_N^0(\Gamma^0).
 \label{ep:initial-tail}
\end{equation}
Then \(r_N\to0\). Under \eqref{ep:prefix}, for all sufficiently
small \(h\) and every \(N\ge0\),
\begin{equation}
 \max_{0\le n\le J}
 \bigl(\mathcal T_N^1(u^n)+\mathcal T_N^0(\tau^n)
                          +\mathcal T_N^0(\Gamma^n)\bigr)
 \le C_{T,R}\bigl(r_N+h\log(e/h)\bigr).
 \label{ep:tail-orbit-bound}
\end{equation}
The constant and the step restriction are independent of \(N,J\).
For a standard smooth dyadic low-frequency cutoff \(S_N\), this implies
\begin{equation}
 \max_{0\le n\le J}
 \bigl(\|(I-S_N)u^n\|_{B^1}
       +\|(I-S_N)\tau^n\|_{B^0}
       +\|(I-S_N)\Gamma^n\|_{B^0}\bigr)
 \le C_{T,R}\bigl(r_N+h\log(e/h)\bigr).
 \label{ep:uniform-tail-bound}
\end{equation}
In particular, the limit superior as \(h\to0\), uniformly over
all admissible prefixes, tends to zero as \(N\to\infty\).
\end{proposition}

\begin{proof}
Set
\[
 X_{N,n}=\vartheta_N(-1)|m_0|
            +\mathcal T_N^0(\tau^n)+\mathcal T_N^0(\Gamma^n).
\]
Localized multiplier bounds and \eqref{ep:tail-weight-ratio} give
\(\mathcal T_N^0(\mathcal Rf)\le C\mathcal T_N^0(f)\) and
\(\mathcal T_N^1(\mathcal Kf)\le C\mathcal T_N^0(f)\).
Thus recovery yields \(\mathcal T_N^1(u^n)\le CX_{N,n}\),
including the constant velocity mode.
Interpolation of the critical bound and
\eqref{ep:stress-pointwise}, with \(\delta=1/2\), gives
\begin{equation}
 \|\tau^n\|_{B^{1/2}}\le C_{T,R}(1+t_n^{-1/4}),\qquad
 h\sum_{n=1}^J\bigl(1+\|\tau^n\|_{B^{1/2}}\bigr)\le C_{T,R}.
 \label{ep:summable-positive-stress}
\end{equation}
If \(J=0\), the assertion follows directly from recovery.
Otherwise we first treat the initial source, for which only
\(\|S^0\|_\infty\le C_R\) is available. Localized heat bounds give
\begin{equation}
 h\mathcal T_N^0(E_hS^0)
 \le h\|E_hS^0\|_{B^0}
 \le C_Rh\left(1+\sum_{j\ge0}e^{-ch2^{2j}}\right)
 \le C_Rh\log(e/h).
 \label{ep:first-source-heat-bound}
\end{equation}
Heat contraction and \eqref{ep:tail-resolvent} therefore imply
\[
 \mathcal T_N^0(\tau^1)
 \le(1+C_Rh)\bigl(\mathcal T_N^0(\tau^0)
                                   +C_Rh\log(e/h)\bigr).
\]
The new stress lies in \(B^1\), so the product and commutator
estimates give
\[
 \mathcal T_N^0(G^0)
 \le C_R\bigl((1+\|\tau^1\|_{B^{1/2}})X_{N,0}
                                  +\mathcal T_N^0(\tau^1)\bigr).
\]
After the good-vorticity solve, using
\(h\|\tau^1\|_{B^{1/2}}\le C_{T,R}h^{3/4}\), we obtain
\begin{equation}
 X_{N,1}\le(1+C_{T,R}h^{3/4})r_N+C_{T,R}h\log(e/h).
 \label{ep:first-tail-step}
\end{equation}

For \(1\le n<J\), the same estimates bound
\(\mathcal T_N^0(S^n)\) by
\(C_R(1+\|\tau^n\|_{B^{1/2}})X_{N,n}\).
Solving first for stress and then for good vorticity gives
\begin{align*}
 \mathcal T_N^0(\tau^{n+1})
 &\le(1+C_Rh)\bigl[\mathcal T_N^0(\tau^n)
       +C_Rh(1+\|\tau^n\|_{B^{1/2}})X_{N,n}\bigr],\\
 \mathcal T_N^0(\Gamma^{n+1})
 &\le(1+C_Rh)\bigl[\mathcal T_N^0(\Gamma^n)
       +C_Rh(1+\|\tau^{n+1}\|_{B^{1/2}})X_{N,n}
       +C_Rh\mathcal T_N^0(\tau^{n+1})\bigr].
\end{align*}
Substitute the first estimate into the second and add the conserved
mean contribution. For \(h\le1\), this yields
\[
 X_{N,n+1}\le
 \bigl[1+C_Rh(1+\|\tau^n\|_{B^{1/2}}
                 +\|\tau^{n+1}\|_{B^{1/2}})\bigr]X_{N,n}.
\]
Discrete Gronwall, \eqref{ep:summable-positive-stress}, and
\eqref{ep:first-tail-step} give
\(\max_{n\le J}X_{N,n}\le C_{T,R}(r_N+h\log(e/h))\).
The additive heat-source error enters only at the first step.
Recovery proves \eqref{ep:tail-orbit-bound}.

Finally, almost orthogonality gives, for \(s=0,1\),
\[
 \|(I-S_N)f\|_{B^s}
 \le C\sum_{j\ge N-C}2^{sj}\|\Delta_jf\|_\infty
 \le C\mathcal T_N^s(f),
\]
since \(\vartheta_N(j)\) is bounded below on \(j\ge N-C\)
independently of \(N\). This proves \eqref{ep:uniform-tail-bound}.
\end{proof}

\subsection{Compactness and identification of numerical limits}
\label{ep:compactness-section}

The frequency-tail bound controls the spatial oscillations uniformly,
and estimates for time increments in weaker Besov norms give compactness
in time. We extend each interpolant constantly after its final time node
to compare numerical intervals of different lengths on \([0,T]\).
The bound on \(E_h\tau^n\) also gives the \(L^1_tB^2\) regularity of the
limiting stress needed for uniqueness.

\begin{proposition}[Strong compactness at critical regularity]
\label{ep:weak-compactness}
Consider numerical prefixes of the sequential scheme initialized
with the fixed critical data of this section and satisfying
\eqref{ep:prefix} with a common bound \(R\).
For a nodal field \(f^n\), denote its affine interpolant by \(f_h\)
and define the left- and right-endpoint piecewise constant interpolants by
\[
 f_h^{\mathrm{L}}(t)=f^n,\qquad
 f_h^{\mathrm{R}}(t)=f^{n+1},
 \qquad t\in[t_n,t_{n+1}),\quad 0\le n<J.
\]
Set all three interpolants equal to \(f^J\) on \([t_J,T]\).
For any sequence with \(h\to0\) and \(Jh\to T_*\in[0,T]\), a
subsequence satisfies
\begin{align}
 u_h&\longrightarrow u_* &&\text{in }C([0,T];B^1),
 \label{ep:u-compact}\\
 \tau_h&\longrightarrow\tau_* &&\text{in }C([0,T];B^0),
 \label{ep:tau-weak-compact}\\
 \Gamma_h&\longrightarrow\Gamma_* &&\text{in }C([0,T];B^0).
 \label{ep:gamma-weak-compact}
\end{align}
The one-sided interpolants converge uniformly in the same spatial
norms. The limits are constant after \(T_*\), and, if \(T_*>0\),
\begin{equation}
 \tau_*\in L^1(0,T_*;B^2),
 \qquad \|\tau_*\|_{L^1(0,T_*;B^2)}\le C_{T,R}.
 \label{ep:limit-bounds}
\end{equation}
\end{proposition}

\begin{proof}
We first estimate time increments in weaker Besov norms:
\begin{equation}
 \|\tau^{n+1}-\tau^n\|_{B^{-2}}
 +\|\Gamma^{n+1}-\Gamma^n\|_{B^{-1}}\le C_Rh.
 \label{ep:weak-increments}
\end{equation}

Since \(B^0\hookrightarrow L^\infty\), \(\|S^n\|_\infty\le C_R\).
Write the stress update as
\[
 \tau^{n+1}-\tau^n
 =(E_h-I)\tau^n+hE_hS^n
       -h\operatorname{div}(u^n\otimes\tau^{n+1}).
\]
The first term has \(B^{-2}\) norm at most \(C_Rh\).
For the other terms, use \(L^\infty\hookrightarrow B^{-2}\) and
\(B^1\cdot B^0\subset B^0\). In the good-vorticity update,
\(Q(\nabla u^n,\tau^{n+1})\) is bounded in \(L^\infty\), hence
its image under \(\mathcal R\) is bounded in \(B^{-1}\).
Together with \eqref{ep:Rcomm} and the divergence form of transport,
this proves \eqref{ep:weak-increments}.

The affine stress and good vorticity are uniformly Lipschitz in
\(B^{-2}\) and \(B^{-1}\), respectively, including their constant
extensions. Their uniform \(B^0\) bounds, the compact embedding
\(B^0\Subset B^{-3}\) on the torus, and Arzel\`a--Ascoli give a
subsequence converging in \(C([0,T];B^{-3})\).
By \eqref{ep:weak-increments}, the piecewise constant interpolants
converge uniformly in \(B^{-3}\) to the same limits.

The tail bound \eqref{ep:tail-orbit-bound} also holds for the
affine and one-sided interpolants and their constant extensions.
For each fixed \(N\), convergence of finite dyadic sums followed
by Fatou's lemma gives
\[
 \sup_{t\in[0,T]}\mathcal T_N^0(f_*(t))\le C_{T,R}r_N,
 \qquad f=\tau,\Gamma.
\]
The first-cell term vanishes as \(h\to0\). A low--high frequency
split therefore yields
\begin{equation}
 \sup_{t\in[0,T]}\|f_h(t)-f_*(t)\|_{B^0}
 \le C_N\sup_{t\in[0,T]}\|f_h(t)-f_*(t)\|_{B^{-3}}
       +C_{T,R}\bigl(r_N+h\log(e/h)\bigr).
 \label{ep:frequency-upgrade}
\end{equation}
Given \(\varepsilon>0\), first choose \(N\) so that
\(C_{T,R}r_N<\varepsilon/2\), and then choose \(h\) small enough
to control the \(B^{-3}\) error and \(h\log(e/h)\).
This proves uniform convergence in \(B^0\). The limits are
continuous as uniform limits of the affine interpolants.
The same split and \eqref{ep:weak-increments} give uniform
convergence of the one-sided interpolants to these limits.
Recovery with the common mean \(m_0\) proves the velocity assertion.
Since the extensions are constant after \(Jh\), their continuous
limits are constant after \(T_*\).

It remains to prove \(\tau_*\in L^1(0,T_*;B^2)\) using the bound on
\(E_h\tau^n\). On any \([0,T']\subset[0,T_*)\),
\(E_h\tau_h^{\mathrm{L}}\to\tau_*\) uniformly in \(B^0\): indeed,
\[
 \|E_h\tau_h^{\mathrm{L}}-\tau_*\|_{B^0}
 \le C\|\tau_h^{\mathrm{L}}-\tau_*\|_{B^0}
        +\|(E_h-I)\tau_*\|_{B^0},
\]
and strong continuity of the heat semigroup is uniform on the compact
range of \(\tau_*\). The second term in
\eqref{ep:stress-variation-bound} gives
\[
 \int_0^{Jh}\|E_h\tau_h^{\mathrm{L}}(t)\|_{B^2}\,dt\le C_{T,R}.
\]
Apply Fatou's lemma to each finite sum of dyadic \(B^2\) seminorms,
then let the number of blocks and \(T'\) increase. This proves
\eqref{ep:limit-bounds}. No smoothing estimate is used on the
artificial constant extensions.
\end{proof}

It remains to identify these limits as solutions. Strong critical
convergence permits passage through the nonlinear products; testing
against smooth functions identifies the discrete heat increment with
the stress diffusion and damping.

\begin{lemma}[Identification of the limit]
\label{ep:limit-identification}
If \(T_*>0\), the limit in Proposition~\ref{ep:weak-compactness} solves
the original Oldroyd--B system on \([0,T_*]\), with the prescribed
initial data, and belongs to
\begin{equation}
 \begin{aligned}
 u_*&\in C([0,T_*];B^1),\qquad
 \Gamma_*\in C([0,T_*];B^0),\\
 \tau_*&\in C([0,T_*];B^0)\cap L^1(0,T_*;B^2).
 \end{aligned}
 \label{ep:limit-class}
\end{equation}
\end{lemma}

\begin{proof}
On every compact time interval inside \([0,T_*)\), the affine
interpolants satisfy in distributions
\begin{align*}
 \partial_t\tau_h+\operatorname{div}(u_h^{\mathrm{L}}\otimes\tau_h^{\mathrm{R}})
 &=\frac{E_h-I}{h}\tau_h^{\mathrm{L}}+E_hS_h^{\mathrm{L}},\\
 \partial_t\Gamma_h+\operatorname{div}(u_h^{\mathrm{L}}\Gamma_h^{\mathrm{R}})&=G_h^{\mathrm{L}}.
\end{align*}
Here \(S_h^{\mathrm{L}}=S^n\) and \(G_h^{\mathrm{L}}=G^n\) on \([t_n,t_{n+1})\).
Critical convergence implies uniform convergence of
\(u_h^{\mathrm{L}},\nabla u_h^{\mathrm{L}},\tau_h^{\mathrm{L}},\tau_h^{\mathrm{R}},\Gamma_h^{\mathrm{R}}\) in \(L^\infty_x\).
In particular,
\[
 Q(\nabla u_h^{\mathrm{L}},\tau_h^{\eta})
 \longrightarrow Q(\nabla u_*,\tau_*)
 \quad\text{in }L^\infty_{t,x},\qquad
 \eta\in\{\mathrm{L},\mathrm{R}\}.
\]
This uses ordinary products in \(L^\infty\), not a \(B^0\) algebra
estimate. Transfer the heat operators to smooth test functions to pass
to the limit in \(E_hS_h^{\mathrm{L}}\) and identify the limit of
\((E_h-I)\tau_h^{\mathrm{L}}/h\) as \((\Delta-a)\tau_*\).

The transport products converge directly. For the commutator, use
\[
 [\mathcal R,u\cdot\nabla]\tau
 =\mathcal R\operatorname{div}(u\otimes\tau)
       -\operatorname{div}(u\otimes\mathcal R\tau).
\]
Both products converge, since \(\mathcal R\) is bounded on \(B^0\).
The limiting equations are therefore \eqref{eq:tau} and
\eqref{eq:gamma}. Recovery, the common mean, and the initial values
also pass to the limit. Combining these two equations yields the
vorticity equation; divergence-free recovery then gives
\eqref{eq:u} with a distributional pressure. The stated regularity,
including continuity at \(T_*\), was already proved in
Proposition~\ref{ep:weak-compactness}.
\end{proof}

\subsection{Critical convergence and continuation}
\label{ep:continuation-section}

The limit-identification lemma and uniqueness for the continuous system
show that every subsequential limit is the prescribed exact solution.
For continuation, we need this convergence to hold uniformly over all
prefix lengths satisfying the same bound \eqref{ep:prefix}.

\begin{proposition}[Conditional convergence on arbitrary prefixes]
\label{ep:conditional-convergence}
Let \((u,\tau)\) satisfy \eqref{ep:exact}, and let
\(\Gamma=\operatorname{curl}u-\mathcal R\tau\).
For every fixed \(R\), there is \(\varepsilon_R(h)\to0\) such that
every exactly initialized numerical prefix satisfying
\eqref{ep:prefix} obeys
\begin{equation}
 \max_{0\le n\le J}
 \left(\|u^n-u(t_n)\|_{B^1}
       +\|\tau^n-\tau(t_n)\|_{B^0}
       +\|\Gamma^n-\Gamma(t_n)\|_{B^0}\right)
 \le\varepsilon_R(h).
 \label{ep:conditional-critical-error}
\end{equation}
The estimate is uniform in the prefix length \(J\); no rate is asserted.
\end{proposition}

\begin{proof}
Take any sequence of admissible prefixes with \(h_k\to0\).
After extraction, \(J_kh_k\to T_*\in[0,T]\), and
Proposition~\ref{ep:weak-compactness} gives uniform critical convergence
of their constant extensions on \([0,T]\).
If \(T_*>0\), identification and uniqueness show that the limit
is the prescribed exact solution on \([0,T_*]\), extended constantly
after \(T_*\).
If \(T_*=0\), the weak increment bound identifies the limit as the
constant initial datum, with the same conclusion.
By time continuity, the exact solution extended constantly after
\(J_kh_k\) converges uniformly in the critical norms to this same limit.
Consequently the maximum nodal error along the extracted sequence
tends to zero. A subsequence contradiction proves the assertion
uniformly over all admissible \(J\).
\end{proof}

We now show that the numerical bound holds through \(T\). The one-step
estimate bounds how far the first iterate exceeding the chosen threshold
could lie above it; conditional convergence then rules out such an
iterate by comparison with the exact solution.

\begin{proof}[Proof of Theorem~\ref{thm:critical-convergence}]
Set
\[
 L=\sup_{0\le t\le T}
   \bigl(|m_0|+\|\tau(t)\|_{B^0}+\|\Gamma(t)\|_{B^0}\bigr),
 \qquad R=L+1.
\]
Choose \(h\) so that Lemma~\ref{ep:one-step} applies at threshold
\(R\), its norm increase is less than one, and
\(\varepsilon_{R+1}(h)<1\) in
Proposition~\ref{ep:conditional-convergence}.
Starting from \(X_0<R\), construct successive nodes while \(X_n\le R\).

If \(j\) were the first node with \(t_j\le T\) and \(X_j>R\),
Lemma~\ref{ep:one-step} would construct this node with \(X_j\le R+1\).
The entire prefix through \(j\) would then satisfy the threshold
\(R+1\), so conditional convergence would give
\[
 X_j\le |m_0|+\|\tau(t_j)\|_{B^0}+\|\Gamma(t_j)\|_{B^0}
                +\varepsilon_{R+1}(h)
 \le L+\varepsilon_{R+1}(h)<R,
\]
a contradiction. Thus every node up to \(T\) is uniquely defined
and satisfies \(X_n\le R\). Applying
\eqref{ep:conditional-critical-error} to the full interval proves
Theorem~\ref{thm:critical-convergence}.
\end{proof}

\section{Algebraic Convergence in the Base Solution Class}
\label{sec:quantitative-proofs}

We prove Theorem~\ref{thm:base-algebraic} by first estimating the error
on an existing numerical interval satisfying a fixed Besov bound.
Lemma~\ref{lem:exponential-reconstruction} constructs a continuous path
through the numerical nodes. After a spatial cutoff, this path satisfies
a forced Oldroyd--B system. Mixed stability and the forcing estimate
then give an error bound in $B^0\times B^{-1}$; see
Lemma~\ref{lem:coupled-mixed-continuous-stability} and
Proposition~\ref{prop:coupled-reconstruction-forcing}. Choosing the cutoff
scale $h^{-1/2}$ and interpolating gives the critical-norm rate in
Proposition~\ref{prop:conditional-convergence}. Finally,
Proposition~\ref{prop:exact-interval-continuation} proves the required
numerical bound and extends the orbit through $T$.

\subsection{Conditional setting and exponential reconstruction}
\label{sec:common-error-identities}

We use \(u(t_n),\tau(t_n),\Gamma(t_n)\) for the exact samples and
\(u^n,\tau^n,\Gamma^n\) for the numerical iterates.
The errors \(e_f^n\), \(f=u,\tau,\Gamma\),
are those of Theorem~\ref{thm:base-algebraic}. We use its uniform grid
\(h=T/M\), \(t_n=nh\), and write \(I_n=[t_n,t_{n+1}]\).
The exact and numerical velocity means agree, so
\begin{equation}
 e_u^n=\mathcal K(e_\Gamma^n+\mathcal Re_\tau^n),
 \qquad \Div e_u^n=0.
 \label{eq:common-error-recovery}
\end{equation}
The multiplier bounds in Lemma~\ref{lem:multiplier-bookkeeping} give,
for any \(r\in\mathbb R\) for which the right-hand side is finite,
\begin{equation}
 \|e_u^n\|_{B^{r+1}}
 \le C_r\bigl(\|e_\Gamma^n\|_{B^r}+\|e_\tau^n\|_{B^r}\bigr).
 \label{eq:common-error-recovery-bound}
\end{equation}

Fix \(0<\alpha\le1\), an exact solution satisfying
\eqref{eq:endpoint-exact-class} on \([0,T]\), and \(R>0\).
Until the continuation argument in
Section~\ref{sec:exact-interval-continuation}, assume that the numerical
iterates have been constructed from the exact initial data for
\(0\le n\le J\), where \(Jh\le T\), and satisfy
\begin{equation}
 \max_{0\le n\le J}X_n\le R,\qquad
 X_n:=|m_0|+\|\tau^n\|_{B^\alpha}+\|\Gamma^n\|_{B^\alpha}.
 \label{eq:conditional-prefix-bound}
\end{equation}
Constants in the conditional estimates may depend on \(T,R,\alpha,a,b\)
and the stated exact-solution norms, but never on \(J\) or \(h\).
The step size is sufficiently small in terms of these fixed quantities.
No existence or numerical bound beyond \(t_J\) is assumed.
Write \(S(v,z)=D(v)-Q(\nabla v,z)\),
\(S^n=S(u^n,\tau^n)\), and
\(G^n=G(u^n,\tau^n)+J^n\), as in \eqref{eq:source-delay}.
Recovery and the positive-index product estimates give
\[
 \max_{n\le J}\|u^n\|_{B^{1+\alpha}}
 +\max_{n<J}\bigl(\|S^n\|_{B^\alpha}+\|G^n\|_{B^\alpha}\bigr)
 \le C_R.
\]
If \(J=0\), every error estimate below is immediate, so in their proofs
we may take \(J\ge1\).

Within each time interval, the reconstructed stress is the heat evolution
of \(\tau^n\) plus a correction that matches \(\tau^{n+1}\).
The reconstruction satisfies the continuous equations with additional
forcing terms, called defects below. This allows us to estimate the
nodal errors using the mixed stability lemma.

\begin{lemma}[Reconstruction and defect identities]
\label{lem:exponential-reconstruction}
Assume \eqref{eq:conditional-prefix-bound}. Set
\begin{equation}
 \zeta^n=\tau^{n+1}-E_h\tau^n
       =h\bigl(E_hS^n-u^n\cdot\nabla\tau^{n+1}\bigr).
 \label{eq:reconstruction-stress-correction}
\end{equation}
For \(0\le\ell\le h\), define
\begin{equation}
\begin{aligned}
 z(t_n+\ell)&=E_\ell\tau^n+\frac{\ell}{h}\zeta^n,\\
 \gamma(t_n+\ell)&=\Gamma^n+\frac{\ell}{h}(\Gamma^{n+1}-\Gamma^n),\qquad
 v=m_0+\mathcal K(\gamma+\mathcal Rz).
\end{aligned}
 \label{eq:exponential-reconstruction}
\end{equation}
These paths match the numerical values at every node. They satisfy
\(v\in C([0,t_J];B^{1+\alpha})\) and
\(z,\gamma\in C([0,t_J];B^\alpha)\), with
\begin{equation}
 \sup_{0\le t\le t_J}
 \bigl(\|v(t)\|_{B^{1+\alpha}}+\|z(t)\|_{B^\alpha}
                              +\|\gamma(t)\|_{B^\alpha}\bigr)\le C_R,
 \qquad
 \|\zeta^n\|_{B^{\alpha-1}}+\|\Gamma^{n+1}-\Gamma^n\|_{B^{\alpha-1}}
 \le C_Rh .
 \label{eq:reconstruction-weak-increments}
\end{equation}
Moreover, \(z\) is symmetric, \(\operatorname{div}v=0\),
\(\langle v\rangle=m_0\), and \(\langle\gamma\rangle=0\).

Define the distributional defects
\[
 d_\tau=z_t-(\Delta-a)z-S(v,z)+v\cdot\nabla z,\qquad
 d_\Gamma=\gamma_t+v\cdot\nabla\gamma-G(v,z).
\]
On \(I_n\), write \(\ell=t-t_n\). Then
\begin{align}
 d_\tau={}&-\frac{\ell}{h}(\Delta-a)\zeta^n
          +(E_h-I)S^n+S^n-S(v,z)\notag\\
          &+v\cdot\nabla z-u^n\cdot\nabla\tau^{n+1},
          \label{eq:reconstruction-stress-defect}\\
 d_\Gamma={}&G^n-G(v,z)
                 +v\cdot\nabla\gamma-u^n\cdot\nabla\Gamma^{n+1}.
          \label{eq:reconstruction-gamma-defect}
\end{align}
The reconstructed velocity satisfies
\begin{equation}
 v_t+\Pa(v\cdot\nabla v)-\Pa\Div z
       =\mathcal K(d_\Gamma+\mathcal Rd_\tau).
 \label{eq:reconstruction-momentum-defect}
\end{equation}
\end{lemma}

\begin{proof}
The scheme gives the formula for \(\zeta^n\) and
\(\Gamma^{n+1}-\Gamma^n=hG^n-hu^n\cdot\nabla\Gamma^{n+1}\).
Since transport is the divergence of a positive-index product,
\[
 \|u^n\cdot\nabla\tau^{n+1}\|_{B^{\alpha-1}}
 +\|u^n\cdot\nabla\Gamma^{n+1}\|_{B^{\alpha-1}}\le C_R.
\]
This proves the weak increment bounds. Also
\(\|\zeta^n\|_{B^\alpha}\le2R\) by heat contraction.
The uniform reconstruction bounds, node matching, symmetry, and mean
conditions now follow directly from the definitions.

Differentiating on each cell proves the defect identities; the heat
evolution of \(\tau^n\) cancels exactly in \(z_t-(\Delta-a)z\).
Continuity across the nodes rules out jump measures. Finally,
\(\omega=\gamma+\mathcal Rz=\curl v\) satisfies
\[
 \omega_t+v\cdot\nabla\omega
   =\curl\Div z+d_\Gamma+\mathcal Rd_\tau .
\]
The same cancellation as in the exact good-vorticity equation,
together with the fixed velocity mean, gives
\eqref{eq:reconstruction-momentum-defect} by periodic Hodge recovery.
\end{proof}

\subsection{Coupled stability in mixed Besov norms}
\label{sec:coupled-optimal}

The operator \(\mathcal K\) maps the defects from \(B^{-1}\) to
\(B^0\), suggesting that we measure the velocity error in \(B^0\)
and the stress error in \(B^{-1}\). Stress diffusion also controls
the time integral of \(\|z-\tau\|_{B^1}\), which bounds the forcing
\(\Div(z-\tau)\) in the velocity difference equation.

\begin{lemma}[Mixed Besov stability for a forced comparison]
\label{lem:coupled-mixed-continuous-stability}
Let \(0<\alpha\le1\), \(a\ge0\), \(b\in[-1,1]\), and \(T>0\).
Let \((u,\tau)\) solve the Oldroyd--B system and let \((v,z)\) solve
\begin{align*}
 \partial_tv+v\cdot\nabla v+\nabla\pi
     &=\Div z+f,\qquad \Div v=0,\\
 \partial_tz+v\cdot\nabla z+az+Q(\nabla v,z)
     &=D(v)+\Delta z+g.
\end{align*}
Assume that both velocities are divergence free and that
\[
 M:=1+\|u\|_{L^\infty_TB^{1+\alpha}}
       +\|v\|_{L^\infty_TB^{1+\alpha}}
       +\|\tau\|_{L^\infty_TB^\alpha}<\infty,
 \qquad
 L:=\|\tau\|_{L^1_TB^1}<\infty.
\]
Set \(w=v-u\) and \(Z=z-\tau\).
For the a priori calculation, assume
\[
 w\in C([0,T];B^0),\qquad
 Z\in C([0,T];B^{-1})\cap L^1(0,T;B^1),
\]
and assume the difference equations hold distributionally, with
\(f\in L^1_TB^0\), \(g\in L^1_TB^{-1}\).
These qualitative regularities hold when the comparison path is
spatially filtered and piecewise smooth in time and the exact solution
has the regularity in \eqref{eq:endpoint-exact-class}.
Then
\begingroup
\allowdisplaybreaks[0]
\begin{align}
 &\|w\|_{L^\infty_TB^0}
   +\|Z\|_{L^\infty_TB^{-1}}
   +\|Z\|_{L^1_TB^1}\notag\\
 &\qquad\le C_{\alpha,b,T,M,L}
 \left(\|w(0)\|_{B^0}+\|Z(0)\|_{B^{-1}}
        +\|f\|_{L^1_TB^0}+\|g\|_{L^1_TB^{-1}}\right).
 \label{eq:coupled-mixed-stability}
\end{align}
\endgroup
The constant is independent of the size of the assumed
\(L^1(0,T;B^1)\) norm of \(Z\); the estimate controls that norm itself.
No separate smoothing bound for the comparison stress \(z\) enters
the constant.
\end{lemma}

\begin{proof}
We use \eqref{ep:positive-multiplier} with \(\sigma=\alpha\), together with
\begin{equation}
 \|FG\|_{B^{-\rho}}
   \le C_\rho\|F\|_{B^\rho}\|G\|_{B^{-\rho}},\qquad \rho>0.
 \label{eq:coupled-zero-sum-product}
\end{equation}
For \eqref{eq:coupled-zero-sum-product}, the paraproduct terms follow
from the usual low-frequency estimates. For the high--high remainder, take
\(A_k=2^{\rho k}\|\Delta_kF\|_\infty\) and
\(B_k=2^{-\rho k}\|\widetilde\Delta_kG\|_\infty\).
The remainder is bounded by
\[
 C\sum_{k\ge-1}A_kB_k\sum_{-1\le j\le k+C}2^{-\rho j}
 \le C_\rho\|A\|_{\ell^1}\|B\|_{\ell^1}.
\]
In particular, the endpoint product estimate
\(B^1\cdot B^{-1}\subset B^{-1}\) is available, although \(B^0\)
itself is not an algebra.

\emph{Velocity: remove the pressure without losing a derivative.}
Since \(v,w\) are divergence free,
\(\Div(v\cdot\nabla w)=\Div(w\cdot\nabla v)\).
Subtracting the momentum equations and eliminating the pressure gives
\begin{equation}
 \partial_tw+v\cdot\nabla w
 =\Pa\Div Z-w\cdot\nabla u
       +(I-\Pa)\bigl(w\cdot\nabla(u+v)\bigr)+\Pa f.
 \label{eq:coupled-mixed-momentum}
\end{equation}
Apply the componentwise \(B^0\) transport estimate and
\eqref{ep:positive-multiplier}. Since the Leray multipliers
are bounded on \(B^0\), for every \(0\le t\le T\),
\begin{equation}
 \|w\|_{L^\infty_tB^0}
 \le \|w(0)\|_{B^0}
   +C_\alpha M\int_0^t\|w(s)\|_{B^0}\,ds
   +C_0\int_0^t\|Z(s)\|_{B^1}\,ds
   +C\|f\|_{L^1_tB^0}.
 \label{eq:coupled-mixed-velocity-bound}
\end{equation}
Only \(\|v\|_{B^1}\) is needed for the transport commutator.
No maximum principle for the Leray projection is used.

\emph{Stress: use diffusion to absorb the lower-order terms.}
The stress difference satisfies
\begin{align}
 \partial_tZ-(\Delta-a)Z
  ={}&-v\cdot\nabla Z-w\cdot\nabla\tau
      -Q(\nabla v,Z)-Q(\nabla w,\tau)
      +D(w)+g.
 \label{eq:coupled-mixed-stress}
\end{align}
Treat its right-hand side as heat forcing. Divergence form and the
two product laws give
\begin{align*}
 \|v\cdot\nabla Z\|_{B^{-1}}
    +\|Q(\nabla v,Z)\|_{B^{-1}}
 &\le C_{\alpha,b}M
       \bigl(\|Z\|_{B^0}+\|Z\|_{B^{-\alpha}}\bigr),\\
 \|w\cdot\nabla\tau\|_{B^{-1}}
    +\|Q(\nabla w,\tau)\|_{B^{-1}}+\|D(w)\|_{B^{-1}}
 &\le C_{\alpha,b}(1+M+\|\tau\|_{B^1})\|w\|_{B^0}.
\end{align*}
In particular, the derivative on \(w\) is paired with the integrable
\(B^1\) norm of the exact stress. Interpolation and Young's inequality
give, for every fixed \(\varepsilon>0\),
\[
 M\bigl(\|Z\|_{B^0}+\|Z\|_{B^{-\alpha}}\bigr)
 \le \varepsilon\|Z\|_{B^1}
       +C_{\alpha,\varepsilon}(1+M^2)\|Z\|_{B^{-1}}.
\]
Here one interpolates \(B^{-1}\) and \(B^1\), with exponents
\(1/2\) for \(B^0\) and \((1+\alpha)/2\) on the weak norm for
\(B^{-\alpha}\); at \(\alpha=1\), the latter is already \(B^{-1}\).
The inhomogeneous heat estimate, with \(\varepsilon\) small enough,
therefore yields
\begin{align}
 \|Z\|_{L^\infty_tB^{-1}}+c_1\|Z\|_{L^1_tB^1}
 \le{}&C_1\|Z(0)\|_{B^{-1}}+C_1\|g\|_{L^1_tB^{-1}}\notag\\
 &+C_{\alpha,b}\int_0^t
   \Bigl((1+M^2)\|Z(s)\|_{B^{-1}}
       +(1+M+\|\tau(s)\|_{B^1})\|w(s)\|_{B^0}\Bigr)\,ds.
 \label{eq:coupled-mixed-stress-bound}
\end{align}
The lowest dyadic blocks contribute \(C\int_0^t\|Z\|_{B^{-1}}\),
already included above, so the estimate also covers \(a=0\) and
the constant stress mode.

Multiply \eqref{eq:coupled-mixed-stress-bound} by a fixed constant
large enough to absorb the \(C_0\|Z\|_{L^1_tB^1}\) term in
\eqref{eq:coupled-mixed-velocity-bound}. Adding the inequalities and
applying Gronwall with the integrable coefficient
\(C_{\alpha,b}(1+M^2+\|\tau(t)\|_{B^1})\) proves
\eqref{eq:coupled-mixed-stability}.
\end{proof}

\begin{remark}[Compatibility with the base exact class]
Interpolation of the exact stress norms gives
\[
 \|\tau\|_{L^1_TB^1}
 \le C T^{(1+\alpha)/2}
       \|\tau\|_{L^\infty_TB^\alpha}^{(1+\alpha)/2}
       \|\tau\|_{L^1_TB^{\alpha+2}}^{(1-\alpha)/2}.
\]
At \(\alpha=1\), use \(T\|\tau\|_{L^\infty_TB^1}\) directly.
Thus the lemma applies to the filtered reconstruction below, with
constants controlled by the stated exact norms and
\eqref{eq:conditional-prefix-bound}. The strict condition
\(\alpha>0\) enters \eqref{ep:positive-multiplier};
its constant need not remain bounded as \(\alpha\downarrow0\).
\end{remark}

\subsection{Discrete stress smoothing}

Lemma~\ref{ep:product-smoothing} gives a uniform bound on the \(h\)-weighted time sum
of the numerical \(B^{1+\alpha}\) stress norms. We use it below to
control both the reconstruction correction \(\zeta^n\) and the sum
of successive stress increments. The latter bound controls the source
correction \(J^n\) introduced by the sequential update.

\begin{proposition}[Integrated numerical stress regularity]
\label{cor:coupled-numerical-smoothing}
Under \eqref{eq:conditional-prefix-bound}, the numerical stress
iterates satisfy
\begin{equation}
 \|\tau^n\|_{B^{1+\alpha}}\le C_{T,R,\alpha}(t_n^{-1/2}+1)
 \quad(1\le n\le J),\qquad
 h\sum_{n=1}^J\|\tau^n\|_{B^{1+\alpha}}\le C_{T,R,\alpha}.
 \label{eq:coupled-integrated-stress}
\end{equation}
For \(\zeta^n=\tau^{n+1}-E_h\tau^n\), one also has
\begin{equation}
 h\sum_{n=0}^{J-1}\|\zeta^n\|_{B^\alpha}\le C_{T,R,\alpha}h.
 \label{eq:coupled-integrated-correction}
\end{equation}
The stress also has uniformly bounded total variation in \(B^\alpha\):
\begin{equation}
 \sum_{n=0}^{J-1}\|\tau^{n+1}-\tau^n\|_{B^\alpha}
 \le C_{T,R,\alpha}.
 \label{eq:coupled-stress-variation}
\end{equation}
The constants are independent of \(J,h\). No \(B^{1+\alpha}\)
norm of the initial stress is required.
\end{proposition}

\begin{proof}
Heat smoothing and the transport resolvent estimate at index
\(1+\alpha\) give \(\tau^n\in B^{1+\alpha}\) for every \(n\ge1\).
Uniqueness of the transport solve in \(B^\alpha\) shows that these
more regular solutions coincide with the given iterates.
Apply \eqref{eq:forced-split-smoothing}--\eqref{eq:forced-split-integrated}
with \(v^n=u^n\), \(z^n=\tau^n\), \(g^n=S^n\),
\(s=1+\alpha\), \(q=\alpha\), and \(\delta=1/2\).
The uniform \(B^\alpha\) source bound and
\(h\sum_{j=1}^n(jh)^{-1/2}\le2\sqrt{t_n}\) give both estimates in
\eqref{eq:coupled-integrated-stress}. Finally the numerical step gives
\[
 \zeta^n=h(E_hS^n-u^n\cdot\nabla\tau^{n+1}),\qquad
 \|\zeta^n\|_{B^\alpha}
 \le C_{T,R,\alpha}h(1+\|\tau^{n+1}\|_{B^{1+\alpha}}).
\]
Multiplication by \(h\) and summation prove
\eqref{eq:coupled-integrated-correction}.
For the variation bound, set
\[
 F_n=E_hS^n-u^n\cdot\nabla\tau^{n+1},\qquad
 \tau^{n+1}=E_h\tau^n+hF_n.
\]
The same source and transport estimates give
\(h\sum_{n<J}\|F_n\|_{B^\alpha}\le C_{T,R,\alpha}\).
Lemma~\ref{ep:heat-maximal}, applied at \(r=\alpha\), therefore proves
\eqref{eq:coupled-stress-variation}.
\end{proof}

\subsection{Conditional error estimate}
\label{sec:conditional-error}

\paragraph{Spatial cutoff.}
Before applying mixed stability, we apply a spatial Fourier cutoff to
the reconstruction so that its defects have the required regularity.
We then choose the cutoff scale to balance the resulting spatial error
against the time-discretization defects and estimate the difference
between the filtered and unfiltered paths. We use the
paths \((v,z,\gamma)\) in \eqref{eq:exponential-reconstruction}, which
match every numerical node, and the defect identities
\eqref{eq:reconstruction-stress-defect}--%
\eqref{eq:reconstruction-momentum-defect}.

Choose a fixed real even function
\(\chi\in C_c^\infty(\mathbb R^2)\), equal to one on the unit ball
and zero outside the ball of radius two, and let
\(P_\Lambda=\chi(D/\Lambda)\), \(\Lambda\ge1\).
Its periodic convolution kernel has integral one, uniformly bounded
\(L^1\) norm, and, for every fixed \(q>0\),
\[
 \int_{\mathbb T^2}|K_\Lambda(y)|d(y,0)^q\,dy
 \le C_q\Lambda^{-q},
\]
where \(d\) is the periodic distance.  These bounds follow by
periodizing the rescaled Schwartz kernel of \(\chi(D)\).
The cutoff is used only in the analysis.  Set
\(v_\Lambda=P_\Lambda v\), \(z_\Lambda=P_\Lambda z\).
The defect identities give
\begin{align}
 (v_\Lambda)_t+\Pa(v_\Lambda\cdot\nabla v_\Lambda)
 &=\Pa\Div z_\Lambda+f_\Lambda,\notag\\
 (z_\Lambda)_t+v_\Lambda\cdot\nabla z_\Lambda+az_\Lambda
                  +Q(\nabla v_\Lambda,z_\Lambda)
 &=D(v_\Lambda)+\Delta z_\Lambda+g_\Lambda,
 \label{eq:coupled-filtered-system}
\end{align}
where
\begingroup
\allowdisplaybreaks[0]
\begin{align}
 f_\Lambda={}&\mathcal K(P_\Lambda d_\Gamma+\mathcal RP_\Lambda d_\tau)
 +\Pa\{v_\Lambda\cdot\nabla v_\Lambda-P_\Lambda(v\cdot\nabla v)\},\notag\\
 g_\Lambda={}&P_\Lambda d_\tau
 +v_\Lambda\cdot\nabla z_\Lambda-P_\Lambda(v\cdot\nabla z)\notag\\
 &+Q(\nabla v_\Lambda,z_\Lambda)-P_\Lambda Q(\nabla v,z).
 \label{eq:coupled-forces}
\end{align}
\endgroup

Filtering also changes nonlinear products. The following estimate
quantifies this error, using the extra derivative on velocity or the
time-integrated extra derivative on stress.

\begin{lemma}[A product-cutoff estimate with one extra derivative]
\label{lem:coupled-product-cutoff}
Let \(0<\alpha\le1\), and let \(P_\Lambda\), \(\Lambda\ge1\), be the
fixed smooth Fourier cutoff used in the analysis above. Define
\[
 \mathcal C_\Lambda(f,g)
   =P_\Lambda(fg)-(P_\Lambda f)(P_\Lambda g).
\]
Then
\begin{equation}
 \|\mathcal C_\Lambda(f,g)\|_{B^0}
 \le C_\alpha\Lambda^{-1-\alpha}
                  \|f\|_{B^{1+\alpha}}\|g\|_{B^\alpha}.
 \label{eq:coupled-cutoff-product}
\end{equation}
Consequently, for divergence-free \(v\in B^{1+\alpha}\),
\begin{equation}
\begin{aligned}
 \|P_\Lambda(v\cdot\nabla v)
       -(P_\Lambda v)\cdot\nabla(P_\Lambda v)\|_{B^0}
 &\le C_\alpha\Lambda^{-1-\alpha}\|v\|_{B^{1+\alpha}}^2,\\
 \|P_\Lambda(v\cdot\nabla z)
       -(P_\Lambda v)\cdot\nabla(P_\Lambda z)\|_{B^{-1}}
 &\le C_\alpha\Lambda^{-1-\alpha}
                  \|v\|_{B^{1+\alpha}}\|z\|_{B^\alpha}.
\end{aligned}
 \label{eq:coupled-cutoff-transport}
\end{equation}
The same product estimate, with its two inputs exchanged, gives
\begin{equation}
 \|P_\Lambda Q(\nabla v,z)
       -Q(\nabla P_\Lambda v,P_\Lambda z)\|_{B^0}
 \le C_{\alpha,b}\Lambda^{-1-\alpha}
                  \|v\|_{B^{1+\alpha}}\|z\|_{B^{1+\alpha}}.
 \label{eq:coupled-cutoff-constitutive}
\end{equation}
\end{lemma}

\begin{proof}
All product estimates are componentwise.  We use the mixed product
bound \eqref{ep:positive-multiplier}.

Let \(K_\Lambda\) be the periodic kernel of \(P_\Lambda\), normalized
to have integral one, and put
\(\delta_yf(x)=f(x-y)-f(x)\).
The exact identity
\begin{equation}
 \mathcal C_\Lambda(f,g)
  =\int_{\mathbb T^2}K_\Lambda(y)\delta_yf\,\delta_yg\,dy
       -(P_\Lambda f-f)(P_\Lambda g-g)
 \label{eq:coupled-cutoff-covariance}
\end{equation}
holds whether or not the kernel is positive.  The localized kernel
bounds give, with periodic distance denoted by \(d\),
\[
 \int_{\mathbb T^2}|K_\Lambda(y)|d(y,0)^{1+\alpha}\,dy
       \le C_\alpha\Lambda^{-1-\alpha}.
\]
Translation and tail estimates give
\begin{align*}
 \|\delta_yf\|_{B^\alpha}
       &\le C d(y,0)\|f\|_{B^{1+\alpha}},&
 \|\delta_yg\|_{B^0}
       &\le C_\alpha d(y,0)^\alpha\|g\|_{B^\alpha},\\
 \|(I-P_\Lambda)f\|_{B^\alpha}
       &\le C\Lambda^{-1}\|f\|_{B^{1+\alpha}},&
 \|(I-P_\Lambda)g\|_{B^0}
       &\le C_\alpha\Lambda^{-\alpha}\|g\|_{B^\alpha}.
\end{align*}
Apply \eqref{ep:positive-multiplier} to the two products in
\eqref{eq:coupled-cutoff-covariance} to obtain
\eqref{eq:coupled-cutoff-product}.

Since \(P_\Lambda\) commutes with derivatives,
\[
 \partial_\ell\mathcal C_\Lambda(v_i,v_j)
  =\mathcal C_\Lambda(\partial_\ell v_i,v_j)
       +\mathcal C_\Lambda(v_i,\partial_\ell v_j).
\]
Use symmetry of \(\mathcal C_\Lambda\), \eqref{eq:coupled-cutoff-product}, and
\(\nabla v\in B^\alpha\).  Divergence form then proves the first
line of \eqref{eq:coupled-cutoff-transport}.  The second follows by
applying \eqref{eq:coupled-cutoff-product} to \(v_i,z\)
and using \(\Div:B^0\to B^{-1}\).
Finally, each component of \(Q\) is a linear combination of products
of a component of \(\nabla v\) and a component of \(z\).
Apply \eqref{eq:coupled-cutoff-product} with \(f=z\) and
\(g=\partial_\ell v_i\) to prove
\eqref{eq:coupled-cutoff-constitutive}.
\end{proof}

The product-cutoff estimate controls the spatial errors in
\(f_\Lambda,g_\Lambda\). Combining it with stress smoothing and the
reconstruction identities gives their total accumulated size.

\begin{proposition}[Mixed-norm reconstruction forcing]
\label{prop:coupled-reconstruction-forcing}
Let \(0<\alpha\le1\), put \(\beta=(1+\alpha)/2\), and use the
exponential reconstruction \eqref{eq:exponential-reconstruction}
and the forces \eqref{eq:coupled-forces} under
\eqref{eq:conditional-prefix-bound}.
Uniformly in the existing interval length \(J\) and in \(\Lambda\ge1\),
\begin{equation}
 \|f_\Lambda\|_{L^1(0,Jh;B^0)}
       +\|g_\Lambda\|_{L^1(0,Jh;B^{-1})}
 \le C_{T,R,\alpha}
       \bigl(h^\beta+h\Lambda^{1-\alpha}
                         +\Lambda^{-1-\alpha}\bigr).
 \label{eq:coupled-reconstruction-forcing}
\end{equation}
In particular, \(\Lambda\asymp h^{-1/2}\) makes the right-hand side
\(O(h^{(1+\alpha)/2})\).
\end{proposition}

\begin{proof}
Work on \([0,t_J]\), with \(\ell=t-t_n\) on \(I_n\).
All constants below are independent of \(J,h,\Lambda\).

\emph{Reconstruction estimates.}
The one-derivative bounds
\eqref{eq:coupled-integrated-stress} control the heat part of \(z\)
on all cells with \(n\ge1\). On the first cell,
\[
 \int_0^h\|E_\ell\tau^0\|_{B^{1+\alpha}}\,d\ell
       \le C(h+\sqrt h)\|\tau^0\|_{B^\alpha}.
\]
For the correction use
\(\|\zeta^n\|_{B^{1+\alpha}}
 \le\|\tau^{n+1}\|_{B^{1+\alpha}}+
       \|E_h\tau^n\|_{B^{1+\alpha}}\),
then the same nodal bounds and heat smoothing of \(\tau^0\).
Together they give
\begin{equation}
 \|z\|_{L^1(0,Jh;B^{1+\alpha})}\le C.
 \label{eq:coupled-reconstruction-high-stress}
\end{equation}
No \(B^{1+\alpha}\) assumption on \(\tau^0\) is needed.

Apply \eqref{eq:E-difference} with \(r=-1\) and \(\rho=1+\alpha\),
and use \eqref{eq:reconstruction-weak-increments}. Uniformly on each
time interval, this gives
\[
 \|z-\tau^n\|_{B^{-1}}
       +\|z-\tau^{n+1}\|_{B^{-1}}\le Ch^\beta.
\]
Using velocity recovery and the good-vorticity increment in \eqref{eq:reconstruction-weak-increments},
we obtain
\begin{equation}
 \sup_{t\in I_n}\|v(t)-u^n\|_{B^0}\le Ch^\beta.
 \label{eq:coupled-reconstruction-velocity-increment}
\end{equation}
For the stress differences in \(B^0\), apply
\eqref{eq:E-difference} with \(r=0\) and \(\rho=1+\alpha\) on
intervals with \(n\ge1\), then sum using
\eqref{eq:coupled-integrated-stress}. On the first interval, only
\(\tau^0\in B^\alpha\) is available, so we use \(\rho=\alpha\);
its time integral is at most \(Ch^{1+\alpha/2}\). The correction
contributes \(O(h)\) by \eqref{eq:coupled-integrated-correction}.
The same argument applies to the difference from either endpoint, since
\(z-\tau^{n+1}=(E_\ell-E_h)\tau^n+(\ell/h-1)\zeta^n\).
Consequently,
\begin{equation}
 \sum_{n=0}^{J-1}\int_{I_n}
   \bigl(\|z-\tau^n\|_{B^0}
            +\|z-\tau^{n+1}\|_{B^0}\bigr)\,dt
       \le Ch^\beta.
 \label{eq:coupled-reconstruction-stress-increments}
\end{equation}

\emph{Stress defects.}
The term containing two spatial derivatives of \(\zeta^n\) satisfies
\begin{equation}
 \sum_n\int_{I_n}
 \left\|P_\Lambda\frac{\ell}{h}(\Delta-a)\zeta^n\right\|_{B^{-1}}dt
 \le C\Lambda^{1-\alpha}
                  h\sum_n\|\zeta^n\|_{B^\alpha}
 \le Ch\Lambda^{1-\alpha}.
 \label{eq:coupled-stiff-defect}
\end{equation}
The term \((E_h-I)S^n\) has norm \(O(h^\beta)\) in \(B^{-1}\).
To estimate the constitutive part of \(S^n-S(v,z)\), write
\[
 Q(\nabla v,z)-Q(\nabla u^n,\tau^n)
  =Q(\nabla(v-u^n),z)
       +Q(\nabla u^n,z-\tau^n).
\]
The first term is controlled in \(B^{-1}\) by
\[
 C_\alpha\|v-u^n\|_{B^0}\|z\|_{B^{1+\alpha}}.
\]
Indeed, multiplication by \(B^{1+\alpha}\) is bounded on \(B^{-1}\);
the input indices of the remainder sum to \(\alpha>0\).
Its time integral is \(O(h^\beta)\) by \eqref{eq:coupled-reconstruction-high-stress}--\eqref{eq:coupled-reconstruction-velocity-increment}.
The second term has an \(O(h^\beta)\) integral even in \(B^0\), by
\eqref{ep:positive-multiplier}, \eqref{eq:coupled-reconstruction-stress-increments}, and the uniform bound on \(\nabla u^n\) in \(B^\alpha\).
The linear difference \(D(v-u^n)\) is controlled by \eqref{eq:coupled-reconstruction-velocity-increment}.

For stress transport, use the divergence of
\[
 (v-u^n)\otimes z+u^n\otimes(z-\tau^{n+1}).
\]
The product estimate \eqref{ep:positive-multiplier}, \eqref{eq:coupled-reconstruction-velocity-increment}--\eqref{eq:coupled-reconstruction-stress-increments}, and the uniform nodal and
reconstruction bounds give an \(O(h^\beta)\) integral in \(B^{-1}\).
Together these facts prove
\begin{equation}
 \|P_\Lambda d_\tau\|_{L^1B^{-1}}
       \le C\bigl(h^\beta+h\Lambda^{1-\alpha}\bigr).
 \label{eq:coupled-stress-defect}
\end{equation}

\emph{Good-vorticity defects.}
The positive-index algebra law and
\eqref{eq:coupled-stress-variation} control the additional term
\(J^n\) in \eqref{eq:source-delay}, even in \(B^\alpha\):
\[
 h\sum_{n<J}\|J^n\|_{B^\alpha}
 \le C_{\alpha,R}h\sum_{n<J}
                 \|\tau^{n+1}-\tau^n\|_{B^\alpha}
 \le C_{T,R,\alpha}h.
\]
In particular,
\begin{equation}
 h\sum_{n<J}\|J^n\|_{B^0}
 \le C_{T,R,\alpha}h.
 \label{eq:delay-accumulated-positive}
\end{equation}
Thus \(J^n\) contributes \(O(h)\) to the time integral
of \(P_\Lambda d_\Gamma\) in \(B^{-1}\), independently of \(\Lambda\).
The preceding linear and constitutive source estimates also control
the corresponding terms of \(G(u^n,\tau^n)-G(v,z)\), because
\(\mathcal R\) is bounded on \(B^{-1}\).
The damping term uses the \(B^{-1}\) stress difference above.
For the commutator difference, split it as
\[
 [\mathcal R,(u^n-v)\cdot\nabla]z
       +[\mathcal R,u^n\cdot\nabla](\tau^n-z).
\]
For divergence-free \(b\),
\[
 [\mathcal R,b\cdot\nabla]f
   =\mathcal R\Div(b\otimes f)
                -\Div(b\otimes\mathcal R f).
\]
The first difference is therefore controlled in \(B^{-1}\) by
\(C_\alpha\|u^n-v\|_{B^0}\|z\|_{B^\alpha}\), and the second by
\(C_\alpha\|u^n\|_{B^\alpha}\|\tau^n-z\|_{B^0}\).
Their integrals are \(O(h^\beta)\).

For the transport difference, use
\[
 v\cdot\nabla\gamma-u^n\cdot\nabla\Gamma^{n+1}
  =\Div\bigl((v-u^n)\gamma\bigr)
       +\Div\bigl(u^n(\gamma-\Gamma^{n+1})\bigr).
\]
The first term has an \(O(h^\beta)\) integral in \(B^{-1}\) by
\eqref{ep:positive-multiplier} and \eqref{eq:coupled-reconstruction-velocity-increment}.  The second uses
\(\|\gamma-\Gamma^{n+1}\|_{B^{\alpha-1}}\le Ch\) from \eqref{eq:reconstruction-weak-increments}
and the multiplier bound
\[
 \|u^n(\gamma-\Gamma^{n+1})\|_{B^{\alpha-1}}
 \le C_\alpha\|u^n\|_{B^{1+\alpha}}
                  \|\gamma-\Gamma^{n+1}\|_{B^{\alpha-1}}.
\]
The two input indices of its remainder sum to \(2\alpha>0\).
Since \(P_\Lambda:B^{\alpha-2}\to B^{-1}\) has norm at most
\(C_\alpha\Lambda^{1-\alpha}\), its filtered time integral is at
most \(Ch\Lambda^{1-\alpha}\).  Therefore
\begin{equation}
 \|P_\Lambda d_\Gamma\|_{L^1B^{-1}}
       \le C\bigl(h^\beta+h\Lambda^{1-\alpha}\bigr).
 \label{eq:coupled-gamma-defect}
\end{equation}
No cancellation between the two separately estimated defects is required.

\emph{Spatial cutoff errors.}
Apply \eqref{eq:coupled-cutoff-transport} to the momentum and stress transport cutoff errors.
The uniform reconstruction norms bound their time integrals in
\(B^0\) and \(B^{-1}\), respectively, by \(C\Lambda^{-1-\alpha}\).
For the constitutive cutoff error, \eqref{eq:coupled-cutoff-constitutive} and \eqref{eq:coupled-reconstruction-high-stress} give
\[
 \|P_\Lambda Q(\nabla v,z)
       -Q(\nabla P_\Lambda v,P_\Lambda z)\|_{L^1B^0}
       \le C\Lambda^{-1-\alpha}.
\]
This also bounds its \(L^1B^{-1}\) norm.
Finally use \eqref{eq:coupled-forces},
\(\mathcal K:B^{-1}\to B^0\), and boundedness of
\(\mathcal R,\mathbb P,P_\Lambda\) on the indicated Besov spaces.
Equations \eqref{eq:coupled-stress-defect}--%
\eqref{eq:coupled-gamma-defect} and the cutoff bounds prove
\eqref{eq:coupled-reconstruction-forcing}.
\end{proof}

The forcing estimate now provides the input for mixed stability.
Choosing \(\Lambda\asymp h^{-1/2}\) balances the cutoff and defect
terms; removing the cutoff and interpolating gives the nodal error
bounds below.

\begin{proposition}[Conditional convergence on an existing grid interval]
\label{prop:conditional-convergence}
Under \eqref{eq:conditional-prefix-bound}, for sufficiently small \(h\),
\begin{align}
 \max_{0\le n\le J}
  \bigl(\|e_u^n\|_{B^0}+\|e_\tau^n\|_{B^{-1}}\bigr)
 &\le C_{T,R,\alpha}h^{(1+\alpha)/2},
 \label{eq:conditional-mixed-rate}\\
 \max_{0\le n\le J}
  \bigl(\|e_\tau^n\|_{B^{-\alpha}}
                    +\|e_\Gamma^n\|_{B^{-\alpha}}\bigr)
 &\le C_{T,R,\alpha}h^\alpha,
 \label{eq:conditional-weak-rate}\\
 \max_{0\le n\le J}
  \bigl(\|e_u^n\|_{B^1}+\|e_\tau^n\|_{B^0}
                    +\|e_\Gamma^n\|_{B^0}\bigr)
 &\le C_{T,R,\alpha}h^{\alpha/2}.
 \label{eq:conditional-critical-rate}
\end{align}
The constants and the step restriction are independent of \(J\).
\end{proposition}

\begin{proof}
All comparisons are on \([0,t_J]\). The preceding forcing estimate
applies by the conditional numerical bound and discrete stress
smoothing. The filtered velocity \(v_\Lambda\) is uniformly bounded
in \(L^\infty(0,t_J;B^{1+\alpha})\), independently of \(\Lambda\).
For each fixed cutoff, the path is spatially smooth and piecewise
smooth in time. Thus
\[
 v_\Lambda-u\in C([0,t_J];B^0),\qquad
 z_\Lambda-\tau\in C([0,t_J];B^{-1})\cap L^1(0,t_J;B^1).
\]
These qualitative properties justify the mixed stability estimate.
The required exact norm \(\|\tau\|_{L^1_TB^1}\) is controlled by
\eqref{eq:endpoint-exact-class}.

The uniform reconstruction bounds in
\eqref{eq:reconstruction-weak-increments} give
\begin{equation}
 \|(I-P_\Lambda)v(t)\|_{B^0}
       +\|(I-P_\Lambda)z(t)\|_{B^{-1}}
 \le C_R\Lambda^{-1-\alpha},\qquad 0\le t\le t_J.
 \label{eq:coupled-mixed-tails}
\end{equation}
Exact initialization gives the same bound for the initial discrepancy
between the filtered path and the exact solution. Apply
Lemma~\ref{lem:coupled-mixed-continuous-stability} to
\eqref{eq:coupled-filtered-system}, and use
Proposition~\ref{prop:coupled-reconstruction-forcing}. Then
\begin{equation}
 \|v_\Lambda-u\|_{L^\infty(0,t_J;B^0)}
 +\|z_\Lambda-\tau\|_{L^\infty(0,t_J;B^{-1})}
 \le C_{T,R,\alpha}\left(
 \Lambda^{-1-\alpha}+h\Lambda^{1-\alpha}
                         +h^{(1+\alpha)/2}\right).
 \label{eq:coupled-filtered-error}
\end{equation}
Take \(\Lambda=h^{-1/2}\), remove the cutoff using
\eqref{eq:coupled-mixed-tails}, and evaluate at the nodes.
Node matching proves \eqref{eq:conditional-mixed-rate}.

The identity \(e_\Gamma^n=\curl e_u^n-\mathcal Re_\tau^n\) yields
\begin{equation}
 \max_{n\le J}
 \bigl(\|e_\tau^n\|_{B^{-1}}+\|e_\Gamma^n\|_{B^{-1}}\bigr)
 \le C_{T,R,\alpha}h^{(1+\alpha)/2}.
 \label{eq:coupled-weak-vorticity}
\end{equation}
Both errors are uniformly bounded in \(B^\alpha\) by the exact
assumptions and \eqref{eq:conditional-prefix-bound}. A single dyadic
interpolation therefore gives, for \(f\in\{\tau,\Gamma\}\) and
\(-1\le s\le\alpha\),
\begin{equation}
 \|e_f^n\|_{B^s}
 \le C_{\alpha,s}
   \|e_f^n\|_{B^{-1}}^{(\alpha-s)/(1+\alpha)}
   \|e_f^n\|_{B^\alpha}^{(1+s)/(1+\alpha)}
 \le C_{T,R,\alpha,s}h^{(\alpha-s)/2}.
 \label{eq:conditional-interpolation}
\end{equation}
At \(s=-1\) and \(s=\alpha\), use the weak error estimate and
the uniform bound directly. Taking \(s=-\alpha\) and \(s=0\)
proves the stress and good-vorticity assertions; velocity recovery
\eqref{eq:common-error-recovery-bound} gives the \(B^1\) velocity error.
\end{proof}

\subsection{Continuation and completion of the main theorem}
\label{sec:exact-interval-continuation}

We now remove the conditional bound. The proof below constructs each
new node using tame stability, and only then applies the conditional
error estimate to the enlarged existing interval.

\begin{proposition}[Continuation along a regular exact trajectory]
\label{prop:exact-interval-continuation}
Fix \(0<\alpha\le1\) and \(T>0\). Assume that the exact solution
satisfies \eqref{eq:endpoint-exact-class} on \([0,T]\), with the
divergence-free velocity and symmetric stress specified above.
Initialize the time-discrete scheme with the exact data and the
conserved mean \(m_0\). For all sufficiently small \(h=T/M\), the
scheme has a unique \(B^\alpha\) orbit through time \(T\), and
\begin{equation}
 \max_{0\le n\le M}
 \bigl(\|u^n\|_{B^{1+\alpha}}+
       \|\tau^n\|_{B^\alpha}+\|\Gamma^n\|_{B^\alpha}\bigr)
 \le C_T.
 \label{eq:exact-interval-discrete-high}
\end{equation}
All estimates in Theorem~\ref{thm:base-algebraic} hold on this interval.
The constants and the upper bound on \(h\) depend only on
\(T,\alpha,a,b\) and the stated exact-solution norms.
No numerical bound is assumed in this conclusion.
\end{proposition}

\begin{proof}
Use \(X_n=|m_0|+\|\tau^n\|_{B^\alpha}+\|\Gamma^n\|_{B^\alpha}\)
and put
\[
 Y_n=|m_0|+\|\tau^n\|_{B^0}+\|\Gamma^n\|_{B^0},\qquad
 L=1+\sup_{0\le t\le T}
 \bigl(|m_0|+\|\tau(t)\|_{B^0}+\|\Gamma(t)\|_{B^0}\bigr).
\]
Recovery gives \(\|u^n\|_{B^1}\le CY_n\) and
\(\|u^n\|_{B^{1+\alpha}}\le CX_n\).
The tame product and commutator estimates give
\[
 \|S^n\|_{B^\alpha}\le C(1+Y_n)X_n,\qquad
 \|G^n\|_{B^\alpha}
 \le C\bigl((1+Y_n+Y_{n+1})X_n
                  +Y_n\|\tau^{n+1}\|_{B^\alpha}\bigr).
\]
Apply \eqref{eq:transport-stability} first to stress, then to good
vorticity. For \(ChY_n\le1/2\), this yields
\begin{equation}
 X_{n+1}\le
 \exp\bigl(C_1h(1+Y_n+Y_{n+1})\bigr)X_n.
 \label{eq:critical-tame-high-recursion}
\end{equation}
This estimate will be used only after controlling both critical norms.

Choose
\[
 R=2(1+X_0)\exp\bigl(C_1(1+2L)T\bigr).
\]
Choose \(h\) small enough for the resolvents and all conditional
estimates with this fixed \(R\), and require
\begin{equation}
 Ch(1+R)^3\le R/4,\qquad
 C_{T,R,\alpha}h^{\alpha/2}\le\tfrac12.
 \label{eq:continuation-smallness}
\end{equation}
Suppose the already constructed nodes satisfy \(X_j<R/2\) and
\(Y_j\le L\) through \(n\). The one-step construction in the proof of
Proposition~\ref{prop:discrete-high} produces node \(n+1\) and gives \(X_{n+1}\le X_n+Ch(1+X_n)^3<3R/4\).
The enlarged interval therefore satisfies the conditional bound \(R\).
Only now apply \eqref{eq:conditional-critical-rate}; it gives
\[
 Y_{n+1}\le |m_0|+\|\tau(t_{n+1})\|_{B^0}
                  +\|\Gamma(t_{n+1})\|_{B^0}+\tfrac12
 \le L-\tfrac12.
\]
With both \(Y_n,Y_{n+1}\le L\), iterate
\eqref{eq:critical-tame-high-recursion} to obtain
\[
 X_{n+1}\le X_0e^{C_1(1+2L)t_{n+1}}<R/2.
\]
The induction starts with the exact initial data and reaches \(M\).
It proves \eqref{eq:exact-interval-discrete-high}. Applying
Proposition~\ref{prop:conditional-convergence} with \(J=M\) proves
all the asserted algebraic rates.
\end{proof}

\section{Rates Under Additional Besov Regularity}
\label{sec:endpoint-rate-family}

This section proves Theorem~\ref{thm:endpoint-rate-family} in three steps.
First, Lemma~\ref{lem:additional-numerical-regularity} gives uniform
$B^{1+\theta}$ bounds for the numerical stress and good vorticity along
the orbit constructed in Section~\ref{sec:quantitative-proofs}.
Second, the residual estimates and the error recurrence give an
$O(h)$ error in $B^{\theta-1}$.
Finally, interpolation between $B^{\theta-1}$ and $B^{1+\theta}$ gives
$h^{(1+\theta)/2}$ in $B^0$, and velocity recovery gives the same rate
in $B^1$.

\subsection{Residual identities and heat quadrature}
\label{sec:shared-residuals}

The numerical orbit exists through \(T\) by
Proposition~\ref{prop:exact-interval-continuation}, used with
\(\alpha=1\). Unless stated otherwise, the residual sums below run from
\(n=0\) to \(M-1\).
Define the exact one-step residuals by
\begin{align}
 R_\tau^n
 &=\tau(t_{n+1})+hu(t_n)\cdot\nabla\tau(t_{n+1})-E_h\tau(t_n)
   -hE_hS(u(t_n),\tau(t_n)),
   \label{eq:Rtau-definition}\\
 R_\Gamma^n
 &=\Gamma(t_{n+1})+hu(t_n)\cdot\nabla\Gamma(t_{n+1})
   -\Gamma(t_n)-hG(u(t_n),\tau(t_n)).
   \label{eq:Rgamma-definition}
\end{align}
Subtracting the numerical scheme gives
\begin{align}
 e_\tau^{n+1}+hu^n\cdot\nabla e_\tau^{n+1}
 ={}&E_he_\tau^n
   +hE_h\bigl(S(u(t_n),\tau(t_n))-S(u^n,\tau^n)\bigr)\notag\\
 &-he_u^n\cdot\nabla\tau(t_{n+1})+R_\tau^n,
   \label{eq:full-stress-error}\\
 e_\Gamma^{n+1}+hu^n\cdot\nabla e_\Gamma^{n+1}
 ={}&e_\Gamma^n
   +h\bigl(G(u(t_n),\tau(t_n))-G(u^n,\tau^n)\bigr)\notag\\
 &-he_u^n\cdot\nabla\Gamma(t_{n+1})+R_\Gamma^n-hJ^n.
   \label{eq:full-gamma-error}
\end{align}
The correction \(-hJ^n\) appears because the exact residuals use
the source evaluated at \((u(t_n),\tau(t_n))\).

The mild stress equation and the good-vorticity equation separate
changes in the solution over a time step from differences between heat
operators. With \(\ell=t_{n+1}-t\),
\begin{align}
 R_\tau^n={}&
 \int_{I_n}E_\ell\bigl(S(u(t),\tau(t))-S(u(t_n),\tau(t_n))\bigr)\,dt\notag\\
 &+\int_{I_n}(E_\ell-E_h)S(u(t_n),\tau(t_n))\,dt\notag\\
 &+\int_{I_n}\bigl(u(t_n)\cdot\nabla\tau(t_{n+1})
                   -u(t)\cdot\nabla\tau(t)\bigr)\,dt\notag\\
 &+\int_{I_n}(I-E_\ell)(u(t)\cdot\nabla\tau(t))\,dt,
 \label{eq:stress-residual-decomposition}\\
 R_\Gamma^n={}&
 \int_{I_n}\bigl(G(u(t),\tau(t))-G(u(t_n),\tau(t_n))\bigr)\,dt\notag\\
 &+\int_{I_n}\bigl(u(t_n)\cdot\nabla\Gamma(t_{n+1})
                   -u(t)\cdot\nabla\Gamma(t)\bigr)\,dt.
 \label{eq:gamma-residual-decomposition}
\end{align}
Thus only \(R_\tau^n\) contains heat-quadrature errors.
For \(s\in\mathbb R\), \(0<\rho\le2\), and \(0<h\le1\),
\eqref{eq:E-difference} and heat contraction give
\begin{align}
 \sum_n\int_0^h\|(E_\ell-E_h)F^n\|_{B^s}\,d\ell
 &\le C_{a,\rho}h^{\rho/2}
       h\sum_n\|F^n\|_{B^{s+\rho}},
 \label{eq:heat-quadrature-rate}\\
 \sum_n\int_{I_n}\|(I-E_{t_{n+1}-t})F(t)\|_{B^s}\,dt
 &\le C_{a,\rho}h^{\rho/2}\|F\|_{L^1_TB^{s+\rho}}.
 \label{eq:heat-identity-rate}
\end{align}
For the first bound use
\(E_\ell-E_h=E_\ell(I-E_{h-\ell})\); the second uses
\(t_{n+1}-t\le h\).

\subsection{Propagation of the error}

The bounds from Section~\ref{sec:quantitative-proofs}, applied with
\(\alpha=1\), already control the numerical stress and good vorticity
in \(B^1\). The final interpolation also requires uniform
\(B^{1+\theta}\) bounds. The following lemma propagates these higher
norms using the established \(B^1\) bounds.

\begin{lemma}[Propagation of additional numerical regularity]
\label{lem:additional-numerical-regularity}
Let \(0<\theta\le1\) and \(Jh\le T\). Suppose an already constructed
orbit of \eqref{eq:scheme-tau-good}--\eqref{eq:scheme-gamma} satisfies
\[
 \max_{0\le n\le J}
 \bigl(\|u^n\|_{B^2}+\|\tau^n\|_{B^1}
                       +\|\Gamma^n\|_{B^1}\bigr)\le R.
\]
If \(\tau^0,\Gamma^0\in B^{1+\theta}\), then, for all sufficiently
small \(h\), depending only on \(R,\theta,a,b\),
\begin{equation}
 \max_{0\le n\le J}
 \bigl(\|u^n\|_{B^{2+\theta}}+\|\tau^n\|_{B^{1+\theta}}
                 +\|\Gamma^n\|_{B^{1+\theta}}\bigr)
 \le C_{T,R,\theta,a,b}
       \bigl(|m_0|+\|\tau^0\|_{B^{1+\theta}}
                    +\|\Gamma^0\|_{B^{1+\theta}}\bigr).
 \label{eq:additional-numerical-regularity}
\end{equation}
\end{lemma}

\begin{proof}
Put \(s=1+\theta\in(1,2]\) and
\[
 K_n=|m_0|+\|\tau^n\|_{B^s}+\|\Gamma^n\|_{B^s}.
\]
Recovery gives \(\|u^n\|_{B^{s+1}}\le CK_n\), whereas the assumed
velocity bound gives \(\|u^n\|_{B^s}\le C R\), since \(s\le2\).
Consequently, Lemma~\ref{lem:transport-stability}, applied at index
\(s\), bounds the inverse of \(I+hu^n\cdot\nabla\) on \(B^s\) by
\(1+Ch\), with \(C\) independent of \(K_n\) and \(h\).

The tame product law used in \eqref{eq:source-Q-detail} and the
commutator estimate \eqref{eq:critical-tame-R} apply at \(s\le2\).
With the assumed lower-norm bounds, they give
\(\|S^n\|_{B^s}\le CK_n\). The stress solve and heat contraction imply
\begin{equation}
 \|\tau^{n+1}\|_{B^s}
 \le(1+Ch)\bigl(\|\tau^n\|_{B^s}+ChK_n\bigr).
 \label{eq:additional-stress-recursion}
\end{equation}
For the good-vorticity source in \eqref{eq:Gn}, the constitutive term
uses the newly computed stress. Its \(B^1\) norm is controlled
by the assumed orbit bound, so
\[
 \|Q(\nabla u^n,\tau^{n+1})\|_{B^s}
 \le C\bigl(\|\tau^{n+1}\|_{B^s}+K_n\bigr).
\]
The remaining source terms depend only on the fields at time \(t_n\).
Hence
\begin{equation}
 \|\Gamma^{n+1}\|_{B^s}
 \le(1+Ch)\bigl(\|\Gamma^n\|_{B^s}
                     +Ch(K_n+\|\tau^{n+1}\|_{B^s})\bigr).
 \label{eq:additional-gamma-recursion}
\end{equation}
Substituting \eqref{eq:additional-stress-recursion} into
\eqref{eq:additional-gamma-recursion}, and increasing \(C\) for
\(h\le1\), yields \(K_{n+1}\le(1+Ch)K_n\).
Discrete Gronwall and recovery prove
\eqref{eq:additional-numerical-regularity}.

To justify the regularity induction, first construct \(\tau^{n+1}\)
using the inverse on \(B^s\). Uniqueness in \(B^1\) identifies this
solution with the existing iterate, so its known \(B^1\) bound can be
used in the good-vorticity source.
The subsequent good-vorticity solve is identified in the same way.
Thus both higher norms are constructed at each step, with a step
restriction depending only on the given lower-order bounds.
\end{proof}

The higher numerical bound completes the estimates needed for
interpolation. We now use the residual identities to obtain an
\(O(h)\) error in \(B^{\theta-1}\), and then recover the asserted
critical-norm rate.

\begin{proof}[Proof of Theorem~\ref{thm:endpoint-rate-family}]
Set \(r=\theta-1\in(-1,0]\).  All sums over \(n\) below run from
\(0\) to \(M-1\), and all constants may depend on the fixed \(\theta\).

\emph{Step 1: uniform bounds and estimates in weaker norms.}
The exact assumptions imply the base class with \(\alpha=1\).
Proposition~\ref{prop:exact-interval-continuation} therefore gives
\begin{equation}
 \max_{0\le n\le M}
 \bigl(\|u^n\|_{B^2}+\|\tau^n\|_{B^1}+\|\Gamma^n\|_{B^1}\bigr)
 \le C_T.                                             \label{eq:endpoint-rate-numerical-high}
\end{equation}
The exact initial data belong to \(B^{1+\theta}\) in stress and
good vorticity. Applying Lemma~\ref{lem:additional-numerical-regularity}
to this already constructed orbit also gives
\begin{equation}
 \max_{0\le n\le M}
 \bigl(\|u^n\|_{B^{2+\theta}}+\|\tau^n\|_{B^{1+\theta}}
                   +\|\Gamma^n\|_{B^{1+\theta}}\bigr)
 \le C_{T,\theta}.
 \label{eq:endpoint-rate-full-numerical-high}
\end{equation}
For this range of \(r\), the product and commutator estimates needed are
\begin{align}
 \|fg\|_{B^r}&\le C_\theta\|f\|_{B^1}\|g\|_{B^r},\notag\\
 \|fw\|_{B^\theta}&\le C_\theta\|f\|_{B^1}\|w\|_{B^\theta},\notag\\
 \|[\mathcal R,v\cdot\nabla]g\|_{B^r}
 &\le C_\theta\|v\|_{B^1}\|g\|_{B^r},
 \qquad \Div v=0.                                    \label{eq:lifted-weak-calculus}
\end{align}
For the first product, \(1+r=\theta>0\) makes the high--high remainder
summable; at \(r=0\), use the mixed \(B^1\cdot B^0\to B^0\) law.
The second is the positive-index product estimate for \(0<\theta\le1\).
For the commutator, the dyadic argument of
Lemma~\ref{lem:gamma-source} works also at \(-1<r\le0\).
The low--high term uses \(\|\nabla v\|_\infty\); the high--high
sum is geometric because \(r+1>0\), and the high--low sum is
\[
 C\sum_k 2^k\|\Delta_kv\|_\infty
   \sum_{\ell\le k-2}2^{-(1-r)(k-\ell)}
                   2^{r\ell}\|\Delta_\ell g\|_\infty
 \le C_\theta\|v\|_{B^1}\|g\|_{B^r}.
\]

Bilinear splitting, \eqref{eq:lifted-weak-calculus}, and the divergence
form of the commutator with the velocity difference yield
\begin{align}
 &\|S(v,z)-S(w,y)\|_{B^r}
   +\|G(v,z)-G(w,y)\|_{B^r}\notag\\
 &\qquad\le C_T\bigl(\|v-w\|_{B^\theta}
                         +\|z-y\|_{B^r}\bigr),        \label{eq:endpoint-rate-source-lipschitz}
\end{align}
when the velocities are bounded in \(B^2\) and the stresses in \(B^1\).
For example,
\[
 [\mathcal R,(v-w)\cdot\nabla]z
 =\mathcal R\Div((v-w)\otimes z)
       -\Div((v-w)\otimes\mathcal Rz),
\]
and both products belong to \(B^\theta\).
Likewise, for \(F=\tau(t_{n+1})\) or \(\Gamma(t_{n+1})\),
\[
 \|e_u^n\cdot\nabla F\|_{B^r}
 \lesssim\|e_u^nF\|_{B^\theta}
 \le C_T\|e_u^n\|_{B^\theta}.
\]
Recovery bounds this last norm by
\(\|e_\tau^n\|_{B^r}+\|e_\Gamma^n\|_{B^r}\).

\emph{Step 2: accumulated consistency.}
The momentum and stress equations, the exact bounds with
\(\alpha=1\), and the good-vorticity equation give
\begin{equation}
 \partial_tu,\partial_t\tau\in L^1_TB^1,\qquad
 \partial_t\Gamma\in L^\infty_TB^\theta.              \label{eq:endpoint-rate-time-derivatives}
\end{equation}
For the last assertion, the definition of \(\Gamma\) gives
\(\Gamma\in C_tB^{1+\theta}\), so \(u\cdot\nabla\Gamma\in
L^\infty_TB^\theta\); the good-vorticity source is bounded in \(B^1\).
For \(u\), the forcing \(\Div\tau\) is integrable in \(B^1\);
the other momentum terms are bounded there. The stress assertion is
\eqref{eq:tau-time-derivative}. Consequently, absolute continuity
and Fubini's theorem give, for \(f=u,\tau\),
\begin{equation}
 \sum_n\int_{I_n}
 \bigl(\|f(t)-f(t_n)\|_{B^1}
       +\|f(t)-f(t_{n+1})\|_{B^1}\bigr)\,dt
 \le 2h\|\partial_tf\|_{L^1_TB^1},
 \label{eq:base-summed-increments}
\end{equation}
while
\(\|\Gamma(t)-\Gamma(t')\|_{B^\theta}\le C_T|t-t'|\).
The stated time-integrated bound on \(\partial_t\tau\) in \(B^1\)
is sufficient for these estimates.

Apply these bounds to
\eqref{eq:stress-residual-decomposition}--\eqref{eq:gamma-residual-decomposition}.
The time integrals of the source differences sum to \(O(h)\) by
\eqref{eq:endpoint-rate-source-lipschitz}, since \(\theta\le1\)
and \(r\le0\). The frozen source satisfies
\[
 \|S(u(t_n),\tau(t_n))\|_{B^{1+\theta}}\le C_T.
\]
Since \(r+2=1+\theta\), applying \eqref{eq:heat-quadrature-rate} with
\(s=r\) and \(\rho=2\) bounds the summed heat-quadrature contribution
by \(C_Th\).

For \(F=\tau\) or \(\Gamma\), divergence form and
\eqref{eq:lifted-weak-calculus} give
\begin{align*}
 &\|u(t_n)\cdot\nabla F(t_{n+1})
                   -u(t)\cdot\nabla F(t)\|_{B^r}\\
 &\qquad\le C_T\bigl(\|u(t_n)-u(t)\|_{B^\theta}
                 +\|F(t_{n+1})-F(t)\|_{B^\theta}\bigr).
\end{align*}
The estimates for increments of \(u,\tau\) within each time step,
together with the Lipschitz bound for \(\Gamma\), therefore give an
\(O(h)\) bound for the sum of both transport residuals. The remaining stress heat increment is
also \(O(h)\): by \eqref{eq:endpoint-rate-class},
\[
 u\cdot\nabla\tau\in L^1_TB^{2+\theta}
                  \hookrightarrow L^1_TB^{r+2},
\]
so \eqref{eq:heat-identity-rate} applies with \(s=r\) and \(\rho=2\).
Combining these source, transport, and heat bounds yields
\begin{equation}
 \sum_n\bigl(\|R_\tau^n\|_{B^r}
                      +\|R_\Gamma^n\|_{B^r}\bigr)\le C_{T,\theta}h.
 \label{eq:lifted-accumulated-residual}
\end{equation}

\emph{Step 3: propagation and interpolation.}
Set \(\mathcal E_n=\|e_\tau^n\|_{B^r}+\|e_\Gamma^n\|_{B^r}\).
Apply the estimates from Step 1 to the stress and good-vorticity error equations
\eqref{eq:full-stress-error}--\eqref{eq:full-gamma-error}.
Lemma~\ref{lem:transport-stability} applies at \(r=\theta-1\in(-1,0]\), by
\eqref{eq:endpoint-rate-numerical-high}.  Together with heat contraction,
it gives
\begin{equation}
 \mathcal E_{n+1}\le(1+C_{T,\theta}h)\mathcal E_n
       +C_{T,\theta}\bigl(\|R_\tau^n\|_{B^r}
                           +\|R_\Gamma^n\|_{B^r}
                           +h\|J^n\|_{B^r}\bigr).
 \label{eq:endpoint-rate-recursion}
\end{equation}
Since \(r\le0\), applying \eqref{eq:delay-accumulated-positive} with
\(\alpha=1\) bounds the accumulated source correction by
\(C_{T,\theta}h\).

Since \(\mathcal E_0=0\), discrete Gronwall and
\eqref{eq:lifted-accumulated-residual} imply
\(\max_n\mathcal E_n\le C_{T,\theta}h\).
For \(0<\theta<1\), dyadic interpolation between
\(B^{\theta-1}\) and \(B^{1+\theta}\) yields
\[
 \|e_f^n\|_{B^0}
 \lesssim_\theta
 \|e_f^n\|_{B^{\theta-1}}^{(1+\theta)/2}
 \|e_f^n\|_{B^{1+\theta}}^{(1-\theta)/2}
 \le C_{T,\theta}h^{(1+\theta)/2},\qquad f=\tau,\Gamma.
\]
The \(B^{1+\theta}\) error norm is uniformly bounded by the assumed
exact-solution regularity and \eqref{eq:endpoint-rate-full-numerical-high}. At \(\theta=1\),
\(r=0\), so the weak estimate already gives first-order convergence.
Velocity recovery proves \eqref{eq:endpoint-rate-family} with exponent
\((1+\theta)/2\).
\end{proof}

\section{Constitutive-Source Obstruction and Error Lower Bounds}
\label{sec:endpoint-sharpness}

The examples in this section explain why the constitutive source uses
the new stress and test whether the convergence exponents can be improved.
The first example shows the failure of a uniform bound for the source
evaluated at the old stress. The advected shears then give error lower
bounds for the full scheme.

\subsection{Obstruction for the old-stress constitutive source}

Uniform bounds on the input fields in the critical spaces do not give
a uniform \(B^0\) bound for the constitutive source evaluated at the old
stress. The following explicit family shows this failure and how the
heat step controls the source formed with the new stress.

\begin{proposition}[An obstruction for the old-stress constitutive source]
\label{prop:critical-source-obstruction}
Let \(a=0\), \(b=1/2\), and \(K=2^\nu\ge64\). Set \(L=K/4\) and
\[
 F_L(x)=\frac4\pi\sum_{\substack{1\le\ell\le L\\\ell\ {\rm odd}}}
       \left(1-\frac{\ell}{L+1}\right)\frac{\sin(\ell x)}{\ell},
 \qquad g_K(x)=\cos(Kx)F_L(x).
\]
Define the following smooth initial data:
\[
 u_K^0=\left(0,\frac{\sin(Kx_1)}K\right),\qquad
 \tau_K^0=\operatorname{diag}(g_K(x_1),-g_K(x_1)).
\]
Let \(\tau_{K,h}^1\) be the first stress update in
\eqref{eq:scheme-tau-good}, for \(h>0\). There are constants
\(c,C>0\), independent of \(K,h\), such that
\begin{align}
 \|u_K^0\|_{B^1}+\|\tau_K^0\|_{B^0}
       +\|\Gamma_K^0\|_{B^0}&\le C,
 \label{eq:source-obstruction-input}\\
 c\log K\le
 \|\mathcal RQ(\nabla u_K^0,\tau_K^0)\|_{B^0}
       &\le C(1+\log K),
 \label{eq:source-obstruction-old}\\
 \|\mathcal RQ(\nabla u_K^0,\tau_{K,h}^1)\|_{B^0}
       &\le C(1+\log K)e^{-9hK^2/16}.
 \label{eq:source-obstruction-new}
\end{align}
Thus the old-stress source map has no uniform \(B^0\) bound on
bounded subsets of \(B^1\times B^0\), even for smooth data.
For every fixed \(h>0\), the new-stress sources in this family tend
to zero as \(K\to\infty\).
\end{proposition}

\begin{proof}
We may use any fixed equivalent smooth Littlewood--Paley decomposition;
the equivalence constants are independent of \(K,h\).
Choose a radial, nonincreasing cutoff \(\chi\), equal to one for
\(|\xi|\le3/4\), zero for \(|\xi|\ge4/3\), and with \(\chi(1)<1\).
Its annular multiplier is
\(\varphi(\xi)=\chi(\xi/2)-\chi(\xi)\ge0\).
On functions of \(x_1\) alone, write \(\varphi(r)=\varphi(r,0)\).
Then
\[
 \operatorname{supp}\varphi\subset\{3/4\le |\xi|\le8/3\},
 \qquad \varphi(r)\ge d_0:=1-\chi(1)>0
       \quad(1\le r\le3/2).
\]

The function \(F_L\) is the Fej\'er mean of the bounded function
\(\operatorname{sgn}(\sin x)\). Since the Fej\'er kernel is nonnegative
and has integral one, \(\|F_L\|_\infty\le1\). The support of \(g_K\) lies in
\(3K/4\le|k_1|\le5K/4\), and hence intersects only a fixed
number of dyadic annuli. Uniform localized kernel bounds give
\(\|g_K\|_{B^0}\le C\|g_K\|_\infty\le C\).
The single velocity frequency similarly gives \(\|u_K^0\|_{B^1}\le C\).
Moreover, \(\mathcal R\tau_K^0=0\) and
\(\Gamma_K^0=\cos(Kx_1)\), proving \eqref{eq:source-obstruction-input}.

Put \(f_K=\cos(Kx_1)\). For each fixed \(x_1\), taking Fourier
coefficients in \(x_2\) gives the transport inverse denominators
\(1+ihk_2(u_K^0)_2(x_1)\), which never vanish. Since the right-hand
side of the first stress step depends only on \(x_1\), the inverse
leaves it unchanged. Direct substitution in \(Q\) gives
\[
 \tau_{K,h}^1=
 \begin{pmatrix}
 E_hg_K & hE_h(f_K/2+f_Kg_K)\\
 hE_h(f_K/2+f_Kg_K)&-E_hg_K
 \end{pmatrix}.
\]
For a symmetric tensor depending only on \(x_1\),
\(\mathcal RA=A_{12}-\langle A_{12}\rangle\).
Both constitutive-source expressions below are odd in \(x_1\), so
their spatial means vanish. Another direct computation yields
\begin{align}
 q_K^{\rm old}:=\mathcal RQ(\nabla u_K^0,\tau_K^0)
       &=-\tfrac12F_L-\tfrac12\cos(2Kx_1)F_L,\notag\\
 q_{K,h}^{\rm new}:=\mathcal RQ(\nabla u_K^0,\tau_{K,h}^1)
       &=-f_KE_hg_K.
 \label{eq:source-obstruction-identities}
\end{align}

For \(2\le j\le \nu-4\), evaluate \(\Delta_jF_L\) at
\(x_j=(\pi/4)2^{-j}\). Every nonzero summand has
\(0<\ell x_j\le2\pi/3\), so all summands are nonnegative.
Restricting to odd integers
\(2^j\le\ell\le(3/2)2^j\), of which there are \(2^{j-2}\), gives
\[
 (\Delta_jF_L)(x_j)
 \ge\frac4\pi\,2^{j-2}\,d_0\,\frac12\,
              \frac{2}{3\,2^j}\,\frac1{\sqrt2}
 =\frac{d_0}{3\pi\sqrt2}.
\]
Here \(1-\ell/(L+1)\ge1/2\) and
\(\sin(\ell x_j)\ge1/\sqrt2\) on the selected integers.
The shifted term \(\cos(2Kx_1)F_L\) has frequencies of magnitude
at least \(7K/4\), and therefore has zero \(\Delta_j\) block
for all these \(j\). Consequently each of the \(\nu-5\) blocks
contributes at least \(d_0/(6\pi\sqrt2)\) to
\(\|q_K^{\rm old}\|_{B^0}\). Since \(\nu-5\ge \nu/6\) for
\(\nu\ge6\), this proves the lower bound in
\eqref{eq:source-obstruction-old}.

For the upper bounds, use normalized Fourier coefficients, so that
\(q(x)=\sum_k\widehat q(k)e^{ik\cdot x}\).
The nonnegative dyadic partition gives
\(\|q\|_{B^0}\le\sum_k|\widehat q(k)|\), and
\[
 \sum_\ell|\widehat F_L(\ell)|
 =\frac4\pi\sum_{\substack{1\le\ell\le L\\\ell\ {\rm odd}}}
       \frac{1-\ell/(L+1)}{\ell}
 \le C(1+\log K).
\]
The Fourier coefficients of \(f_K\) have \(\ell^1\) norm one, so
multiplication by \(f_K\) does not increase the \(\ell^1\) norm of
the Fourier coefficients. This proves the old-source upper bound. Since all frequencies
of \(g_K\) have magnitude at least \(3K/4\), the heat multiplier
also gives
\[
 \sum_k|\widehat q_{K,h}^{\rm new}(k)|
 \le\sum_k|\widehat{E_hg_K}(k)|
 \le e^{-9hK^2/16}\sum_k|\widehat g_K(k)|
 \le C(1+\log K)e^{-9hK^2/16}.
\]
This proves \eqref{eq:source-obstruction-new} and the assertions.
\end{proof}

This proposition concerns the constitutive source on a bounded family
of critical data. It does not assert nonconvergence of a time-stepping
scheme that evaluates this source at the old stress.

\subsection{Error lower bounds for advected shears}

A fixed advected shear with constant stress reduces the scheme to
implicit transport. At critical regularity it gives a logarithmic
error lower bound; with \(\alpha>0\) additional Besov derivatives, it
rules out every algebraic exponent larger than \(\alpha/2\).

\begin{proposition}[Logarithmic lower bound for the full time-discrete scheme]
\label{prop:endpt-lower-bound}
Let \(\gamma>0\), \(a=b=0\), \(T\in(0,2\pi)\), and take the fixed,
\(h\)-independent initial data
\begin{gather*}
 C=\begin{pmatrix}0&0\\0&1\end{pmatrix},\qquad
 \tau_0=C,\qquad u_0=(1,W(x_1)),\\
 W(x_1)=\sum_{m\ge1}(1+m)^{-\gamma-1}2^{-m}\sin(2^mx_1).
\end{gather*}
They give the exact solution
\[
 u(t,x)=(1,W(x_1-t)),\qquad \tau(t,x)=C,\qquad
 \Gamma(t,x)=W'(x_1-t),
\]
with \(u\in C([0,T];B^1)\),
\(\Gamma\in C([0,T];B^0)\), smooth constant stress, and
\(m_0=\langle u(t)\rangle=(1,0)\).
Initialize the spatially untruncated scheme
\eqref{eq:recover-discrete}--\eqref{eq:scheme-gamma} exactly, with this
stored mean. Each transport equation has a unique bounded solution for
every \(h>0\), and \(\tau^n=C\) at every step.  If \(h=T/M\le\frac12\),
\(M\ge2\), then
\begin{equation}
 \|\Gamma(T)-\Gamma^M\|_{L^\infty(\mathbb T^2)}
 \ge c_\gamma\bigl(1+\log_2\tfrac1h\bigr)^{-\gamma}.
 \label{eq:endpt-lower-bound}
\end{equation}
The same lower bound holds in \(B^0\), and
\[
 \|u(T)-u^M\|_{B^1}
 \ge c'_\gamma\bigl(1+\log_2\tfrac1h\bigr)^{-\gamma},\qquad
 \|\tau(T)-\tau^M\|_{B^0}=0,
\]
where \(c_\gamma,c'_\gamma>0\) are independent of \(h\).
Thus even for this fixed critical datum the full scheme has no positive
algebraic rate in the original velocity norm.
\end{proposition}

\begin{proof}
Write \(a_m=(1+m)^{-\gamma-1}\).  Since \(\sum_ma_m<\infty\),
\(W'=\Gamma_0=\sum_ma_m\cos(2^mx_1)\in B^0\) and
\(W\in B^1\cap C^1\), both with zero mean.  Their finite
Fourier sums have tails bounded in these respective Besov norms by a
constant times \(\sum_{m>L}a_m\); continuity of the translated finite sums
proves the asserted time continuity.  The constant stress belongs to
\(C([0,T];B^0)\cap L^1(0,T;B^2)\).
For any shear \(v=(1,V(x_1))\), put \(d=V'/2\).  At \(b=0\),
\[
 \Omega(v)=\begin{pmatrix}0&-d\\d&0\end{pmatrix},\qquad
 Q(\nabla v,C)=C\Omega(v)-\Omega(v)C
 =\begin{pmatrix}0&d\\d&0\end{pmatrix}=D(v).
\]
Moreover \(\mathcal RC=0\) and
\([\mathcal R,v\cdot\nabla]C=0\), so \(G(v,C)=0\).
Consequently the constant stress solves \eqref{eq:tau}, while
\(\partial_tu+u\cdot\nabla u=0\) and \(\Div C=0\) verify
\eqref{eq:u} with constant pressure.  The recovery convention gives
\(\mathcal Kq=(0,\partial_1^{-1}q)\) for zero-mean \(q=q(x_1)\),
where the primitive has zero mean.

Suppose inductively that \(\tau^n=C\) and \(\Gamma^n\in B^0\)
has zero mean and depends only on \(x_1\). Then
\(u^n=(1,\partial_1^{-1}\Gamma^n)\), and the preceding cancellations
give \(S^n=0\), \(E_hC=C\). For every \(h>0\), the characteristic
formula defines the unique bounded inverse of \(I+hu^n\cdot\nabla\).
For a right-hand side depending only on \(x_1\), it reduces to
\[
 (I+hu^n\cdot\nabla)^{-1}g(x_1)
   =\int_0^\infty e^{-s}g(x_1-hs)\,ds.
\]
To prove uniqueness among bounded weak solutions, note that the volume-preserving
\(C^1\) change of variables
\(y_1=x_1\), \(y_2=x_2-\partial_1^{-2}\Gamma^n(x_1)\)
transforms \(u^n\cdot\nabla\) into \(\partial_{y_1}\), also in
the weak formulation. The periodic primitive has zero mean, and the
transformed resolvent has nonvanishing Fourier multipliers
\(1+ihk_1\), proving uniqueness.
First the stress solve gives \(\tau^{n+1}=C\), hence \(G^n=0\).
The good-vorticity solve is therefore
\[
 \Gamma^{n+1}+h\partial_1\Gamma^{n+1}=\Gamma^n,\qquad
 u^{n+1}=(1,\partial_1^{-1}\Gamma^{n+1}).
\]
Translations preserve the mean and contract every dyadic block, so the
integral formula preserves \(B^0\). This completes the induction.

Set \(\lambda_m=(1+ih2^m)^{-1}\).  Fourier diagonalization gives
\[
 \Gamma^n(x)=\sum_{m\ge1}a_m
       \operatorname{Re}\!\left(\lambda_m^n e^{i2^mx_1}\right).
\]
Absolute uniform convergence permits evaluation at \(x_1=T\), where
\begin{align*}
 \|\Gamma(T)-\Gamma^M\|_{L^\infty}
 &\ge\sum_ma_m\left[1-\operatorname{Re}
                     (\lambda_m^M e^{i2^mT})\right]\\
 &\ge\sum_ma_m(1-|\lambda_m|^M).
\end{align*}
If \(2^m\ge h^{-1}\), then
\(|\lambda_m|^M=(1+h^24^m)^{-M/2}\le2^{-M/2}\le\frac12\).  Therefore
\[
 \|\Gamma(T)-\Gamma^M\|_{L^\infty}
 \ge\frac12\sum_{m\ge\log_2(1/h)}(1+m)^{-\gamma-1}
 \ge c_\gamma(1+\log_2(1/h))^{-\gamma}.
\]
The \(B^0\) lower bound follows from its embedding into
\(L^\infty\).  Since the stress error vanishes and the means agree,
\(e_u=u(T)-u^M\) and \(e_\Gamma=\Gamma(T)-\Gamma^M\) satisfy
\[
 e_u=\mathcal K e_\Gamma,\qquad \curl e_u=e_\Gamma,\qquad
 \|e_\Gamma\|_{B^0}\lesssim\|e_u\|_{B^1}.
\]
The last inequality follows from boundedness of differentiation from
\(B^1\) to \(B^0\) and proves the velocity lower bound.

\end{proof}

For the logarithmic and algebraic upper bounds below, write
\(P_h=(I+h\partial_1)^{-1}\) and
\(\mathsf T_hf=f(\,\cdot-he_1)\), \(e_1=(1,0)\).
Taylor's formula in the integral representation of \(P_h\) gives
\(\|(P_h-\mathsf T_h)f\|_\infty
\le h^2\|\partial_1^2f\|_\infty/2\) for smooth \(f\), since
\(\int_0^\infty e^{-s}(s-1)\,ds=0\) and
\(\int_0^\infty e^{-s}(s-1)^2\,ds=1\).
Both operators contract \(L^\infty\) and commute with dyadic blocks.
Telescoping and Bernstein's inequality give, for any
\(q_0=q_0(x_1)\in B^0\), \(j\ge-1\), and \(nh\le T\),
\begin{equation}
 \|\Delta_j(P_h^nq_0-\mathsf T_h^nq_0)\|_\infty
 \le C\min\{1,Th\,2^{2j}\}\|\Delta_jq_0\|_\infty.
 \label{eq:shear-dyadic-error}
\end{equation}

\begin{remark}[Matching logarithmic rate for the fixed shear]
\label{rem:critical-shear-matching}
For the shear in Proposition~\ref{prop:endpt-lower-bound}, put
\(a_m=(1+m)^{-\gamma-1}\) and
\(L=\lfloor\tfrac12\log_2(1/h)\rfloor\), with \(0<h\le1/4\).
Applying \eqref{eq:shear-dyadic-error} to the lacunary modes and
recovering velocity gives the matching upper bound
\[
 \begin{aligned}
 \|\Gamma(T)-\Gamma^M\|_{B^0}+\|u(T)-u^M\|_{B^1}
 &\le C_T\left(h\sum_{m\le L}4^ma_m+\sum_{m>L}a_m\right)\\
 &\le C_{T,\gamma}\bigl((1+L)^{-\gamma-1}+(1+L)^{-\gamma}\bigr)
 \le C_{T,\gamma}(1+\log_2(1/h))^{-\gamma}.
 \end{aligned}
\]
Together with the lower bound, this proves the exact logarithmic order for this
fixed shear family; it does not assert a universal rate for critical data.
\end{remark}

The logarithmic example uses an initial shear profile in \(B^1\) that
belongs to no \(B^{1+\varepsilon}\) with \(\varepsilon>0\).
We now impose a \(B^{1+\alpha}\) bound on the shear profile. The same
transport recurrence gives the upper rate \(h^{\alpha/2}\), and a fixed
lacunary profile excludes every larger exponent.

\begin{proposition}[Optimal algebraic exponent for advected shears]
\label{prop:positive-regularity-lower-bound}
Fix \(0<\alpha\le2\), \(T>0\), and \(a=b=0\).  Let
\(W=W(x_1)\in B^{1+\alpha}(\mathbb T^2)\) have zero mean and set
\[
 \tau_0=\begin{pmatrix}0&0\\0&1\end{pmatrix},\qquad
 u_0=(1,W(x_1)).
\]
The exact solution \(u(t)=(1,W(x_1-t))\), \(\tau(t)=\tau_0\)
satisfies \eqref{eq:endpoint-exact-class} for \(0<\alpha\le1\), and
\eqref{eq:endpoint-rate-class} with \(\theta=\alpha-1\) for
\(1<\alpha\le2\). Initialize the spatially untruncated scheme
exactly and retain the mean \(m_0=(1,0)\).
For every \(h=T/M\le1\), it satisfies \(\tau^n=\tau_0\) and
\begin{equation}
 \max_{0\le n\le M}
 \left(\|u(t_n)-u^n\|_{B^1}
       +\|\Gamma(t_n)-\Gamma^n\|_{B^0}\right)
 \le C_{T,\alpha}h^{\alpha/2}\|W\|_{B^{1+\alpha}}.
 \label{eq:positive-regularity-shear-upper}
\end{equation}
For the fixed, \(h\)-independent choice
\[
 W=W_\alpha,\qquad
 W_\alpha(x_1)=\sum_{m\ge1}
       \frac{2^{-m(1+\alpha)}}{(1+m)^2}\sin(2^mx_1),
\]
the exponent \(\alpha/2\) cannot be increased.
For \(M_m=4^m\) and \(h_m=T/M_m\), its errors at time \(T\) satisfy,
for all sufficiently large \(m\),
\begin{equation}
 \|u(T)-u^{M_m}\|_{B^1}
 \ge c_{T,\alpha}\frac{h_m^{\alpha/2}}{(1+m)^2}
 \asymp\frac{h_m^{\alpha/2}}{(1+|\log h_m|)^2}.
 \label{eq:positive-regularity-lower-bound}
\end{equation}
The same lower bound holds for the \(B^0\) good-vorticity error,
whereas the stress error is zero.  For this fixed datum neither error
is \(O(h^r)\) for any \(r>\alpha/2\).
\end{proposition}

\begin{proof}
Translations are continuous on \(B^{1+\alpha}\), as follows by
approximation with finite dyadic sums.  Together with the constant
stress, this gives the asserted exact-solution class in both ranges
of \(\alpha\).
The cancellations and unique transport solves in the preceding proof
apply to any such \(W\), for every \(T>0\).  In particular,
\(\tau^n=\tau_0\), \(\Gamma_0=W'\), and
\(\Gamma^{n+1}=(I+h\partial_1)^{-1}\Gamma^n\).
Use \eqref{eq:shear-dyadic-error} and
\(\min\{1,x\}\le x^{\alpha/2}\), valid for \(0<\alpha\le2\).
Summation gives
\(\max_{nh\le T}\|\Gamma(t_n)-\Gamma^n\|_{B^0}
\le C_{T,\alpha}h^{\alpha/2}\|\Gamma_0\|_{B^\alpha}\).
Velocity recovery and
\(\|\Gamma_0\|_{B^\alpha}\lesssim\|W\|_{B^{1+\alpha}}\)
prove \eqref{eq:positive-regularity-shear-upper}.

For the lower bound, take \(W=W_\alpha\).
The weighted dyadic coefficients of \(W_\alpha\) in \(B^{1+\alpha}\)
and of \(W_\alpha'\) in \(B^\alpha\) are summable, being bounded by
\(C(1+m)^{-2}\).
The mode of frequency \(k_m=(2^m,0)\) in
\(\Gamma^n\) is multiplied at each step by \((1+ih2^m)^{-1}\).
For \(h=h_m\) and \(M=M_m\), the modulus of the accumulated numerical
multiplier is
\[
 |(1+ih_m2^m)^{-M_m}|
 =(1+T^24^{-m})^{-4^m/2}\longrightarrow e^{-T^2/2}<1.
\]
The exact multiplier \(e^{-iT2^m}\) has modulus one, so the reverse
triangle inequality yields
\[
 \left|\widehat{\Gamma(T)-\Gamma^{M_m}}(k_m)\right|
 \ge \frac{2^{-m\alpha}}{2(1+m)^2}
       \left[1-(1+T^24^{-m})^{-4^m/2}\right]
 \ge c_T\frac{2^{-m\alpha}}{(1+m)^2}.
\]
A Fourier coefficient is bounded by the \(L^\infty\) norm, which is
controlled by the \(B^0\) norm.  Since the velocity means agree and
\(\curl(u(T)-u^{M_m})=\Gamma(T)-\Gamma^{M_m}\), bounded
differentiation gives the same lower bound in \(B^1\).
Now \(2^{-m\alpha}=T^{-\alpha/2}h_m^{\alpha/2}\) and
\(1+m\asymp1+|\log h_m|\), proving
\eqref{eq:positive-regularity-lower-bound}.  Dividing by \(h_m^r\)
gives a quantity tending to infinity whenever \(r>\alpha/2\).
\end{proof}

Together with Theorems~\ref{thm:base-algebraic}
and~\ref{thm:endpoint-rate-family}, the preceding proposition proves
optimality of the algebraic exponent \(\alpha/2\), \(0<\alpha\le2\),
over the stated exact-solution classes allowing \(a=b=0\) and
nonzero velocity means. The fixed-datum lower bound leaves open
logarithmic improvements and does not prescribe the rate of every
individual solution.

For transport by the constant velocity \(m_0\), the resolvent has
multiplier modulus
\((1+h^2(m_0\cdot k)^2)^{-1/2}\). Over \(T/h\) steps it produces
nonvanishing attenuation at \(|m_0\cdot k|\asymp h^{-1/2}\),
which is the scale used in the algebraic lower bound.
Applying the scheme in the Galilean frame moving with \(m_0=(1,0)\)
is exact on the stationary shear. Thus these lower bounds concern
the Eulerian scheme; they do not cover prescribed \(a>0\), zero-mean
velocity classes, fixed finite Fourier resolution, or other algorithms.
Since \(J^n=0\) on these constant-stress examples, the sequential
source correction does not alter the bounds.

\section{Numerical Experiments}
\label{sec:numerics}

We first test temporal convergence for smooth nonlinear data, then
isolate the effect of the new stress on the critical constitutive source.
The final calculation retains infinitely many shear modes and illustrates
the lower bound of Proposition~\ref{prop:endpt-lower-bound}.

\subsection{Temporal convergence for smooth nonlinear data}
\label{sec:sequential-new-test}

On the periodic grid \(\mathcal G_N=\{2\pi\ell/N:
\ell\in\{0,\ldots,N-1\}^2\}\), we use Fourier differentiation,
the projector
\(\widehat{\Pi_N^\#f}(k)=\mathbf1_{\{|k|_\infty<N/3\}}\widehat f(k)\),
and GMRES for the two transport solves.

\begin{remark}[Stability of the projected transport solves]
\label{rem:projected-transport-solves}
Let \(X_N=\operatorname{Ran}\Pi_N^\#\), let
\(v_N\in X_N^2\) be real and divergence free, and set
\(A_Nf=\Pi_N^\#(v_N\cdot\nabla f)\) on \(X_N\).
The strict \(2/3\) cutoff makes the retained coefficients of each
pseudospectral quadratic product of fields in \(X_N\) equal to its Fourier--Galerkin
projection. Orthogonality of \(\Pi_N^\#\) and periodic integration by
parts give, for \(f,z\in X_N\),
\[
 \langle A_Nf,z\rangle_{L^2}
 =-\langle f,A_Nz\rangle_{L^2}.
\]
Consequently, if \((I+hA_N)f=g\), then
\[
 \|f\|_{L^2}^2+h^2\|A_Nf\|_{L^2}^2=\|g\|_{L^2}^2.
\]
Thus \(I+hA_N\) is invertible on the finite-dimensional space
\(X_N\) for every \(h>0\), and its inverse is an \(L^2\) contraction,
uniformly in \(N\). These statements apply componentwise to stress.
For this fixed advector, an approximate solve \(\widetilde f\in X_N\)
with residual \(r=g-(I+hA_N)\widetilde f\) therefore satisfies
\(\|\widetilde f-f\|_{L^2}\le\|r\|_{L^2}\).
This bounds the error of each frozen linear solve. It does not supply
the uniform critical Besov bounds required for the projected analogue
of Section~\ref{ep:sec-smoothing}, and no fully discrete convergence
theorem is inferred from it.
\end{remark}

The initial data are projected from
\[
 \psi_4=\tfrac12\cos^4x\cos^4y,\qquad
 u_0=-\nabla^\perp\psi_4,\qquad
 \tau_0=\tfrac12I+
 \begin{pmatrix}u_{0,1}&u_{0,2}\\u_{0,2}&-u_{0,1}\end{pmatrix}.
\]
We take \(N=32\), \(a=0.1\), \(b=1/2\), \(T=0.02\),
and \(h=T/M\) for \(M=4,8,16,32,64\), with relative GMRES
tolerance \(10^{-13}\). The reference is a sequential trajectory on
the same grid with \(M=2048\).

For the error norms, set
\[
 \chi_{\rm LP}(\xi)=\eta\!\left(\frac{4/3-|\xi|}{4/3-3/4}\right),
 \qquad \eta(t)=\frac{b_+(t)}{b_+(t)+b_+(1-t)},\qquad
 \varphi(\xi)=\chi_{\rm LP}(\xi/2)-\chi_{\rm LP}(\xi),
\]
where \(b_+(t)=e^{-1/t}\) for \(t>0\) and zero otherwise.
We Fourier-pad each error and sample its smooth dyadic blocks on
\((4N)^2\) points, using Euclidean or Frobenius component norms.
The resulting sampled Besov norms are denoted by \(\|\cdot\|_{B^s,\mathrm{samp}}\).
For a trigonometric polynomial supported in \(|k|_\infty<N/3\),
Bernstein's inequality and a nearest grid point give
\[
 (1-\pi/6)\|f\|_{B^s}\le\|f\|_{B^s,\mathrm{samp}}\le\|f\|_{B^s}.
\]

Table~\ref{tab:sequential-new-test} reports the errors at \(T\).
For \(\mathcal E_h=\|e_u\|_{B^1,\mathrm{samp}}+\|e_\tau\|_{B^0,\mathrm{samp}}\),
fitting \(\log\mathcal E_h\) against \(\log h\) over
\(M=8,16,32,64\) gives an observed order of \(1.009\).
The finest value is \(\mathcal E_h=2.472\times10^{-4}\).
\begin{table}[htbp]
\centering\small
\begin{tabular}{rrrr}
\toprule
\(M\) & \(\|e_u\|_{B^1,\mathrm{samp}}\) & \(\|e_\tau\|_{B^0,\mathrm{samp}}\)
& \(\|e_\Gamma\|_{B^0,\mathrm{samp}}\)\\
\midrule
4 & $1.0572\times10^{-3}$ & $2.9493\times10^{-3}$ & $1.4850\times10^{-3}$ \\
8 & $5.3398\times10^{-4}$ & $1.4835\times10^{-3}$ & $7.4208\times10^{-4}$ \\
16 & $2.6756\times10^{-4}$ & $7.4179\times10^{-4}$ & $3.6984\times10^{-4}$ \\
32 & $1.3313\times10^{-4}$ & $3.6870\times10^{-4}$ & $1.8353\times10^{-4}$ \\
64 & $6.5608\times10^{-5}$ & $1.8161\times10^{-4}$ & $9.0324\times10^{-5}$ \\
\bottomrule
\end{tabular}
\caption{Sequential scheme at fixed \(N=32\), with \(h=0.02/M\).
All errors use the sequential reference with \(M=2048\).}
\label{tab:sequential-new-test}
\end{table}

Halving the reference step from \(T/1024\) to \(T/2048\) changes the
finest \(\mathcal E_h\) by \(3.23\%\). Increasing the norm-sampling grid from
\(2N\) to \(4N\) changes \(\mathcal E_h\) by at most \(0.65\%\).
The largest relative GMRES residual is \(9.98\times10^{-14}\).
Residuals of the projected linear systems were evaluated at every step
after filtering and symmetrization. The sum of their normalized
\(H^2\) norms over both solves and all steps is below
\(2.33\times10^{-9}\) for every \(N=32\) trajectory, including the
references; the norm uses the Fourier weights \((1+|k|^2)^2\).
These checks measure sensitivity at fixed spatial resolution and do
not certify error against the continuous PDE solution.

To isolate the source change, we also replace
\(\mathcal RQ(\nabla u^n,\tau^{n+1})\) in \eqref{eq:Gn} by
\(\mathcal RQ(\nabla u^n,\tau^n)\), leaving all other settings unchanged.
This comparison scheme, which evaluates the source at the old stress,
has fitted order \(1.008\) and finest error \(2.777\times10^{-4}\). At \(M=64\), its difference from the
sequential solution in the sampled \(B^1\times B^0\) norm is
\(7.122\times10^{-5}\). A frozen-source reference with \(M=2048\)
changes the finest \(\mathcal E_h\) values by at most \(0.44\%\).

We also repeat the sequential runs on solver grids \(N=64,128\),
with the same settings and a separate \(M=2048\) reference on each grid.
Embedding normalized Fourier coefficients in a common array and sampling
all comparison norms on \(512^2\) points gives a fitted order of
\(1.009\) on all three grids, using \(M=8,16,32,64\).
The reference-field differences in the sampled \(B^1\times B^0\) sum
are \(9.725\times10^{-6}\) between \(N=32\) and \(64\), and
\(2.998\times10^{-12}\) between \(N=64\) and \(128\).

\subsection{The constitutive source at critical regularity}
\label{sec:critical-source-test}

We evaluate the family in Proposition~\ref{prop:critical-source-obstruction}
with \(a=0\), \(b=1/2\), \(h=2^{-12}\), and dyadic \(K\ge16\).
The data and exact first-step formulas are those used in its proof;
in particular, \eqref{eq:source-obstruction-identities} gives
\[
 \begin{aligned}
 \mathcal RQ_{\rm old}&:=\mathcal RQ(\nabla u_K^0,\tau_K^0)
                      =-\cos(Kx_1)g_K,\\
 \mathcal RQ_{\rm new}&:=\mathcal RQ(\nabla u_K^0,\tau_{K,h}^1)
                      =-\cos(Kx_1)E_hg_K.
 \end{aligned}
\]
These formulas also hold at \(K=16,32\); the restriction \(K\ge64\)
in the proposition is used for its stated lower bound.
The proposition proves logarithmic growth of the old-source norm and
the bound \(C(1+\log K)e^{-9hK^2/16}\) for the new-source norm.
The table below samples this family at the fixed step size.

\begin{table}[htbp]
\centering\small
\begin{tabular}{rrrr}
\toprule
\(K\) & \(\|\tau_K^0\|_{B^0,\mathrm{samp}}\)
& \(\|\mathcal RQ_{\rm old}\|_{B^0,\mathrm{samp}}\)
& \(\|\mathcal RQ_{\rm new}\|_{B^0,\mathrm{samp}}\)\\
\midrule
32  & 1.320 & 1.249 & 0.9720 \\
64  & 1.363 & 1.460 & 0.5385 \\
128 & 1.388 & 1.670 & 0.03306 \\
256 & 1.401 & 1.880 & $1.670\times10^{-6}$ \\
\bottomrule
\end{tabular}
\caption{Constitutive sources after one stress update, at fixed
\(h=2^{-12}\). The input velocity satisfies
\(\|\nabla u_K^0\|_{B^0}=1\) for every \(K\).}
\label{tab:critical-source}
\end{table}

Table~\ref{tab:critical-source} is evaluated directly from the Fourier
coefficient formulas, using the same smooth dyadic cutoff as above.
Each block is sampled on \(64N\) points, where \(N=8K\); halving this
sampling grid changes the tabulated norms by less than \(0.001\%\).
At \(K=16\), the exact formulas also agree with one step of the
two-dimensional implementation to absolute error below \(2\times10^{-14}\).
The old-source norms increase with \(K\), while the new-source norms
decrease once \(K\) reaches the heat scale \(h^{-1/2}=64\).
The data vary with \(K\); this table tests the constitutive source on
bounded critical data and does not assign a temporal convergence order.

\subsection{A fixed critical shear with logarithmic convergence}
\label{sec:critical-shear-numerics}

We use the fixed datum in
Proposition~\ref{prop:endpt-lower-bound} with $\gamma=T=1$, $a=b=0$, and
$\tau_0=\operatorname{diag}(0,1)$. Thus
\[
 u_0=(1,W(x_1)),\qquad
 W(x_1)=\sum_{m\ge1}(1+m)^{-2}2^{-m}\sin(2^mx_1).
\]
This fixed datum belongs to $B^1$ but to no $B^{1+\alpha}$ with
$\alpha>0$. The sequential source correction vanishes identically on this
family, so the sequential scheme follows the same implicit-transport recurrence.
We evaluate the spatially untruncated recurrence in closed form at
$h_j=2^{-2j}$ and $M_j=2^{2j}$, including the infinite frequency tail.
Set $a_m=(1+m)^{-2}$ and
\[
 q_m(h)=(1+ih2^m)^{-1/h}e^{i2^m},\qquad
 L(h)=\sum_{m\ge1}a_m(1-\operatorname{Re}q_m(h)),\qquad
 A(h)=\sum_{m\ge1}a_m|1-q_m(h)|.
\]
Here $L(h)$ is the good-vorticity error at $x_1=T=1$, and
\[
 L(h)\le\|\Gamma(1)-\Gamma^M\|_{L^\infty}\le A(h),\qquad
 \|\Gamma(1)-\Gamma^M\|_{B^0}+\|u(1)-u^M\|_{B^1}\lesssim A(h).
\]
The majorant $A$ is not an evaluated Besov norm. For $m>j+32$, the
coefficient tail is summed as $\zeta(2,j+34)$, with $q_m$ replaced by zero.
Monotonicity of $|q_m|$ bounds the omitted contribution by
$|q_{j+33}|\zeta(2,j+34)$. Including individually omitted tiny multipliers,
the analytical omission bound is below $1.5\times10^{-109}$ for the
reported evaluations; this bound does not include floating-point rounding.
Calculations at 80 and 120 decimal digits agree to absolute error below
$10^{-80}$. We compare $h_j$ with $h_{2j}=h_j^2$ using the effective order
\[
 p_j=\frac{\log L(h_j)-\log L(h_{2j})}{2j\log2}.
\]

\begin{figure}[!htb]
\centering
\includegraphics[width=0.82\textwidth]{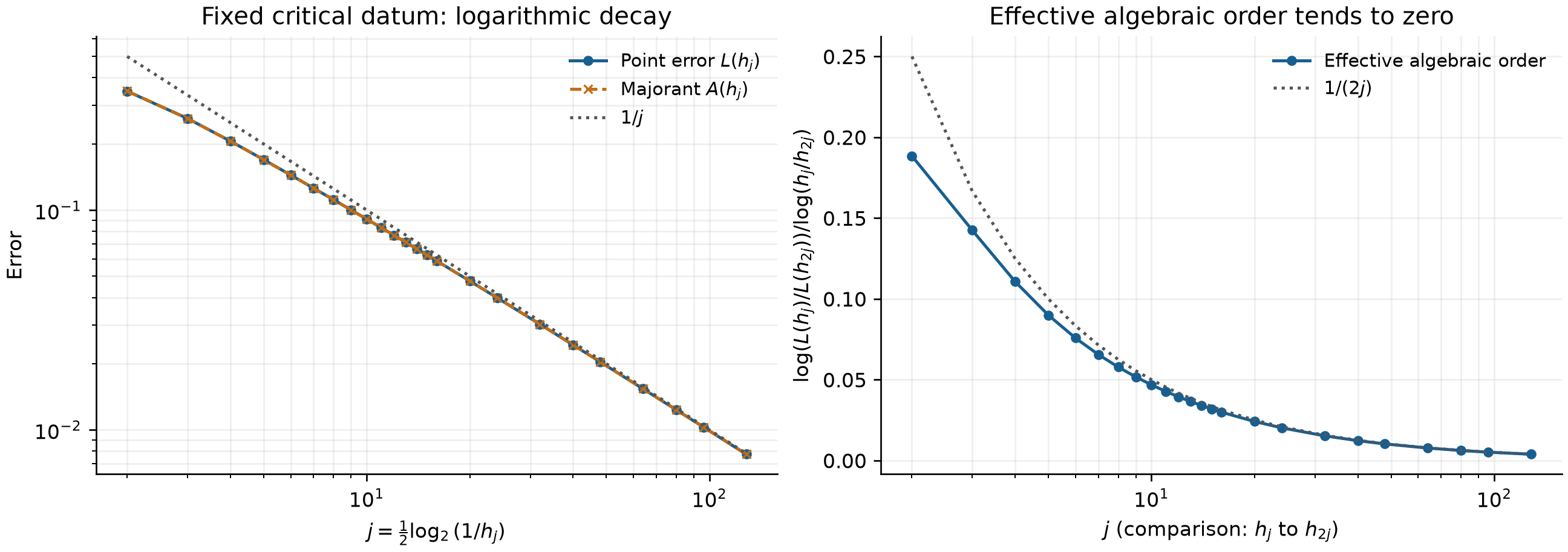}
\caption{Fixed critical shear: the point error \(L(h_j)\) and coefficient
majorant \(A(h_j)\) follow the logarithmic scale \(1/j\) (left), while
the effective algebraic order \(p_j\) decreases toward zero (right).}
\label{fig:critical-shear-log}
\end{figure}

For \(j=4,8,16,32,64,128\), \(jL(h_j)\) increases from
\(0.824\) to \(0.992\), while \(p_j\) decreases from \(0.111\) to
\(0.00388\), consistently with Remark~\ref{rem:critical-shear-matching}.
Figure~\ref{fig:critical-shear-log} shows the logarithmic decay and
vanishing effective order.
The logarithmic rate is specific to this family. A fixed eight-mode truncation gives
an effective order of \(0.999999998\) between \(j=20\) and \(j=24\):
a fixed spatial cutoff eventually hides the critical behavior.

\section{Conclusion}
\label{sec:conclusion}

The sequential good-vorticity scheme converges in the critical Besov
norms on every prescribed interval satisfying the exact-solution
assumptions. When the velocity is controlled only in \(B^1\), the gain
of one stress derivative at the time nodes is optimal, even for smooth
initial stress. Full two-derivative bounds hold for \(E_h\tau^n\)
after summation in time with weight \(h\).
Together, Theorems~\ref{thm:base-algebraic}
and~\ref{thm:endpoint-rate-family} give the rate \(h^{\alpha/2}\)
for \(0<\alpha\le2\), reaching first order at \(\alpha=2\).
Fixed advected shears establish optimality of these algebraic exponents
over the class allowing \(a=b=0\) and nonzero velocity means, and show
that no positive algebraic rate can hold for all critical data.
The constitutive-source example separately proves the failure of a
uniform old-stress source bound on bounded critical data.

For the smooth nonlinear test, the observed temporal order is close to
one on all three spatial grids. The source test illustrates how smoothing
the high-frequency stress before multiplication controls the constitutive
source. The infinite-mode shear exhibits the logarithmic decay established
for this fixed family.

For a fully discrete extension, one remaining question is whether the
projected transport resolvent admits suitable critical Besov bounds
uniform in the spatial cutoff. The \(L^2\) contraction in
Remark~\ref{rem:projected-transport-solves} controls each frozen linear
solve; the critical convergence proof uses the unprojected equation.

\FloatBarrier

\appendix

\section{Sharpness of discrete smoothing}
\label{sec:smoothing-sharpness}

This appendix establishes two distinct sharpness statements. For the
abstract heat recurrence with arbitrary forcing, the optimal frequency
weight grows like \(2^{2j}\) until it reaches \(h^{-1}\), as expressed
by the saturated norm below. A separate example for the nonlinear scheme
shows that, with the velocity controlled only in \(B^1\), the stress
at a new time node need not belong to \(B^{1+\varepsilon}\) for any
\(\varepsilon>0\), even when the initial stress is smooth.

In the recurrence of Lemma~\ref{ep:heat-maximal}, the forcing
\(h f_n\) is added after the heat operator acts. It therefore enters
the new node without being smoothed. Define
\begin{equation}
 \|f\|_{B_h^{r+2}}
   =\|\Delta_{-1}f\|_\infty
     +\sum_{j\ge0}2^{rj}\min(2^{2j},h^{-1})
                                      \|\Delta_jf\|_\infty .
 \label{ep:saturated-norm}
\end{equation}
Under the hypotheses of that lemma, the nodal values satisfy
\begin{equation}
 h\sum_{n=0}^{J}\|z^n\|_{B_h^{r+2}}
 \le C_{T,a,r}\left(\|z^0\|_{B^r}
                 +h\sum_{n=0}^{J-1}\|f_n\|_{B^r}\right).
 \label{eq:saturated-heat-bound}
\end{equation}
Indeed, apply the discrete Duhamel formula blockwise. The identity
\(h\min(2^{2j},h^{-1})=\min(h2^{2j},1)\) and the first geometric
sum in \eqref{ep:heat-geometric-sums} bound the initial and forcing
contributions, including the unsmoothed term. Heat contraction
controls the lowest block, including constants.

For the stress iterates under \eqref{ep:prefix}, applying
\eqref{eq:saturated-heat-bound} with \(r=0\) and \(f_n=F_n\)
from \eqref{ep:heat-forcing-definition} gives
\[
 h\sum_{n=0}^{J}\|\tau^n\|_{B_h^2}\le C_{T,R,\delta}.
\]
Before the transport solve, the stress satisfies a full two-derivative
bound after summation in time:
\begin{equation}
 h\sum_{n=0}^{J-1}\|E_h(\tau^n+hS^n)\|_{B^2}
       \le C_{T,R,\delta}.
 \label{ep:pretransport-smoothing}
\end{equation}
To see this, use \eqref{ep:stress-variation-bound} for \(E_h\tau^n\).
For \(n\ge1\), the mixed product law and
\eqref{ep:stress-integrated} give
\(h\sum_{n=1}^{J-1}\|S^n\|_{B^0}\le C_{T,R,\delta}\).
Heat smoothing then bounds
\(h^2\sum_{n=1}^{J-1}\|E_hS^n\|_{B^2}\) uniformly.
For \(n=0\), use
\(\|E_hS^0\|_{B^2}\le Ch^{-1}\|S^0\|_\infty\).
This proves \eqref{ep:pretransport-smoothing}.

\paragraph{Optimal weight for the abstract heat recurrence.}
The saturated weight in \eqref{eq:saturated-heat-bound} is optimal
for an arbitrary \(B^r\) forcing. At a resolved frequency,
\(h2^{2j}\le1\le T2^{2j}\), the homogeneous heat evolution has an
\(h\)-weighted discrete time sum comparable to \(2^{-2j}\) times its
initial amplitude; hence the weight cannot exceed \(C2^{2j}\).
At an unresolved frequency, where \(h2^{2j}>1\), take \(z^0=0\) and
apply forcing to a single Fourier mode only at \(n=0\). Since \(z^1=hf_0\), any admissible weight
\(\mu_{j,h}\) must satisfy
\[
 h^2\mu_{j,h}\|f_0\|_{B^r}
 \le Ch\|f_0\|_{B^r},
\]
so \(\mu_{j,h}\le Ch^{-1}\).

\paragraph{Sharp nodal regularity for the nonlinear scheme.}
For the nonlinear split scheme, the final transport solve prevents
full nodal \(B^2\) smoothing: the stress may fail to belong to
\(B^{1+\varepsilon}\) for every \(\varepsilon>0\), even when its
initial value is smooth. Take \(a=b=0\) and
\[
 u^0=(0,W(x_1)),\qquad \tau^0=\tfrac12\cos x_2\,I,
 \qquad
 W(x_1)=\sum_{j\ge1}j^{-2}2^{-j}\cos(2^jx_1).
\]
Then \(W\in B^1\), while \(W\notin B^{1+\varepsilon}\) for every
\(\varepsilon>0\). Since \(b=0\) and the initial stress is isotropic,
\(Q\) vanishes in the first stress source. Also, \(D(u^0)\) has zero
trace.
Consequently \(q^1=\operatorname{tr}\tau^1\) satisfies
\[
 q^1+hW(x_1)\partial_2q^1=e^{-h}\cos x_2,\qquad
 q^1=e^{-h}\left(
 \frac{\cos x_2}{1+h^2W(x_1)^2}
 +\frac{hW(x_1)\sin x_2}{1+h^2W(x_1)^2}\right).
\]
If \(q^1\in B^{1+\varepsilon}(\mathbb T^2)\), projection onto its
cosine and sine modes in \(x_2\) gives
\[
 \frac1{1+h^2W^2},\quad
 \frac{hW}{1+h^2W^2}\in B^{1+\varepsilon}(\mathbb T).
\]
The first coefficient is bounded away from zero. Taking the quotient,
the composition and product laws give
\(W\in B^{1+\varepsilon}\), a contradiction. Thus the one-derivative
nodal gain is sharp. Uniform bounds after summation in time still give
two derivatives for \(E_h\tau^n\), by
Corollary~\ref{ep:stress-variation}, and for the stress before the
transport solve, by \eqref{ep:pretransport-smoothing}.

\section*{Statements and Declarations}

\noindent\textbf{Funding.}
X. Cheng is supported in part by the NSFC (Grant Nos. 12401270 and 42450192), the Natural Science Foundation of Shanghai (Grant No. 24ZR1404200), and the Shanghai Magnolia Talent Plan Pujiang Project (Grant No. 24PJA007). C. Yuan is supported by the NSFC (Grant No. 123B2008).

\smallskip\noindent\textbf{Competing interests.}
The authors declare no competing interests.

\smallskip\noindent\textbf{Data availability.}
All numerical experiments use the initial data, parameters, and
evaluation procedures specified in Section~\ref{sec:numerics}.
No external datasets were used.

\smallskip\noindent\textbf{Use of artificial intelligence.}
OpenAI Codex assisted with drafting, restructuring and review of the
mathematical arguments, and with numerical-code development and checking.
The authors are responsible for the content of the manuscript.

\end{document}